\documentclass[11pt]{article}
\usepackage[french, english]{babel}
\usepackage[utf8]{inputenc}

\usepackage{amsmath,amsfonts,amssymb,amsthm,mathrsfs}
\usepackage{minitoc}
\usepackage{bbm}

\renewcommand{\leq}{\leqslant}
\renewcommand{\geq}{\geqslant}
\usepackage[mono=false]{libertine}
\useosf 
\usepackage{mathpazo}

\usepackage{wasysym}

\newcommand{\w}{\Circle}

\usepackage[dvipsnames]{xcolor}
\usepackage[a4paper,vmargin={2.5cm,2.5cm},hmargin={2cm,2cm}]{geometry}
\usepackage[font=sf, labelfont={sf,bf}, margin=1cm]{caption}
\usepackage{graphicx,graphics}
\usepackage{subcaption}
\usepackage{epsfig}
\usepackage{latexsym}
\usepackage{stmaryrd}
\usepackage{ae,aecompl}

 \usepackage[colorlinks=true]{hyperref}
\usepackage{pstricks}
\usepackage{enumerate}
\usepackage{tikz}
\usepackage{fontawesome}
\usetikzlibrary{arrows,automata}
\numberwithin{equation}{section}

\newtheorem{theorem}{Theorem}[section]
\newtheorem{corollary}[theorem]{Corollary}
\newtheorem{lemma}[theorem]{Lemma}
\newtheorem{proposition}[theorem]{Proposition}
\theoremstyle{definition}

\newcommand{\D}{\mathbb{D}}		
\newcommand{\E}{\mathbb{E}}		

\newcommand{\N}{\mathbb{N}}	

\renewcommand{\P}{\mathbb{P}}	
	
\newcommand{\R}{\mathbb{R}}		

\DeclareMathAlphabet{\mathpzc}{OT1}{pzc}{m}{it}
\newcommand{\cA}{\mathcal A}		

\newcommand{\cE}{\mathcal{E}}	
\newcommand{\cK}{\mathcal{K}}		
\newcommand{\cP}{\mathcal{P}}		
\newcommand{\cR}{\mathcal{R}}	
\newcommand{\cS}{\mathcal{S}}
\newcommand{\cT}{\mathcal{T}}		
\newcommand{\cV}{\mathcal{V}}

\newcommand{\1}{\mathbf{1}}			

\title{Scaling limits for peeling  explorations of the Brownian half-plane}
\usepackage{calrsfs}

\author{ Mathieu Mourichoux\thanks{Sorbonne Université, LPSM, 4 Place Jussieu, 75005 Paris, France. \hfill  \href{mailto:}{\texttt{mourichoux@lpsm.paris}}} \,\,\, and \,\, Armand Riera\thanks{Sorbonne Université, LPSM, 4 Place Jussieu, 75005 Paris, France. \hfill  \href{mailto:}{\texttt{riera@lpsm.paris}}
}  
}

\date{}

\usepackage{calrsfs}
\usepackage{enumitem}
\usepackage{enumerate}
\def\w{\mathrm{w}}

\def\N{\mathbb{N}}

\def\D{\mathbb{D}}
\def\P{\mathbb{P}}

\def\E{\mathbb{E}}

\def\R{\mathbb{R}}
\def\z{\mathcal{Z}}

\def\m{\mathcal{M}}

\def\ve{{\varepsilon}}

\def\dd{\mathrm{d}}
\def\wh{\widehat}

\def\tr{\mathrm{tr}}

\newcommand{\bigA}{\mathord{\scalebox{1.2}{$\mathpzc{A}$}}}
\newcommand{\bigB}{\mathord{\scalebox{1.2}{$\mathpzc{B}$}}}
\newcommand{\bigC}{\mathord{\scalebox{1.2}{$\mathpzc{C}$}}}
\begin{document}

\maketitle

\begin{abstract}
   We construct and study peeling explorations of the Brownian half-plane by successively revealing  hulls of fixed radius. We prove that for any peeling process,  as the radius of the hulls tends to zero, the sequence of discovered hulls converges to an explicit family of Brownian disks. We also consider the particular case of what we call \emph{nearest-neighbour peeling explorations}, where the center of each new hull is chosen on the outer boundary of the previously discovered hull, naturally giving rise to an exploration curve. We provide scaling limits results for these curves, jointly with the convergence of the associated boundary-length and volume processes.
\end{abstract}
\setcounter{tocdepth}{2}
\tableofcontents

\section{Introduction}

This work is concerned with models of random geometry, called Brownian surfaces, that appear as scaling limits of random planar maps. More precisely, we are interested in the study of the Brownian half-plane $(\mathfrak{H},D_\infty)$, which appears as the scaling limit of the uniform infinite quadrangulation of the half-plane (abbreviated as UIHPQ) \cite{BaurMiermontRay,GM2017}. This space is naturally equipped with a volume measure $\mathrm{Vol}_\infty$, as well as a parametrization $\Lambda_\infty$ of its boundary.

In order to study this random metric space, one relies on explicit constructions based on random labeled trees, which are the continuous counterparts of discrete bijections. There exist several such equivalent constructions \cite{CC20218,LGR3,LGR6}, and each of them is well suited to studying specific geometric quantities of the Brownian half-plane. However, in the discrete setting, there is another path to studying the UIHPQ, which is called the peeling process. Informally, a peeling process is a Markovian exploration of the UIHPQ where the faces are discovered one by one according to an algorithm depending only on the discovered region. These processes were first introduced by Watabiki~\cite{Watabiki1995} in the physics literature, and were first used rigorously by Angel~\cite{Angel2003}; we refer to \cite{CurienPeeling} for a detailed introduction to this topic. 

The study of the peeling processes of random planar maps, and in particular of their scaling limits, has been very fruitful and has drawn a lot of attention in the last decade. For instance, Curien and Le Gall proved in \cite{CLG2017} scaling limit results for general peeling processes. They also obtained several results for peeling processes based on different metrics on the UIPQ or UIPT, which turn out to be related to the metric exploration of Brownian surfaces \cite{CurienLeGallHull}. Then, Bertoin, Curien and Kortchemski \cite{BertoinCurienKortchemski} studied a branching version of the metric peeling of random Boltzmann planar maps, and proved that some statistics converge toward a fully characterised growth-fragmentation process. See also \cite{LGR1} for the analogous result in the continuum. 

On the other hand, Gwynne and Miller~\cite{GM2021} studied, through a peeling process, the scaling limit of a critical percolation interface in large Boltzmann quadrangulations and in the UIHPQ. In particular, they obtained a tightness result for the percolation interface and proved that the connected components of the complement of any limiting curve must be Brownian disks. Using the equivalence between Brownian geometry \cite{MillerSheffieldI,MillerSheffieldII,MillerSheffieldIII} and $\sqrt{8/3}$-LQG surfaces together with a characterization of SLE$_\kappa$ on LQG surfaces \cite{GwynneMillerCharacterization}, they were able to prove that the percolation interface actually converges towards an SLE$_6$ on an independent $\sqrt{8/3}$-LQG surface, parametrized by quantum natural time.

The purpose of this article is to introduce and study the continuous counterpart of the peeling processes for the Brownian half-plane. The rigorous definition of these peeling explorations will be given in Section~\ref{sec:peeling}; let us first give an informal description of them. In the discrete world, the analysis of the peeling processes of the UIHPQ relies heavily on the spatial Markov property of this space, which informally states that after revealing a face incident to the boundary, the unique infinite component of the complement is still distributed as the UIHPQ and is independent of the discovered region. In order to keep a similar picture in the continuum, we define our peeling explorations as revealing consecutive \emph{hulls}, where for every $x\in\mathfrak{H}$, the hull of radius $r>0$ around $x$, that we denote by $B_r^\bullet(x,\mathfrak{H})$, is the usual ball of radius $r$ around $x$, together with the bounded connected components of its complement. Indeed, it was recently proved in~\cite{LGR6} that such sets satisfy a spatial Markov property in the Brownian half-plane, as we will explain below (see also~\cite{LG2019,LGR3,LGR4} for other results on Markov properties in Brownian surfaces).
Fix $r>0$. The peeling exploration evolves as follows:
\begin{itemize}
  \item First, we reveal the hull of radius $r$ around $\Lambda_\infty(0)$.
  \item Then, by the spatial Markov property of the Brownian half-plane, the space $\mathfrak{H}\setminus (B_r^{\bullet}(\Lambda_\infty(0),\mathfrak{H})\setminus\partial B_r^{\bullet}(\Lambda_\infty(0),\mathfrak{H}))$ equipped with the unique extension of the intrinsic distance on its interior, that we denote by $\mathfrak{H}^{(1)}(r)$, is still distributed as a Brownian half-plane, and is independent of $B_r^{\bullet}(\Lambda_\infty(0),\mathfrak{H})$.
  \item We can then iterate by choosing a new point $\mathbf{x}_1(r)$ on the boundary of $\mathfrak{H}^{(1)}(r)$ \emph{in a way that does not depend on this space}, and repeating the operation.
\end{itemize}
\begin{figure}[ht]
  \centering
    \vspace{-0.4cm}
  \begin{subfigure}[b]{0.45\linewidth}
    \centering
    \includegraphics[width=\linewidth]{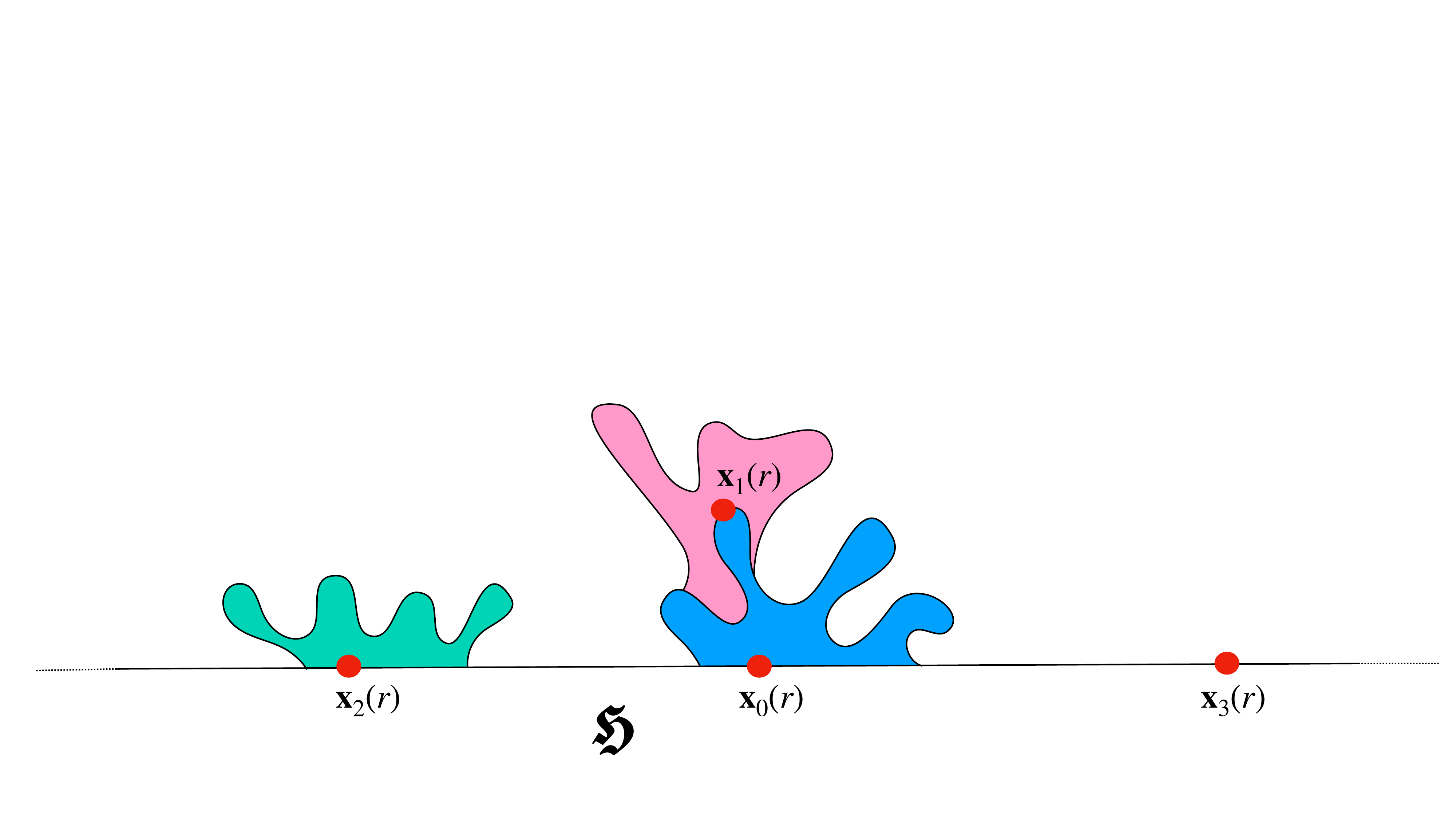}
  \end{subfigure}
  \hfill
  \begin{subfigure}[b]{0.45\linewidth}
    \centering
    \includegraphics[width=\linewidth]{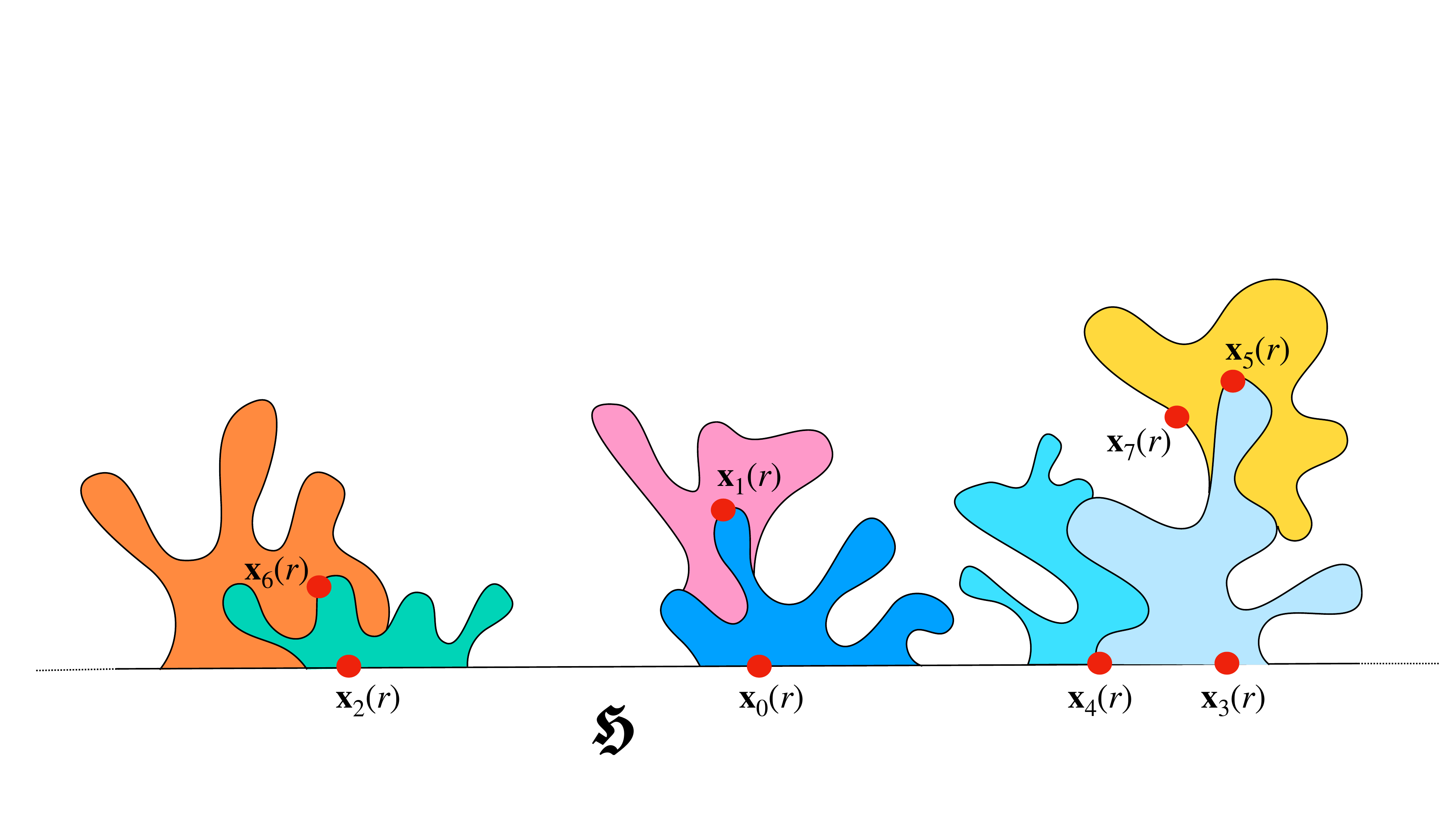}
  \end{subfigure}

  \vspace{-0.4cm}

  \begin{subfigure}[b]{0.45\linewidth}
    \centering
    \includegraphics[width=\linewidth]{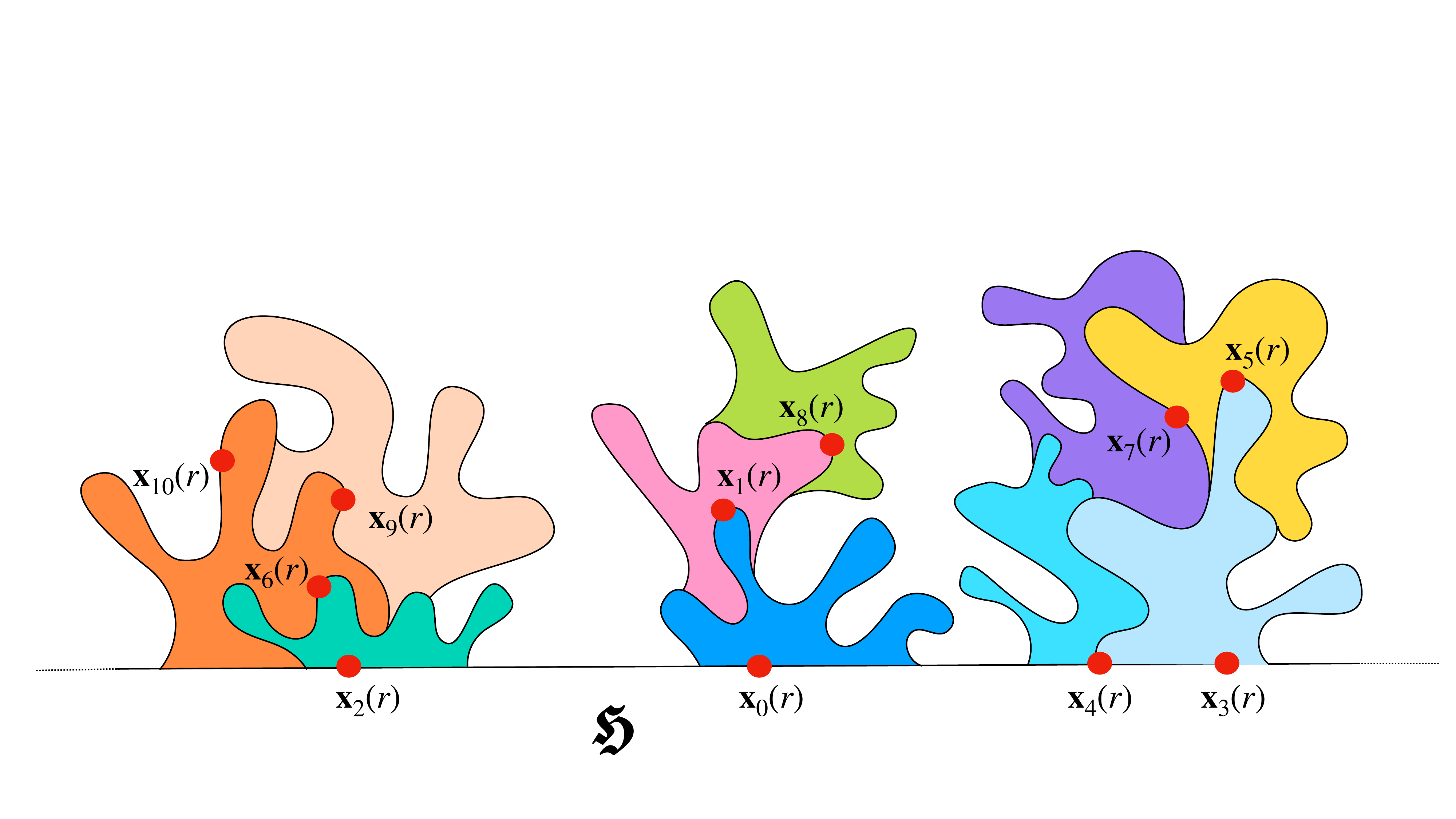}
  \end{subfigure}
  \hfill
  \begin{subfigure}[b]{0.45\linewidth}
    \centering
    \includegraphics[width=\linewidth]{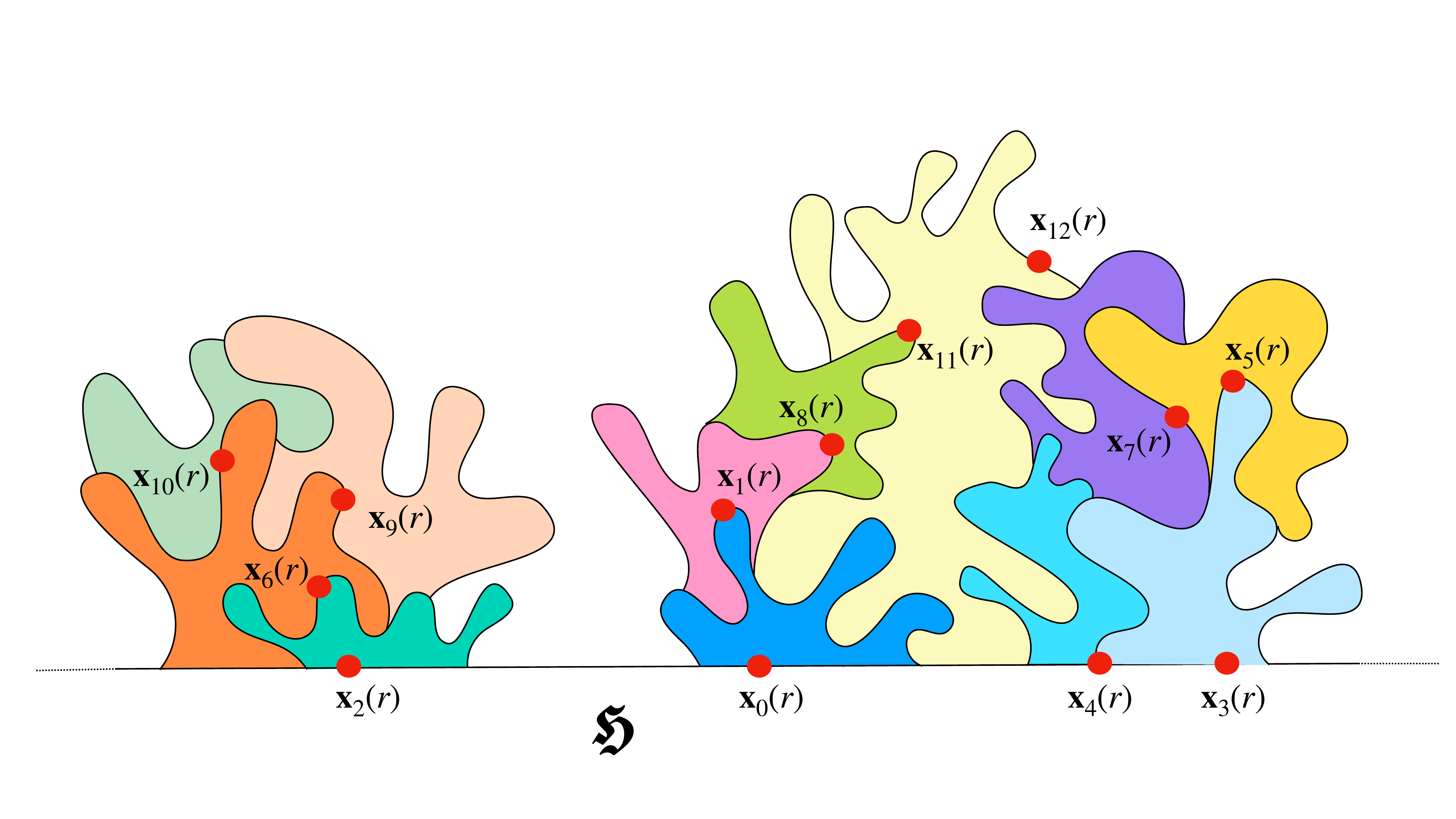}
  \end{subfigure}
  \caption{Illustration of a peeling exploration of the Brownian half-plane after three, seven, ten, and twelve steps, respectively.}
  \label{fig:peeling}
\end{figure}
Repeating this procedure gives rise to a sequence of points $(\mathbf{x}_n(r))_{n\geq0}$ and to a sequence of metric spaces $\big(\mathfrak{H}^{(n)}(r)\big)_{n\geq 0}$, where for every $n\geq 1$, $\mathbf{x}_n(r)$ lies on the boundary of a Brownian half-plane $\mathfrak{H}^{(n)}(r)$. We emphasize that even though the choices of the points cannot depend on the future of the exploration, they can depend on its past. There are several natural ways to choose the new points. For instance, the new point can be chosen uniformly at random on the boundary of the explored region, or can be the leftmost point on the boundary of the explored region. As one might guess, the freedom in the choice of the successive points may result in very different explorations of the space.

Once the peeling explorations are properly introduced, our goal is to obtain scaling limit statements for a fixed peeling process \emph{when the radius of the discovered hulls tends to zero}.
Our results in this direction are twofold. First, we prove general scaling limit statements for the sequence of discovered hulls and their boundary lengths and volumes, which hold \emph{for any} peeling exploration. Then, we focus on a particular subclass of peeling explorations, which we call \emph{nearest-neighbour peeling explorations}, and prove convergence and tightness statements about statistics of these peeling explorations. Let us explain these two points and give an informal presentation of our main results.

First, fix any peeling exploration, and let $r>0$ be the radius of the consecutive hulls. The first hull comes with a total boundary length $L_1(r)$, which can be split into four parts: the left and right top boundary lengths and the left and right bottom boundary lengths, that we denote respectively by $Y_1^{(\ell)}(r)$, $Y_1^{(\mathrm{r})}(r)$, $X_1^{(\ell)}(r)$, $X_1^{(\mathrm{r})}(r)$, as shown in Figure \ref{fig:hull}. 

\begin{figure}[htbp!]
   \centering
   \includegraphics[width=9cm]{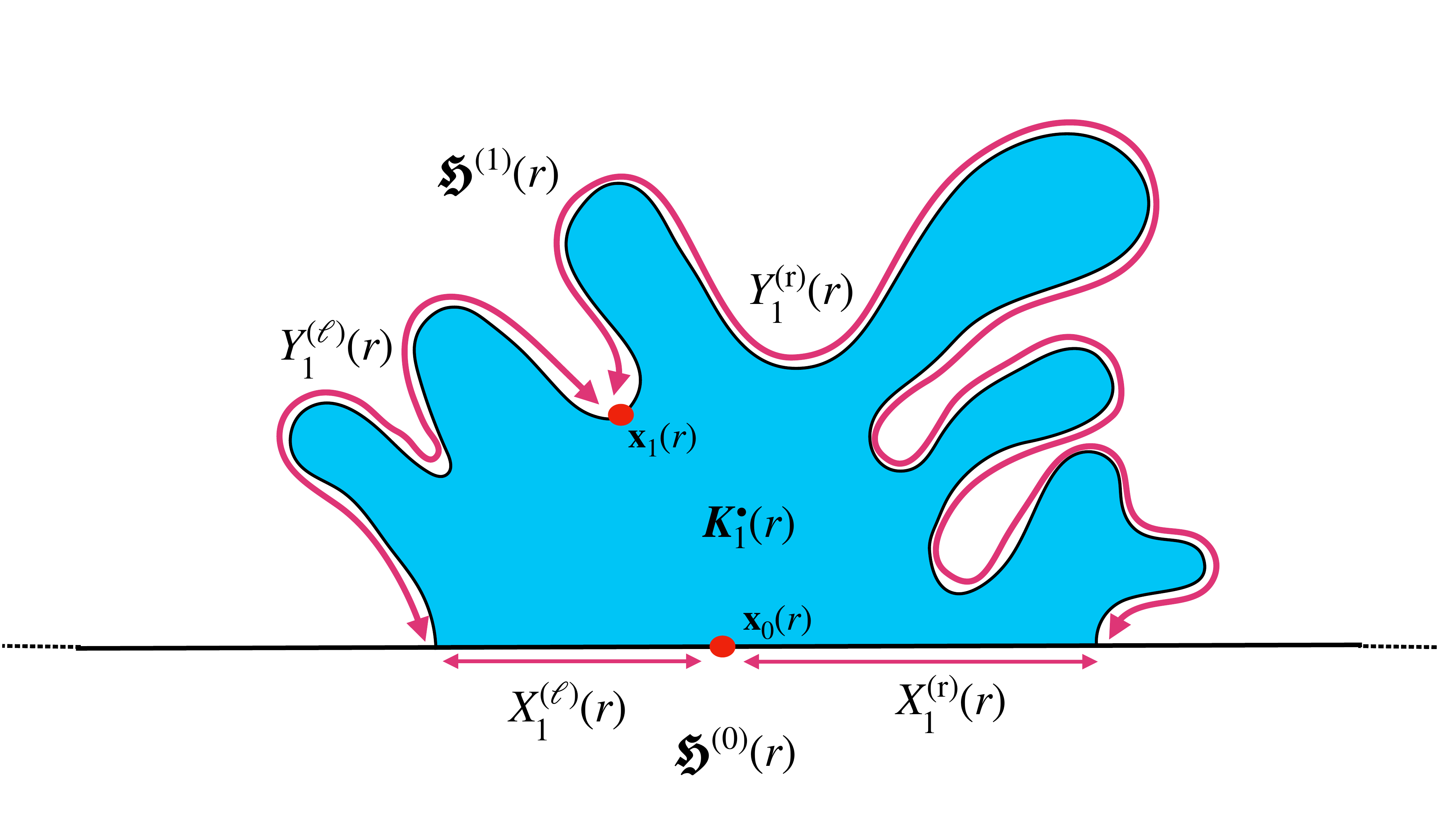}
   \caption{Illustration of the variables $X_1^{(\ell)}(r),X_1^{(\mathrm{r})}(r),Y_1^{(\ell)}(r)$ and $Y_1^{(\mathrm{r})}(r)$.}
   \label{fig:hull}
\end{figure}

Our first result, which may be of independent interest, is about the convergence of a conditioned hull of radius $r$ as $r\to 0$. Indeed, using the explicit construction of a hull given in~\cite{LGR6}, it is possible to construct a hull of radius $r$ conditioned to have prescribed boundary lengths, see Section \ref{sec:hull}. This hull is equipped with the unique extension of the intrinsic distance on its interior to its boundary\footnote{The uniqueness of this extension is not immediate, see \cite[Proposition 5]{LGR4}}, with its natural volume measure and decorated with a boundary curve. We will be interested in the regime where one of the bottom boundary lengths is conditioned to converge towards $z>0$ as $r\to0$, while all the other boundary lengths tend to zero at the right speed, see Section~\ref{sec:discovered:regions} for more details. In this setup, we prove in Theorem \ref{theo:convergence:hull:} that as $r$ tends to zero, the sequence of conditioned hulls converges in distribution towards a Brownian disk with perimeter $z$ for the Gromov--Hausdorff--Prokhorov--uniform topology. 

Then, for every $n\geq 0$, we let $\big(Y_n^{(\ell)}(r),Y_n^{(\mathrm{r})}(r),X_n^{(\ell)}(r),X_n^{(\mathrm{r})}(r)\big)$ denote the different boundary lengths of $B_r^{\bullet}(\mathbf{x}_n,\mathfrak{H}^{(n)}(r))$. As we will see, it is relevant to keep track of the successive bottom boundary lengths as follows. First, for every $n\geq 0$, set
$$\widetilde{H}_n^{(\ell)}(r)=\sum_{k=1}^{n}\big(\frac{r^2}{3}-X_k^{(\ell)}(r)\big)\quad,\quad \widetilde{H}_n^{(\mathrm{r})}(r)=\sum_{k=1}^{n}\big(\frac{r^2}{3}-X_k^{(\mathrm{r})}(r)\big)
$$
and
$$V_n(r)=\sum_{k=0}^{n}\mathrm{Vol}_\infty\big(B_r^{\bullet}(\mathbf{x}_k,\mathfrak{H}^{(k)}(r))\big).$$
Then, for every $n\geq 1$, let $\mathbf{K}_n^{\bullet}(r)$ stand for the space $B_r^{\bullet}(\mathbf{x}_{n-1},\mathfrak{H}^{(n-1)}(r))$, which we view as a curve-decorated metric measure space as described above. We consider the sequence of hulls $\big(\mathbf{K}_n^\bullet(r)\big)_{n\geq 1}$ ordered by decreasing values of $e^{-r^3n}L_n(r)$, that we denote by
\begin{equation*}
    \mathcal{K}^\bullet(r)=\big(r^3 j(n),L_{j(n)}(r),\mathbf{K}_{j(n)}^\bullet(r)\big)_{n\geq 1},
\end{equation*}
where $(j(n),L_{j(n)}(r), K_{j(n)}^\bullet(r))_{n\geq 1}$ is the sequence $(n, L_n(r) , K_n^\bullet(r))_{n\geq 1}$ reordered in decreasing order of $\exp(-r^3n)L_n(r)$. 

One of our main results is to identify the scaling limit of these processes. To state it, consider two independent totally asymmetric $3/2$-stable processes $\big(\cS_t^{(\ell)},\cS_t^{(\mathrm{r})}\big)_{t\geq 0}$ with no upward jumps and with Lévy measure
\begin{equation*}
\frac{1}{3\sqrt{6\pi}}\,|x|^{-5/2}\mathbf{1}_{\{x<0\}}\,\mathrm{d}x,
\end{equation*}
started from $0$. Next, we consider $(t_i^{(\ell)})_{i\geq 1}$ and $(t_i^{(\mathrm{r})})_{i\geq 1}$, the jumping times of $\cS^{(\ell)}$ and $\cS^{(\mathrm{r})}$, and $\xi_i^{(\ell)}$ and $\xi_i^{(\mathrm{r})}$, $i\geq 1$, i.i.d. random variables (and independent of $\big(\cS_t^{(\ell)},\cS_t^{(\mathrm{r})}\big)$) with density
\begin{equation*}
\frac{\mathbf{1}_{u>0}}{\sqrt{2\pi u^5}} \cdot \exp\big(-\frac{1}{2u}\big).   
\end{equation*}
Finally, we set
\begin{equation*}
    \cV^{(\ell)}_t=\sum_{t_i^{(\ell)}\leq t} (\Delta \cS_{t_i}^{(\ell)})^2\cdot \xi_i^{(\ell)}\quad \text{ ; }\quad \cV^{(\mathrm{r})}_t=\sum_{t_i^{(\mathrm{r})}\leq t} (\Delta \cS_{t_i}^{(\mathrm{r})})^2\cdot \xi_i^{(\mathrm{r})} \quad \text{ ; }\quad \cV_t=\cV^{(\ell)}_t+ \cV^{(\mathrm{r})}_t.
\end{equation*}
In Theorem~\ref{th:conv:familledisque}, we prove that, for any peeling exploration, the following convergence holds:
$$\left(\left(\widetilde{H}_{\lfloor r^{-3}t\rfloor}^{(\ell)}(r),\widetilde{H}_{\lfloor r^{-3}t\rfloor}^{(\mathrm{r})}(r),V_{\lfloor r^{-3}t\rfloor}(r)\right)_{t\geq 0},\mathcal{K}^\bullet(r)\right)\xrightarrow[r\rightarrow0]{(d)}\left(\left(\cS_t^{(\ell)},\cS_t^{(\mathrm{r})},\cV_t\right)_{t\geq0},\mathcal{K}^\bullet\right),$$
where $(\cS^{(\ell)},\cS^{(\mathrm{r})},\cV)$ is defined as above, and $\cK^\bullet$ is a sequence of Brownian disks $(\mathbb{D}_i)_{i\geq1}$ whose boundary lengths $(\ell_i)_{i\geq1}$ are given by the sizes of the jumps of $\cS^{(\ell)}$ and $\cS^{(\mathrm{r})}$, and which are independent given their boundary sizes. We also associate the corresponding jumping times $(t_i)_{i\geq1}$ to these disks, and rearrange the sequence in decreasing order of $(e^{-t_i}\ell_i)_{i\geq 1}$. Moreover, almost surely, for every $t\geq0$, it holds that
    $$\cV_t=\sum_{i:t_i\leq t}\mathrm{Vol}(\D_i).$$
In light of this result, we see that any ``non-trivial" scaling limit of a peeling exploration is related to Brownian disks and $3/2$-stable processes. This is also reminiscent of results obtained by studying the peeling process of random planar maps, see \cite[Theorem 1]{CLG2017}.

Our second result deals with a particular class of peeling explorations that we call \emph{nearest-neighbour peeling explorations}. For these peeling explorations, the next point to choose must lie on the outer boundary of the previously discovered hull, see Figure \ref{fig:near}.
\par
More precisely, for every $n\geq 0$, observe that $Y_n^{(\ell)}(r)+Y_n^{(\mathrm{r})}(r)$ is the boundary length of the top part of the hull, and let $(U_n(r))_{n\geq 1}$ be a sequence of i.i.d.\ random variables with values in $[0,1]$, which are also independent of all the other random variables considered here (but whose common law may depend on $r$). We will restrict ourselves to the case where $\mathbf{x}_{n+1}$ is chosen on the top boundary of $B_r^{\bullet}(\mathbf{x}_n,\mathfrak{H}^{(n)}(r))$ at boundary length $(Y_n^{(\ell)}(r)+Y_n^{(\mathrm{r})}(r))\cdot U_n(r)$, and will also assume that there exist $d\in\R$ and $\vartheta\geq 0$ 
\begin{equation}
   \E[U_1(r)]=1/2+d\cdot  r^\vartheta
\end{equation}
for $r$ small enough, with the convention that $\vartheta=1$ if $d=0$.
Observe that if $\vartheta=0$ and since the support of $U_1(r)$ is included in $[0,1]$, we necessarily have $d\in[-1/2,1/2]$. As we will see, the behaviour of the peeling exploration heavily depends on the value of $\vartheta$, with a phase transition occurring at $\vartheta=1$. This class of peeling explorations naturally gives rise to an exploration curve $\Gamma^{(r)}$ which connects the points $(\mathbf{x}_n)_{n\in\N}$ by piecewise geodesic segments, as illustrated in Figure~\ref{fig:near}.

\begin{figure}[htbp!]
   \centering
   \includegraphics[width=17cm]{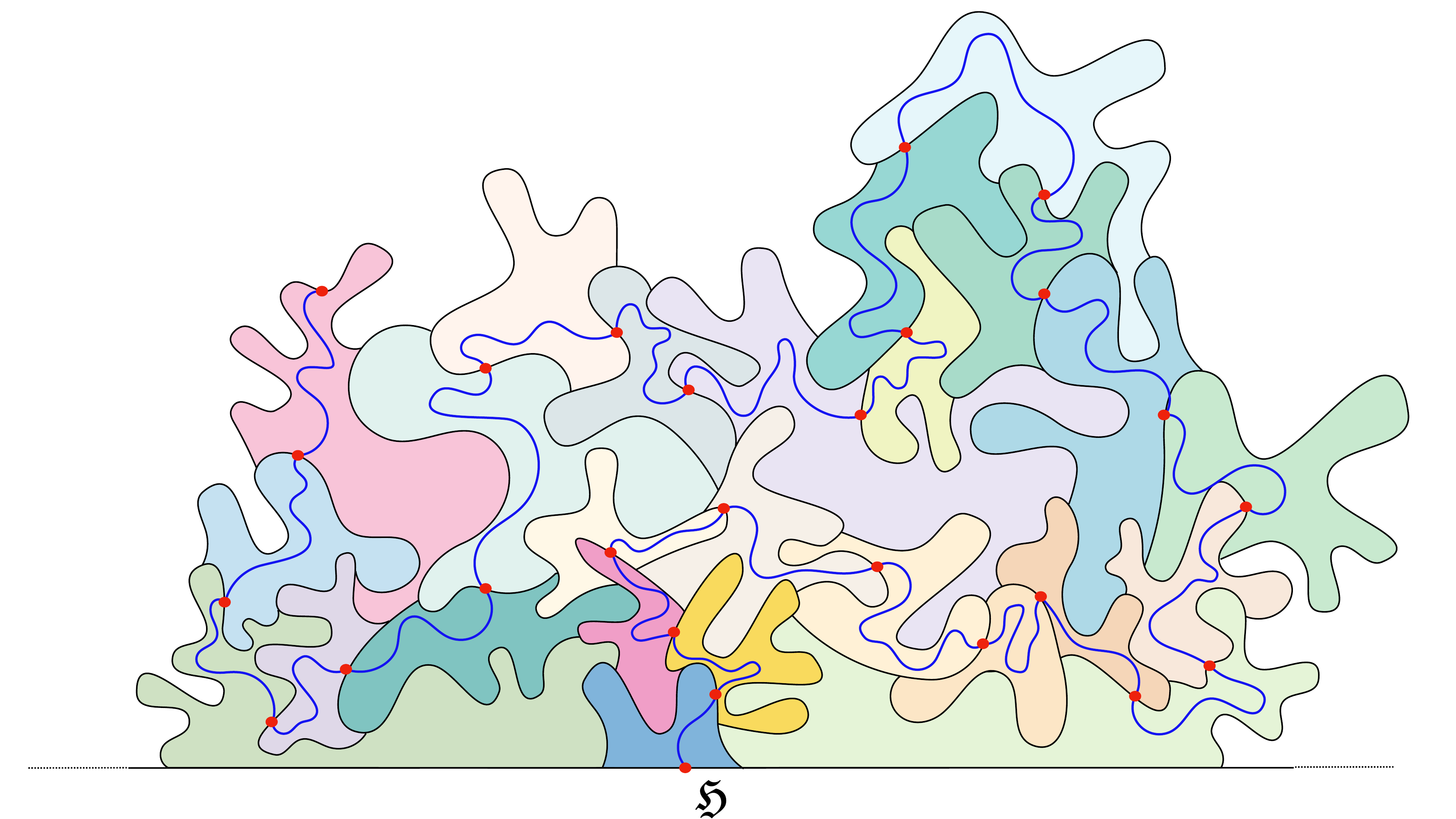}
   \caption{Illustration of the first few steps of the nearest-neighbour peeling process, together with the associated curve.}
   \label{fig:near}
\end{figure}

In this specific case, instead of keeping track of the processes $\widetilde{H}^{(\ell)}(r)$ and $\widetilde{H}^{(\mathrm{r})}(r)$, it is more relevant to study the processes $H^{(\ell)}(r)$, $H^{(\mathrm{r})}(r)$, defined as 
$$H_n^{(\ell)}(r)=\sum_{k=1}^{n}\big(Y_k^{(\ell)}(r)-X_k^{(\ell)}(r)\big),\qquad H_n^{(\mathrm{r})}(r)=\sum_{k=1}^{n}\big(Y_k^{(\mathrm{r})}(r)-X_k^{(\mathrm{r})}(r)\big).$$
Indeed, for nearest-neighbour peeling explorations, these two processes have a very natural geometric interpretation and are closely related to the so-called horodistance process considered in \cite{C2015}. 

Our main contribution is Theorem \ref{theo:exi:Gamma}, which informally contains the following results. Fix $d\in\R$, $\vartheta\geq0$, and set $\vartheta^\prime:=3\wedge (2+\vartheta)$. We start by proving that there is a phase transition for the scaling limit of the processes $\big(H^{(\ell)}(r),H^{(\mathrm{r})}(r),V(r)\big)$ occurring at $\vartheta=1$. More precisely, we prove that the tuple $$\big( H_{\lfloor r^{-\vartheta^\prime}t\rfloor}^{(\ell)}(r), H_{\lfloor r^{-\vartheta^\prime}t\rfloor}^{(\mathrm{r})}(r), V_{\lfloor r^{-\vartheta^\prime}t\rfloor}(r)\big)_{t\geq 0}$$ converges in distribution as $r\to0$ towards the following:
\begin{itemize}
    \item if $\vartheta\in(0,1)$ and $d\neq 0$, the limit is
    $\Big(\frac{2d}{3} t,-\frac{2d}{3} t,0\Big)_{t\geq 0},$
    \item if $\vartheta=1$ and $d\neq 0$, the limit is
    $\Big(\cS_t^{(\ell)}+\frac{2d t}{3},\cS^{(\mathrm{r})}_t-\frac{2d t}{3},\cV_t\Big)_{t\geq 0},$
      \item if $\vartheta>1$ or $d=0$, the limit is
    $\left(\cS_t^{(\ell)},\cS^{(\mathrm{r})}_t,\cV_t\right)_{t\geq 0}.$
\end{itemize}
Moreover, this convergence holds jointly with that of Theorem~\ref{th:conv:familledisque}. Then, we also prove results about the Brownian half-plane decorated with the curve $\Gamma^{(r)}$. More precisely for any fixed parameters $d\in\mathbb{R}$ and $\vartheta\geq 0$, we prove that the family of random variables
$$\Big(\mathfrak{H},D_\infty, \Lambda_\infty(0), \mathrm{Vol}_\infty, \Lambda_\infty, \big(\Gamma^{(r)}(r^{-\vartheta^\prime} t)\big)_{t\geq 0}\Big)$$
with $r\in(0,1)$ is asymptotically tight as $r\to0$, i.e., any sequence extracted from this family with parameters $r_k\to 0$ is tight. This tightness result holds jointly with the convergence of the processes $\big(H^{(\ell)}(r),H^{(\mathrm{r})}(r),V(r)\big)$ and $\mathbf{K}^{\bullet}(r)$. Finally, we show that, up to enlarging the probability space, every limiting curve $\Gamma$ can be constructed on top of \emph{the same} Brownian half-plane.

In addition, in the regime where $\vartheta\in[0,1)$ and $d\neq 0$, the drift induced by the random variables $(U_n(r))_{n\in\mathbb{N}}$ becomes ``stronger" than the randomness coming from the exploration, and we show in Corollary \ref{coro:limite:bord} that the family of curves $(\Gamma^{(r)}(r^{-\vartheta'}\cdot))_{r\in(0,1)}$ converges as $r\to0$ towards a curve $\Gamma$ that lies in one half of the boundary of $\mathfrak{H}$, i.e.
$$\Gamma(t)=\Lambda_\infty\big(\frac{2d}{3}t\big),\quad t\geq 0. $$

We emphasis that some parts of our theorems and the different processes involved are reminiscent of the work \cite{GM2021} on the scaling limit of a critical percolation interface in random quadrangulations mentioned earlier. Consequently, using the results of \cite{GwynneMillerCharacterization} and the equivalence between Brownian geometry and $\sqrt{8/3}$-LQG surfaces, we believe it is possible to identify SLE$_6$ on top of an independent Brownian half-plane as the unique limit $\Gamma$ of Theorem \ref{theo:exi:Gamma} in the regime where $\vartheta>1$, or $\vartheta\leq 1$ and $d=0$. However, we did not include this approach in this work since we are currently working to prove uniqueness in the whole regime $\vartheta=1$. Moreover, it appears that the limiting curve in the regime $\vartheta=1$ and $d\neq 0$ has been considered neither in Brownian geometry nor in LQG. Nonetheless, we believe that these curves are massive SLE$_6$ (with a mass parameter depending only on $d$) on top of an independent $\sqrt{8/3}$-LQG surface, parametrized by quantum natural time.
To support this guess, note that in this regime, the processes $H^{(\ell)}$ and $H^{(\mathrm{r})}$ converge toward drifted $3/2$-stable processes, and one can check that, by slightly modifying the arguments of~\cite{GM2021}, the very same processes appear in the context of near-critical percolation on random maps. Since massive SLE$_6$ is conjectured to be the scaling limit of near-critical percolation interfaces on regular lattices~\cite{GarbanPeteSchramm,MakarovSmirnov}, and in light of the results of~\cite{GM2021}, it is also natural to expect that the scaling limit of near-critical percolation interfaces on random maps should be massive SLE$_6$.
Moreover, if it turns out to be true, it would suggest that there exists a drifted version of the mating-of-trees theory~\cite{DuplantierMillerSheffield}, at least for the parameter $\gamma=\sqrt{8/3}$.

To conclude, let us mention that this is a preliminary version of this work, and it will soon be updated to include more results about the limiting curves, as well as other types of peeling explorations.

\subsubsection*{Organisation of this paper}

This paper is organized as follows. Sections~\ref{sec:preli} and~\ref{Brownian half-plane} contain a number of preliminaries about random snakes, surfaces encoded by labeled trees, Brownian disks, and the Brownian half-plane. In Section~\ref{sec:peeling}, we rigorously introduce the peeling explorations of the Brownian half-plane. Sections~\ref{sec:boundary} and~\ref{sec:discovered:regions} are devoted to the proofs of Theorems \ref{theo:convergence:hull:} and \ref{th:conv:familledisque}. Finally, the nearest-neighbour peeling explorations are introduced in Section~\ref{sec:nearest}, whereas Sections~\ref{sec:estimates} and~\ref{sec:peeling:scaling} are devoted to the proof of Theorem~\ref{theo:exi:Gamma}. 

\subsubsection*{Acknowledgements}

The authors thank Jean-François Le Gall, Jason Miller, Fael Rebei, Alejandro Rosales-Ortiz and especially Grégory Miermont for several stimulating discussions. They are also grateful to Thomas Budzinski and Nicolas Curien for organizing the conference ``Combinatorial and hyperbolic random planar geometry'', held at CIRM in 2025 and supported by ERC Consolidator Grant No. 101087572, during which substantial progress was made on this project. Finally, the first author thanks Simon Renouf for his help in preparing the manuscript.

\subsubsection*{AI statement}

In this work, the authors used large language models as tools for bibliographic research, proofreading, and improving the English writing.

\section{Preliminaries}\label{sec:preli}
We start with some standard preliminaries on the Brownian snake, the construction of decorated compact metric spaces from labeled trees, and the space of equivalence classes of decorated boundedly compact length spaces. These notions are now standard tools in Brownian geometry and are used in many works in this field. We will follow a presentation similar to that in \cite{LGR3,LGR6,MourichouxBigeodesic,RR2026}.
\subsection{Brownian snake}
\subsubsection{Snake trajectories}
We now describe the formalism of \emph{snake trajectories}, which gives a convenient way to encode labeled trees; we refer the reader to~\cite{ALG,DLG2002,Zurich} for a detailed account. Let $\mathcal{W}$ be the set of all \emph{finite paths} with real values, that is, continuous functions
$$\w : [0,\zeta_\w] \longrightarrow \R$$
with finite \emph{lifetime} $\zeta_\w \in [0,\infty)$.  
The space $\mathcal{W}$ is equipped with the distance
$$d_{\mathcal{W}}(\w,\w'):= \big| \zeta_\w - \zeta_{\w'} \big|+ \sup_{t \geq 0} \big| \w(t \wedge \zeta_\w) - \w'(t \wedge \zeta_{\w'}) \big|,$$
which compares both the lifetimes and the shapes of the paths.  
The \emph{tip} of $\w$ is $\widehat{\w} := \w(\zeta_\w)$.  
For $x \in \R$, we write $\mathcal{W}_x := \{\w \in \mathcal{W} : \w(0) = x\}$, identifying the element of $\mathcal{W}_x$ with zero lifetime with the constant   $x\in \mathbb{R}$.

\medskip
\noindent\textbf{Snake trajectories.}
Fix $x \in \R$. A \emph{snake trajectory} starting from $x$ is a continuous mapping
$$s \longmapsto \omega_s \in \mathcal{W}_x, \quad s \geq 0,$$
satisfying:
\begin{enumerate}[label=(\roman*)]
    \item \emph{Finite duration:}
    $$\omega_0 = x, \quad \sigma(\omega) := \sup\{ s \geq 0 : \omega_s \neq x \} < \infty,$$
    with $\sigma(\omega) = 0$ if $\omega_s = x$ for all $s \geq 0$.\footnote{In the literature  there are other related notions. In particular, see \cite{2024structure} for the notion  of snake paths which does not assume that the lifetime is finite \cite{2024structure}. }

    \item \emph{Snake property:} For $0 \leq s \leq s'$, we have
    $$\omega_s(t) = \omega_{s'}(t), \quad\text{for all}\quad t \in \big[ 0,\, \min_{s \leq r \leq s'} \zeta_{\omega_r} \big].$$
    Informally, the snake property entails that between times $s$ and $s^\prime$ the  two paths $\omega_s$ and $\omega_{s^\prime}$ coincide up to the smallest lifetime  in the interval $[s,s^\prime]$.
\end{enumerate}

\noindent The set of all snake trajectories starting  from $x$ is denoted by $\mathcal{S}_x$, and we let $\mathcal{S} := \bigcup_{x \in \R} \mathcal{S}_x$.  
For $\omega \in \mathcal{S}$, we write
$$W_s(\omega) := \omega_s, \quad 
\widehat{W}_s(\omega) := \widehat{\omega}_s, \quad
\zeta_s(\omega) := \zeta_{\omega_s},$$
and we introduce the distance:
$$d_{\mathcal{S}}(\omega,\omega'):= \big| \sigma(\omega) - \sigma(\omega') \big|+ \sup_{s \geq 0} d_{\mathcal{W}}\big( W_s(\omega), W_s(\omega') \big).$$
A trajectory $\omega$ is fully determined by its \emph{lifetime process} $s \mapsto \zeta_s(\omega)$ and its \emph{tip process} $s \mapsto \widehat{W}_s(\omega)$ \cite[Proposition~8]{ALG}.  
To simplify notation, we write
$$W_*(\omega) := \inf_{s \geq 0} \widehat{W}_s(\omega).$$

\medskip
\noindent\textbf{Encoding a labeled $\R$--tree.} We can naturally  associate  a labeled $\R$--tree to every snake trajectory as follows. Given $\omega \in \mathcal{S}_x$,  for every $s,t \in [0,\sigma(\omega)]$, we set
$$d_\omega(s,t):= \zeta_s(\omega) + \zeta_t(\omega)- 2\inf_{r \in [s \wedge t,\, s \vee t]} \zeta_r(\omega),$$
which is a continuous pseudo–distance on $[0,\sigma(\omega)]$. We write $s \sim_{d_\omega} t$ whenever $d_\omega(s,t) = 0$, and we introduce the quotient space
$$\mathcal{T}_\omega := [0,\sigma(\omega)] / \! \sim_{d_\omega}.$$
The space $(\mathcal{T}_\omega, d_\omega)$ is a compact $\R$–tree, and we let $\pi_\omega : [0,\sigma(\omega)] \to \mathcal{T}_\omega$ be the canonical  projection.  The \emph{volume measure} $V_\omega$ is the pushforward of Lebesgue measure by $\pi_\omega$ and the root is $\rho_\omega := \pi_\omega(0)$. Let us now introduce standard nomenclature for rooted $\mathbb{R}$-trees. 
For $u,v \in \mathcal{T}_\omega$, we write 
$\llbracket u,v \rrbracket$ for the range of the unique geodesic connecting $u$ and $v$.  We call $\llbracket \rho_\omega, u \rrbracket$ the   \emph{ancestral line} of $u$ and we say that $u$ is a \emph{descendant} of $v$ if $v\in \llbracket \rho_\omega, u \rrbracket$.  The \emph{multiplicity} of $u$ is the number of connected components of $\mathcal{T}_\omega \setminus \{u\}$. Now remark, that  by property~(ii) above, if $\pi_\omega(s) = \pi_\omega(s')$ then we must have $W_s(\omega) = W_{s'}(\omega)$.  
Hence, the map $s \mapsto W_s(\omega)$ factorizes through $\mathcal{T}_\omega$. In particular, we can set
$$\ell_u(\omega) := \widehat{W}_s(\omega),\quad\text{for any }~s\in \pi_{\omega}^{-1}(u),$$
and we call $(\ell_u(\omega))_{u\in \mathcal{T}_\omega}$ the labels of $\mathcal{T}_\omega$. Informally, for each $s\in[0,\sigma(\omega)]$, the path $W_s(\omega)$ records the sequence of labels along the ancestral line of $\pi_\omega(s)$. Let us now introduce the notion of subtrajectory.
\medskip \\
\noindent \textbf{Subtrajectories. }Fix $s<t$ such that $\zeta_s(\omega) = \zeta_t(\omega)$ and $\zeta_r(\omega) > \zeta_t(\omega)$ for all $r \in (s,t)$.   The subtrajectory of $\omega$ associated with $[s,t]$ is the snake trajectory $\omega'$ defined, for every $r\geq 0$, by $$\omega'_r(u):=\omega_{(r+s)\wedge t}\big(\zeta_s(\omega)+u\big),\quad 0\leq u\leq \zeta_{(r+s)\wedge t}(\omega)-\zeta_s(\omega).$$ In particular, it holds that: $$\zeta_{\omega'_r}=\zeta_{(r+s)\wedge t}(\omega)-\zeta_s(\omega),\quad r\geq 0.$$

\noindent Informally, $\omega^{\prime}$ encodes the labels $(\ell_{v}(\omega):~v\in \pi_{\omega}([s,t]))$ associated with  the interval $[s, t]$. 

 We conclude this  section by introducing some useful operations.
 
\medskip
\noindent\textbf{Truncation and exit local time.}
Let $y < x$. For $\w \in \mathcal{W}_x$, set
$$\tau_y(\w) := \inf\{ t \in [0,\zeta_\w] : \w(t) = y \} \in [0,\infty],$$
with $\inf\varnothing := \infty$.  
Given $\omega \in \mathcal{S}_x$, define the time change
$$\eta_s(\omega) := \inf\Big\{ t \geq 0 : \int_0^t \mathbf{1}_{\{\zeta_{\omega_r} \leq \tau_y(\omega_r)\}}\, \mathrm{d}r > s \Big\},$$
and the \emph{truncated snake}
$$\big( \tr_y(\omega) \big)_s := \omega_{\eta_s(\omega)}, \quad s \geq 0.$$
This yields another snake trajectory \cite[Prop.~10]{ALG}. Roughly speaking, in terms of  trees, the truncation operation corresponds  to pruning $\mathcal{T}_\omega$ at the first occurrence of label $y$ along each ancestral line. Finally, the \emph{exit local time} at level $y$ is given by:
\begin{equation}\label{formu-exit}
L_t^y(\omega)
:= \liminf_{\varepsilon \downarrow 0} \frac{1}{\varepsilon^2}
\int_0^t \mathbf{1}_{\{ \tau_y(\omega_s) = \infty,\, \widehat{W}_s(\omega) < y + \varepsilon \}}\, \mathrm{d}s,
\quad t \geq 0.
\end{equation}
We stress that, when the liminf in~\eqref{formu-exit} is a  limit uniform on compacts, the process $(L_t^y(\omega))_{t \geq 0}$ is continuous, non–decreasing, and supported on $\{t\geq 0: \tau_y(\omega_t) = \zeta_{\omega_t}\}$.  
The terminal value $L_\infty^y(\omega)$ is the \emph{exit measure} from level $y$ and by \eqref{formu-exit} it depends only on $\tr_y(\omega)$, i.e.  $L_\infty^y(\omega)= L_\infty^y(\tr_y(\omega))$. We will also use the standard notation:
$$\mathcal{Z}_y(\omega):=L_\infty^y(\omega).$$

\subsubsection{The Brownian snake excursion measure}
\label{sna-mea}

Fix $x \in \R$.  
The \emph{Brownian snake excursion measure} $\N_x$ is the $\sigma$--finite measure on $\mathcal{S}_x$ characterized by the following two properties.

\medskip
\noindent\textbf{Characterization.}  
Under $\N_x$,
\begin{enumerate}[label=(\roman*)]
    \item \emph{Lifetime process:}  
    The process $(\zeta_s)_{s \geq 0}$ has the law of the Itô measure of positive excursions of linear Brownian motion, normalized such that:
$$\N_x\Big( \sup_{s \geq 0} \zeta_s > \varepsilon \Big) = \frac{1}{2\varepsilon}\ , \quad \text{ for }~\varepsilon > 0.$$

    \item \emph{Tip process:}  
    Conditionally on $(\zeta_s)_{s \geq 0}$, the process $(\widehat{W}_s)_{s \geq 0}$ is Gaussian with mean $x$ and covariance
    $$(s,s') \mapsto \min_{\,s \wedge s' \,\leq r \,\le\, s \vee s'} \zeta_r.$$
\end{enumerate}

\noindent Informally, the lifetime process behaves like a Brownian excursion, and given $(\zeta_s)$, each $W_s$ evolves as a linear Brownian motion started at $x$ with lifetime $\zeta_s$, truncated whenever $\zeta_s$ decreases and extended when it increases.  
We also mention that the measure $\N_x$ can  be seen as the excursion measure away from $x$ for the Markov process on $\mathcal{W}_x$ known as the \emph{Brownian snake}.  
A detailed treatment appears in~\cite{Zurich}. Let us conclude this section with a quick discussion on exit local times and the special Markov property. Fix $y < x$.  
Under $\N_x$, the $\liminf$ in~\eqref{formu-exit} may be replaced by a limit uniform in $t \geq 0$. Specifically, we have:
\begin{equation}\label{uniform:L}
    \lim_{\varepsilon \downarrow 0} \ \sup_{t \geq 0} 
    \Big|\, L_t^y(W) - \varepsilon^{-2} \! \int_0^t \mathbf{1}_{\{\tau_y(W_s) = \infty,\, \widehat{W}_s < y + \varepsilon\}} \,\mathrm{d}s \,\Big| = 0,
    \quad \N_x\text{-a.e.}
\end{equation}
Moreover, $\N_x$--a.e., the topological support of $\mathrm{d}L_t^y$ is exactly
$$\big\{\, s \in [0,\sigma] : \tau_y(W_s) = \zeta_s \,\big\},$$
so that $L^y_\infty(W) > 0$ if and only if at least one path $W_s$ hits the level $y$. Using this property and \cite[Section VI.1]{Zurich}, it follows that:
\begin{equation}\label{hittingpro}
    \N_x\big( W_* \leq y \big) = \frac{3}{2 (x-y)^2}.
\end{equation}
Let us close this short  section by introducing the special Markov property of the Brownian snake.

\paragraph{The special Markov property.} Consider $y<x$  and for  $\omega\in\mathcal{S}_x$, we write
$(\alpha_j,\beta_j)$, $j\in J$, for the connected components of $\{r\in[0,\sigma(\omega)]:~\zeta_{r}(\omega)> \tau_{y}(\omega_r)\}$, and $\omega^j$,  $j\in J$, for the sub-trajectory associated with $[\alpha_j,\beta_j]$. The special Markov property states that, under $\N_x$, conditionally on the truncation $\tr_y(\omega)$,  the point measure
\begin{equation}\label{measure:special:mark}
\sum \limits_{j\in J} \delta_{(L^y_{\alpha_j},\omega^{j})}
\end{equation}
is Poisson with intensity $\mathbf{1}_{[0,\z_y]}(t)\,\mathrm{d} t \,\mathbb{N}_y(\mathrm{d} \omega')$. We refer the reader to the Appendix of~\cite{LG2018} for a proof of this result, and \cite[Section 3]{2024structure} for a generalization to larger classes of models. 
In what follows, we shall make use of the following minor extension. 
Let $x > y > 0$. 
Then the special Markov property continues to hold in the same form under 
$\mathbb{N}_x(\,\cdot \cap \{ W_* > 0 \})$,
provided the intensity measure is replaced by
$$\mathbf{1}_{[0,\z_y]}(t)\,\mathrm{d}t\,\mathbb{N}_y\big(\mathrm{d}\omega' \cap \{ W_*(\omega') > 0 \}\big).$$
The proof, which is straightforward, is omitted.

\subsection{From labeled trees to decorated compact metric spaces}\label{Section:Coding:Pair}
In this section, we explain how to construct a  decorated compact metric space from a real function and a point measure. In this direction we rely on the notion of coding pairs, see \cite[Section 2.4]{LGR3} and \cite[Section 8.2]{These:R}  for similar constructions and \cite{RR2026,RRO2026} for  generalizations.  A coding pair is a pair $(\mathrm{w},\mathcal{P})$ such that:

\begin{enumerate}
    \item[(i)] $\mathrm{w}:[0,z]\to\mathbb{R}_+$ is a continuous function, with $z>0$,  $\mathrm{w}(0)=\mathrm{w}(z)=0$ and $\mathrm{w}(t)>0$ for $t\in(0,z)$;
    \item[(ii)] $\mathcal{P}=\sum_{i\in I}\delta_{(t_i,\omega^i)}$ is a point measure on $[0,z]\times \mathcal{S}$ such that, for every $i\in I$, $\omega^i\in \mathcal{S}_{\mathrm{w}(t_i)}$ with $(\omega^i)_*>0$;
    \item[(iii)] the numbers $(t_i)_{i\in I}$ are distinct;
    \item[(iv)] the function
    \begin{equation}\label{eq:mass:contour}
            u\longmapsto \beta_u := \sum_{i\in I}\mathbf{1}_{\{t_i\leq u\}}\sigma(\omega^i)
    \end{equation}
    takes finite values and is monotone increasing on $[0,z]$;
    \item[(v)]for every $\varepsilon>0$,
    $$ \#\Big\{i\in I:\sup_{s\in[0,\sigma(\omega^i)]}\big(\big|\widehat\omega^i_s-\mathrm{w}(t_i)\big|+\zeta_s(\omega^i)\big)>\varepsilon
    \Big\}<\infty.$$
\end{enumerate}
Let $(\mathrm{w},\mathcal{P})$ be a coding pair with $\mathrm{w}:[0,z]\to\mathbb{R}_+$. We can associate to $(\mathrm{w},\mathcal{P})$ a compact $\mathbb{R}$-tree $\mathcal{T}$ where informally, the function $\mathrm{w}$ describes the spine of a tree (and the labels on it), while the point measure $\cP$ represents the subtrees lying on one side of the spine. 
Formally, the tree $\mathcal{T}$ is defined from the disjoint union
$$[0,z]\cup\left(\bigcup_{i\in I}\mathcal{T}_{\omega^i}\right)$$
by identifying the root $\rho_{\omega^i}$ of $\mathcal{T}_{\omega^i}$ with the point $t_i \in [0,z]$ for each $i \in I$. Distances are defined by:
\begin{itemize}
    \item for  $u,v \in [0,z]$: $\dd_{\mathcal{T}}(u,v) = |u-v|$;
        \item for  $u,v \in \mathcal{T}_{\omega^i}$: $\dd_{\mathcal{T}}(u,v) = d_{\omega^i}(u,v)$;
    \item for $u \in [0,z]$, $v \in \mathcal{T}_{\omega^i}$: $\dd_{\mathcal{T}}(u,v) = |u-t_i| + d_{\omega_i}(\rho_{\omega^i},v)$;
    \item for  $u \in \mathcal{T}_{\omega^i}$, $v \in \mathcal{T}_{\omega^j}$ with $i \neq j$: $\dd_{\mathcal{T}}(u,v) = d_{\omega^i}(u,\rho_{(\omega^i)}) + |t_i-t_j| + d_{\omega^j}(\rho_{\omega^j},v)$.
\end{itemize}
We also assign a label $\ell_u$ to each $u\in\mathcal{T}$ by taking
$$\ell_u:=\mathrm{w}(u),\quad\text{if }u\in[0,z],\quad\text{and}\quad\ell_u:=\ell_u(\omega^i),\quad\text{if }u\in\mathcal{T}_{\omega^i},$$
and we let $\mathrm{Vol}_{\mathcal{T}}$ be the measure on $\mathcal{T}$ which gives no mass to the spine and whose restriction to $\mathcal{T}_{\omega^i}$ is $V_{\omega^i}$. It will be useful to equip $\mathcal{T}$ with an order structure inherited from the coding pair. First, observe that $\beta_z=\sum_{i\in I}\sigma(\omega^i)$ is the total mass of $\mathrm{Vol}_{\mathcal{T}}$, where $\beta$ was defined in \eqref{eq:mass:contour}. We write $\beta_{u-}$ for the left limits of $\beta$, and take $\beta_{0-}=0$ by convention. Note that if $u$ is a discontinuity point of $\beta$, then there is a unique $i\in I$ such that $t_i=u$ and $\beta_u-\beta_{u-}=\sigma(\omega^i)$. We define a function $\cE:[0,\beta_z]\rightarrow \cT$ as follows.
For every $s\in[0,\beta_z]$, there exists a unique $u$ such that
$\beta_{u-}\leq s\leq \beta_u.$ Then :
\begin{itemize}
    \item If there exists $i\in I$ such that $t_i=u$, we set $\mathcal{E}_s:=\pi_{\omega^i}(s-\beta_{t_i-})$;
    \item and otherwise we set $\mathcal{E}_s:=u.$
\end{itemize}
The process $\mathcal{E}$ is called the contour exploration of $\mathcal{T}$. As a consequence of the assumptions made on $(\mathrm{w},\cP)$, it is straightforward to verify that $\cE$ is continuous, and to verify that the volume measure on $\mathcal{T}$ is the pushforward of Lebesgue measure on $[0,\beta_z]$ under the mapping $s\mapsto\mathcal{E}_s$. This exploration process allows  to define a notion of interval on $\mathcal{T}$. By convention, for $s,t\in[0,\beta_z]$ with $s<t$, we write
$$[t,s]_{\mathcal{T}}:=[0,s]\cup[t,\beta_z],\quad [s,t]_{\mathcal{T}}:=[s,t].$$
For every $u,v\in\mathcal{T}$ with $u\ne v$, let $[s,t]_{\mathcal{T}}$ be the smallest interval such that $\mathcal{E}_s=u$ and $\mathcal{E}_t=v$.
We put
$$[|u,v|]_{\mathcal{T}}:=\{\mathcal{E}_r:r\in[s,t]_{\mathcal{T}}\}.$$
By convention, if $u=v$, we take $[|u,v|]_{\mathcal{T}}:=\{u\}. $ Notice that, in general, $[|u,v|]_{\mathcal{T}}\ne[|v,u|]_{\mathcal{T}}$ if $u\neq v.$

\noindent \textbf{Pseudo–distance and quotient space.}  
For $u,v \in \mathcal{T}$, define
\begin{equation}\label{formula-Dcirc}
D^{\circ}(u,v) := \ell_u + \ell_v - 2\cdot \max\Big( \min_{c \in [|u,v|]_{\mathcal{T}}} \ell_c,\ \min_{c \in [|v,u|]_{\mathcal{T}}} \ell_c \Big),
\end{equation}
and
\begin{equation}\label{def-D}
D(u,v) := \inf_{u=u_0,\dots,u_p=v} \ \sum_{i=1}^p D^{\circ}(u_{i-1},u_i),
\end{equation}
where the infimum is over all choices of the integer $p\geq 1$ and of the
finite sequence $u_0,u_1,\ldots,u_p$ in $\mathcal{T}$ such that $u_0=u$ and
$u_p=v$.  By (i) one has $D(0,z) = 0$. The function $D$ is a pseudo-distance on $\mathcal{T}$ and  we write $\Pi$ for the canonical projection and set
$$\Theta(\mathrm{w},\mathcal{P}):=(\mathcal{T}/\{D=0\},D, \Pi(0), \operatorname{Vol},(\Pi(t))_{t\in[0,z]}),$$
where $\operatorname{Vol}$ is the pushforward of $\mathcal{V}_{\mathcal{T}}$ under $\Pi$. Observe that the distance $D$ satisfies the following upper and lower bounds:
\begin{equation}\label{inegalité:label}
D(a,b)\leq D^\circ(a,b)\quad\text{and}\quad D(a,b)\geq |\ell_a-\ell_b|. 
\end{equation}
In particular, we have $\ell_a=\ell_b$ for any $a,b\in\mathcal{T}$ such that $\Pi(a)=\Pi(b)$. This entails that the labelling function factorizes through $\Pi$, and we set $\ell_x=\ell_{u}$, where $u$ is any element $\Pi^{-1}(\{x\})$. We now specify how to use this general construction to define the Brownian disk.

\paragraph{The Brownian disk.} We follow the presentation in~\cite{LG2022} of  the construction of the Brownian disk of perimeter~$z>0$.\footnote{The case $z=1$ is the one considered in~\cite{LG2022}, but the general case is obtained by scaling.}  In this direction we consider  a probability space $(\Omega_z, \mathcal{F}^{z},\P_z)$ and a random variable $(\mathbf{b},\m)$ on this space, where $\mathbf{b} = (\mathbf{b}_t)_{0 \leq t \leq z}$ is a five–dimensional Bessel bridge from $0$ to $0$ with lifetime $z$, and, conditionally on $\mathbf{b}$, $\m(\mathrm{d}t\,\mathrm{d}\omega)$ is a Poisson point measure on $[0,z] \times \mathcal{S}$ with intensity
$$2\cdot\mathbf{1}_{\{W_*(\omega) > 0\}}\,\mathrm{d}t\ \mathbb{N}_{\sqrt{3}\,\mathbf{b}_t}(\mathrm{d}\omega).$$
We write $\m = \sum_{j \in J} \delta_{(t_j,\omega_j)}$ and set $\Sigma := \sum_{j \in J} \sigma(\omega_j)$. The restriction $W_*(\omega_j) > 0$ ensures that all labels of the  snakes trajectories $(\omega_j)_{j\in J}$ are nonnegative. It follows from the definition that $(\sqrt{3}\cdot \mathbf{b},\m)$ is a coding pair.

To simplify notation, under $\mathbb{P}_z$, we set 
\begin{equation}\label{first:disk}
    (\D, D,  \Pi(0),\mathrm{Vol} ,(\Pi(t))_{t\in[0,z]}):= \Theta(\sqrt{3}\cdot \mathbf{b}, \m).
\end{equation}
By \cite{LG2022}, under $\P_z$, the 5-uplet \eqref{first:disk} is a free Brownian disk of perimeter $z$ pointed at a uniform boundary point equipped with a standard boundary curve, and if we forget the function $(\Pi(t))_{t\in[0,z]}$, the  quadruplet:
$$\big(\D, D, \Pi(0), \mathrm{Vol} \big)$$
is a free Brownian disk of perimeter $z$ pointed at a uniform point of the boundary. Let us conclude with a discussion of almost sure properties verify under $\mathbb{P}_z$. By ~\cite[Theorem~13]{LG2019}, $D(a,b) = 0$ iff $D^{\circ}(a,b) = 0$. Moreover, by \cite[Theorem 2]{Bet:2015}, $\big(\D, D\big)$ is homeomorphic to the closed unit disk of the complex plane and we write $\partial_\circ \D$ for the set of all points with no neighborhood homeomorphic to the open disk.  We stress that we use the notation $\partial_\circ$ since this definition does not coincide with the topological boundary of $\D$ as a metric space, which is empty.
The mapping $[0,z] \ni t \mapsto \Pi(t)$ is a simple loop whose range is $\partial_\circ\D$. This parametrization is particularly adapted since it follows that
$$ t= \lim \limits_{\varepsilon \to 0} \varepsilon^{-2}\int_{\D} \mathrm{Vol}(\mathrm{d}x )\mathbf{1}_{D(\Pi([0,t]),x)<\varepsilon},\quad \text{for every } t\in [0,z],$$
see \cite[Theorem~9]{LG2019}. This is the  reason why   $[0,z] \ni t \mapsto \Pi(t)$ is called a standard boundary curve.
\subsection{Equivalence classes of decorated  metric spaces }
In this section, we present variants of the Gromov-Prohorov-uniform distance for decorated boundedly compact spaces. The definitions follow \cite[Section 1.3]{GM2017} and \cite[Section 2.6.3 and 3.7.3]{K2023}.
\par
Let $(E,d)$ be a compact metric space. We denote by
$C_0(\mathbb{R},E)$ the space of continuous curves $\gamma:\mathbb{R}\to E$ such that $\gamma(t)$ admits limits as $t\to\infty$ and as $t\to-\infty$.  A curve defined on a compact interval is viewed as an element of $C_0(\mathbb{R},E)$ by extending it constantly outside this interval. A pointed two-curve-decorated compact measure metric space is a sextuple
$$ \mathcal{E}=(E,d,\rho,\mu,\gamma_1,\gamma_2),$$
where $(E,d)$ is a compact metric space, $\rho\in E$, $\mu$ is a finite Borel measure on
$E$, and $\gamma_1,\gamma_2\in C_0(\mathbb{R},E)$ satisfy
$\gamma_1(0)=\gamma_2(0)=\rho$.
For two such spaces
$$ \mathcal{E}=(E,d,\rho,\mu,\gamma_1,\gamma_2)\quad\text{and}\quad \mathcal{E}'=(E',d',\rho',\mu',\gamma_1',\gamma_2'),$$
we set
$$
\begin{aligned}
  d_{\mathrm{GHPU}}\big(\mathcal{E},\mathcal{E}'\big):=\inf_{\Phi,\Phi'}
  \Big( & d_H^Z\big(\Phi(E),\Phi'(E')\big) \vee d_P^Z\big(\Phi_\#\mu,\Phi'_\#\mu'\big)  \vee d_Z\big(\Phi(\rho),\Phi'(\rho')\big) \\
  & \vee \sup_{t\in\mathbb{R}}  d_Z\big(\Phi(\gamma_1(t)),\Phi'(\gamma_1'(t))\big)  \vee \sup_{t\in\mathbb{R}}  d_Z\big(\Phi(\gamma_2(t)),\Phi'(\gamma_2'(t))\big)\Big),
\end{aligned}
$$
where the infimum is over all isometric embeddings $\Phi:E\to Z$ and $\Phi':E'\to Z$ into the same compact metric space $(Z,d_Z)$. Here $d_H^Z$ and $d_P^Z$ denote the Hausdorff and Prokhorov distances in $Z$. This defines a
pseudo-distance on the collection of pointed two-curve-decorated compact measure metric spaces.
\par 
Next, recall that a metric space is boundedly compact if all its closed balls are compact. A pointed two-curve-decorated boundedly compact length measure metric space is a sextuple
$$ \mathcal{X}=(X,d,\rho,\mu,\gamma_1,\gamma_2),$$
where $(X,d)$ is a boundedly compact length space, $\rho\in X$, $\mu$ is a Borel measure on $X$ which is finite on compact sets, and
$\gamma_1,\gamma_2:\mathbb{R}\to X$ are continuous and satisfy $\gamma_1(0)=\gamma_2(0)=\rho$. Let $\mathcal{X}=(X,d,\rho,\mu,\gamma_1,\gamma_2)$ be such a space. For $r>0$, let $B_r(X)$ be the closed ball of radius $r$ centered at $\rho$. For $i\in\{1,2\}$, we set
$$ \underline \tau_r^{\gamma_i} := (-r)\vee \sup\big\{t<0:d\big(\rho,\gamma_i(t)\big)=r\big\},\quad
  \overline \tau_r^{\gamma_i} :=r\wedge \inf\big\{t>0:d\big(\rho,\gamma_i(t)\big)=r\big\},$$
with the conventions $\sup\emptyset=-\infty$ and $\inf\emptyset=\infty$. Define
$$\gamma_i^r(t):= \gamma_i\big(\big(t\wedge\overline \tau_r^{\gamma_i}\big) \vee\underline \tau_r^{\gamma_i}\big), \quad t\in\mathbb{R},\quad i\in\{1,2\},$$
and
$$ \mathcal{B}_r(\mathcal{X}):=\big( B_r(X),d|_{B_r(X)},\rho,\mu|_{B_r(X)},\gamma_1^r,\gamma_2^r \big).$$
We stress that by construction, the sextuplet $\mathcal{B}_r(\mathcal{X})$ is a pointed two-curve-decorated compact measure metric space. If $\mathcal{X}'=(X',d',\rho',\mu',\gamma_1',\gamma_2')$ is another such space,
we set
$$ d_{\mathrm{GHPU}}^\infty\big(\mathcal{X},\mathcal{X}'\big):= \int_0^\infty \mathrm{d}r\,\exp(-r)\cdot \big(1\wedge d_{\mathrm{GHPU}}\big(\mathcal{B}_r(\mathcal{X}),\mathcal{B}_r(\mathcal{X}') \big)\big).$$
We have
$d_{\mathrm{GHPU}}^\infty(\mathcal{X},\mathcal{X}')=0$ if and only if
$\mathcal{X}$ and $\mathcal{X}'$ are equivalent in the sense that there exists
a bijective isometry $\varphi:X\to X'$ such that
$$\varphi(\rho)=\rho',\quad\varphi_\#\mu=\mu',\quad\varphi\circ\gamma_1=\gamma_1',\quad\varphi\circ\gamma_2=\gamma_2'.$$
We write $\mathcal{X}\sim\mathcal{X}'$ if $\mathcal{X}$ and $\mathcal{X}'$ are equivalent,  and we let $\mathbb{M}_\infty^{\mathrm{GHPU}}$ be the set of equivalence classes of pointed two-curve-decorated boundedly
compact measure metric spaces. We write $[\mathcal{X}]$ for the equivalence
class of $\mathcal{X}$. By the standard properties of the local GHPU topology,
$d_{\mathrm{GHPU}}^\infty$ is a distance on $\mathbb{M}_\infty^{\mathrm{GHPU}}$, and the topology induced by $d_{\mathrm{GHPU}}^\infty$ is Polish\footnote{However, the distance $d_{\mathrm{GHPU}}^\infty$  is not complete.}, see \cite[Section 4.3]{K2023}.  In the sequel, pointed measure metric spaces without non-trivial curves are identified with the special case where both curves are constant and equal to the root. Spaces equipped with only one curve $\gamma$ are identified with the special case where the first curve is $\gamma$ and the second curve is constant and equal to $\gamma(0)$.

We conclude by introducing a scaling operation in
$\mathbb{M}_\infty^{\mathrm{GHPU}}$. For every $\lambda>0$ and every pointed two-curve-decorated boundedly compact measure metric space $\mathcal{X}$, we write
$$ \lambda\cdot\mathcal{X}:=\big(  X,\lambda d,\rho,\lambda^4\mu, \gamma_1(\lambda^{-2}\cdot),\gamma_2(\lambda^{-3}\cdot) \big).$$
Since this operation is clearly compatible with the equivalence relation,  without ambiguity, we set:
$$\lambda\cdot[\mathcal{X}]:=[\lambda\cdot\mathcal{X}].$$

\section{The Brownian half-plane}\label{Brownian half-plane}
\subsection{Construction of the Brownian half-plane}

We present the Caraceni-Curien construction of the Brownian half-plane \cite{CC20218,LG2022}. This construction is the infinite-volume counterpart of that of the Brownian disk presented in Section \ref{Section:Coding:Pair}. To this end, we consider a probability space $(\Omega, \mathcal{F}, \P )$ and, on this space, a process $\mathcal{R} = (\mathcal{R}_t)_{t\in\mathbb{R}}$ such that $(\mathcal{R}_t)_{t\geq 0}$ and $(\mathcal{R}_{-t})_{t\geq 0}$ are independent five-dimensional Bessel processes started from $0$. Moreover, conditionally on $\mathcal{R}$, we consider a Poisson point measure $\mathcal{N}(\mathrm{d}t\,\mathrm{d}\omega)$ on $\mathbb{R} \times \mathcal{S}$ with intensity
$$2\cdot \mathbf{1}_{\{W_*(\omega) > 0\}}\,\mathrm{d}t \ \mathbb{N}_{\sqrt{3}\,\mathcal{R}_t}(\mathrm{d}\omega).$$
We write $\mathcal{N} = \sum_{i\in I} \delta_{(t_i,\omega^i)}$.  We then build the metric space $\mathfrak{T}_\infty$ by taking the disjoint union
$$\mathbb{R} \ \cup\ \bigcup_{i\in I} \mathcal{T}_{\omega^i}$$
and identifying the root $\rho_{\omega^i}$ of  $\mathcal{T}_{\omega^i}$ with the point $t_i \in \mathbb{R}$, for every $i\in I$. The metric $\mathrm{d}_{\mathfrak{T}_\infty}$ is defined so that $\mathbb{R}$ is isometric to the real line, each $\mathcal{T}_{\omega^i}$ carries its intrinsic metric $d_{\omega^i}$, and distances between points in different components are computed via their attachment points. 
 By construction,  $(\mathfrak{T}_\infty,\mathrm{d}_{\mathfrak{T}_\infty})$ is a non-compact $\mathbb{R}$-tree, and its volume measure is the sum of the volume measures on  $\mathcal{T}_{\omega^i}$, $i\in I$.

The exploration process $(\mathcal{E}^\infty_t)_{t\in\mathbb{R}}$ is then obtained by concatenating the contour maps $\pi_{\omega^i}$ in the order induced by the  $t_i$, $i\in I$, with $\mathcal{E}^\infty_0 = 0$ and range $\mathbb{R}_+ \cup \bigcup_{t_i>0} \mathcal{T}_{\omega^i}$ for $t\geq 0$. We refer to~\cite[Section~4.1]{LGR3} for a formal definition in a slightly different context. Given $s,t\in\mathbb{R}$, we define
$$[s,t]_\infty :=\begin{cases}
[s,t], & \text{if } s \leq t, \\
[s,\infty) \cup (-\infty,t], & \text{if } t < s.
\end{cases}$$
For any $u,v\in \mathfrak{T}_\infty$ there is a smallest interval for inclusion $[s,t]_\infty$ with $\mathcal{E}^\infty_s = u$ and $\mathcal{E}^\infty_t = v$, and we write:
$$[|u,v|]_\infty := \{ \mathcal{E}^\infty_r : r \in [s,t]_\infty \}.$$
We can now associate labels to the points of $\mathfrak{T}_\infty$, and define the analog of the pseudo-distances \eqref{def-D}. In this direction, labels are assigned as follows. If $u\in\mathbb{R}$ we set $\ell^\infty_u := \sqrt{3}\,\mathcal{R}_u$, while if $u\in\mathcal{T}_{\omega^i}$ we take $\ell^\infty_u$ to be the intrinsic label in $\mathcal{T}_{\omega^i}$ (the definitions agree at attachment points). For $u,v\in\mathfrak{T}_\infty$ we then define
\begin{equation}\label{eq:Dcirc}
D^\circ_\infty(u,v) := \ell^\infty_u + \ell^\infty_v - 2\max\!\left( \min_{c\in[|u,v|]_\infty} \ell^\infty_c, \ \min_{c\in[|v,u|]_\infty} \ell^\infty_c \right),
\end{equation}
and
\begin{equation}\label{eq:D}
D_\infty (u,v) := \inf_{u=u_0,\dots,u_p=v} \ \sum_{k=1}^p D^\circ_\infty (u_{k-1},u_k),
\end{equation}
where the infimum is over all $p\geq 1$ and intermediate points $u_1,\dots,u_{p-1}\in\mathfrak{T}_\infty$. It follows by the definition and triangle inequality that:
\begin{equation}\label{borne:triviale}
D_\infty(u,v)\geq |\ell^\infty_u-\ell^\infty_v|
\end{equation}
and, it holds that:
\begin{equation}\label{points:indentifie:D:inf}
 D_\infty (u,v) = 0\quad \text{ if and only if }   D_\infty ^\circ(u,v) = 0,
\end{equation}
we refer to  \cite[Section 3.1]{LGR6} for details. We can now define the Brownian half-plane. To this end, we set $\mathfrak{H} := \mathfrak{T}_\infty / \{ D_\infty =0 \}$ and write $\Pi_\infty$ for the canonical projection. The quotient space $\mathfrak{H}$ is endowed with the induced metric (still denoted $D_\infty$) and with the volume measure $\mathrm{Vol}_\infty$, which is the pushforward of the volume measure on $\mathfrak{T}_\infty$ by $\Pi_\infty$.
The inequality $D_\infty(a,b) \geq |\ell^\infty_a - \ell^\infty_b|$ implies that $\Pi_\infty(a)=\Pi_\infty(b)$ forces $\ell^\infty_a = \ell^\infty_b$. In particular,  if we let $\mathbf{x}_0 := \Pi_\infty(0)$, distances from $\mathbf{x}_0$ are given by the  labels: 
\begin{equation}\label{distance:to:the:root}
D_\infty(\mathbf{x}_0,\Pi_\infty(a)) = \ell^\infty_a,\quad a\in \mathfrak{T}_\infty.
\end{equation}
The metric space $(\mathfrak{H},D_\infty)$ is a boundedly compact length space homeomorphic to the closed half-plane. We call the set of points with no neighborhood homeomorphic to the open disk the boundary of $\mathfrak{H}$ and denote it by $\partial_\circ \mathfrak{H}$. We have $\partial_\circ \mathfrak{H} = \Pi_\infty(\mathbb{R})$. We recall  that we use the notation $\partial_\circ$ since this definition does not coincide with the topological boundary. Furthermore, for every $t\geq 0$, it holds that:
\begin{equation}\label{stand:Brownian:HPlane}
    t=\lim_{\varepsilon\to 0}\varepsilon^{-2}\int_{\mathfrak{H}}\mathrm{Vol}_\infty(\mathrm{d}x)\mathbf{1}_{\{D_\infty(\Pi_\infty([0,t]),x)\leq\varepsilon\}}=\lim_{\varepsilon\to 0}\varepsilon^{-2}\int_{\mathfrak{H}}\mathrm{Vol}_\infty(\mathrm{d}x)\mathbf{1}_{\{D_\infty(\Pi_\infty([-t,0]),x)\leq\varepsilon\}}.
\end{equation}
We set $\Lambda_\infty(t) := \Pi_\infty(t)$, for $t\in\mathbb{R}$. The pointed measure metric space $(\mathfrak{H},D_\infty,\Lambda_\infty(0),\mathrm{Vol}_\infty)$ is a Brownian half-plane, and $(\mathfrak{H},D_\infty,\Lambda_\infty(0),\mathrm{Vol}_\infty,\Lambda_\infty)$ is a \emph{curve-decorated Brownian half-plane} where by \eqref{stand:Brownian:HPlane}, we say that $\Lambda_\infty$ is a standard boundary curve. Again, we refer to  \cite[Section~3.1]{LGR6} for details.  Remark that  \eqref{distance:to:the:root} makes it simple to control distances along $\partial_\circ \mathfrak{H}$. In particular, we have the following result.
\begin{lemma}\label{control:distance:bord}
There exists $c,C>0$ such that, for every $r,x>0$, we have
$$ \P\big(\sup_{t\in [-r,r]} D_\infty(\Lambda_\infty(0),\Lambda_\infty(t))>x\big)\leq C\cdot   \exp\big(-c\cdot \frac{x^2}{ r}\big).$$
\end{lemma}
\begin{proof}
By \eqref{distance:to:the:root}, we derive that:
$$ \P(\sup_{t\in [-r,r]} D_\infty(\Lambda_\infty(0),\Lambda_\infty(t))>x)=\P(\sup_{t\in [-r,r]} \mathcal{R}_t>x/\sqrt{3})\leq 2 \cdot \P(\sup_{t\in [0,r]} \mathcal{R}_t>x/\sqrt{3}).$$
Using the fact that $\cR$ has the same law as the norm of a five-dimensional standard Brownian motion and applying a union bound, it follows that  the last display is bounded above by $40\cdot \exp(-x^2/(30r))$.
\end{proof}
\noindent The curve-decorated Brownian half-plane satisfies the following invariance properties:
\begin{itemize}
\item[(i)] Invariance by shifting along the boundary curve. For every $r\in\mathbb{R}$, we have:
\begin{equation}\label{invariance:shiflting}
\big[\mathfrak{H},D_\infty ,\Lambda_\infty(0),\mathrm{Vol}_\infty ,\Lambda_\infty \big]\overset{(d)}{=}\big[\mathfrak{H},D_\infty ,\Lambda_\infty(r),\mathrm{Vol}_\infty ,\Lambda_\infty(r+\cdot)\big]. 
\end{equation}
\item[(ii)] Scaling invariance. For every $\lambda>0$, we have:
\begin{equation}\label{invariance:scaling}
\big[\mathfrak{H},\lambda D_\infty,\Lambda_\infty(0),\lambda^4\mathrm{Vol}_\infty,\Lambda_\infty(\lambda^{-2}\cdot)\big]\overset{(d)}{=}\big[\mathfrak{H},D_\infty,\Lambda_\infty(0),\mathrm{Vol}_\infty,\Lambda_\infty\big].
\end{equation}
\end{itemize}
\subsection{The spatial Markov property}\label{sec:bhp:markov}
We now introduce the spatial Markov property established in \cite{LGR6}. This property is reminiscent of the peeling process of Boltzmann planar maps. In the rest of the section, we fix $r>0$. First, remark that by \eqref{distance:to:the:root} we have
$$B_r(\mathfrak{H}) = \Pi_\infty\big(\{ a\in\mathfrak{T}_\infty : \ell^\infty_a \leq r \}\big).$$
Since the Brownian half-plane is one ended, the complement of  $B_r(\mathfrak{H})$ has a unique unbounded connected  component. The \emph{hull} $B_r^\bullet(\mathfrak{H})$ is the complement of this unbounded connected component. It admits the label characterization
\begin{equation}\label{hull}
B_r^\bullet(\mathfrak{H}) = \Pi_\infty\big(\{ a\in\mathfrak{T}_\infty : m_a \leq r \}\big),
\end{equation}
where $m_a$ is the minimal label along the geodesic from $a$ to infinity not passing through $0$; we refer to \cite[Section~3.1]{LGR6} for a proof. The topological boundary of the hull is
\begin{equation}\label{eq:bdry-hull}
\partial B_r^\bullet(\mathfrak{H}) = \Pi_\infty\left( \{ a \in\mathfrak{T}_\infty: \ell^\infty_a = r \ \text{and} \ \ell^\infty_b > r ~\text{for all}~ b\in G_a\setminus\{a\} \} \right),
\end{equation}
where $G_a$ denotes the unique geodesic in $\mathfrak{T}_\infty$ from $a$ to infinity avoiding $0$.  Setting
\begin{equation}\label{dernier temps d'atteinte}
\beta_r := \sup\{ s\geq 0 : \sqrt{3}\,\mathcal{R}_{-s} \leq r \}\quad \text{ and } \quad
\gamma_r := \sup\{ s\geq 0 : \sqrt{3}\,\mathcal{R}_s \leq r \},
\end{equation}
we have:
$$\Pi_\infty\big([-\beta_r, \gamma_r]\big)= \partial_\circ \mathfrak{H} \cap  B_r^\bullet(\mathfrak{H}).$$ 
By standard results the distribution of the last passage time of Bessel processes is known, see e.g.  \cite{Getoor1979}, and it holds that $\gamma_r$ and $\beta_r$ are independent random variables with common density
\begin{equation}\label{densite-bord}
    t\longmapsto \mathbf{1}_{\R_+}(t)\,\frac{3^{-3/2}r^3}{\sqrt{2\pi}}\,t^{-5/2}\exp\left(-\frac{r^2}{6t}\right).
\end{equation}
Our goal now is to view $B_r^{\bullet}(\mathfrak{H})$ as a compact measured metric space equipped with a curve parametrizing $(\partial B_r^\bullet(\mathfrak{H}))\cup(\partial_\circ \mathfrak{H}\cap B_r^\bullet(\mathfrak{H}))$. To this end, we recall from  \cite[Section~3.2]{LGR6}, that the intrinsic distance on $B_r^\bullet(\mathfrak{H})\setminus\partial B_r^\bullet(\mathfrak{H})$ has a unique continuous extension to $B_r^\bullet(\mathfrak{H})$ that we denote by $D_r^\bullet$. We also let $\mathrm{Vol}_{r}^\bullet$ be the restriction of the volume measure to $B_r^\bullet(\mathfrak{H})$. It remains to define the curve parametrizing $$(\partial B_r^\bullet(\mathfrak{H}))\cup(\partial_\circ \mathfrak{H}\cap B_r^\bullet(\mathfrak{H})).$$ In this direction, for $i\in I$ with $t_i \notin [-\beta_r,\gamma_r]$, let $(L_s^r(\omega^i))_{s\geq 0}$ denote the exit local time at level $r$, and set $\alpha_i := \inf\{ s\in \mathbb{R} : \mathcal{E}^\infty_s \in \mathcal{T}_{\omega^i}\}$. Define
$$ L_s^{\mathfrak{H},r} := \sum_{t_i\notin [-\beta_r,\gamma_r]} L^r_{(s-\alpha_i)\vee 0}(\omega^i) \quad \text{ and }\quad  Z_r := L_\infty^{\mathfrak{H},r} = \sum_{t_i\notin [-\beta_r,\gamma_r]} L_\infty^r (\omega^i).$$

The variable $Z_r$ has already been studied in the literature and admits an explicit distribution. Specifically, by \cite[Proposition~4]{LGR6}, $Z_r$ is independent of $(\beta_r,\gamma_r)$ and is exponentially distributed with mean $2r^2/3$, that is, with density
\begin{equation}\label{densite-bord-Z}
    t\longmapsto \mathbf{1}_{\R_+}(t)\,\frac{3}{2r^{2}}\exp\big(-\frac{3t}{2r^2} \big).
\end{equation} 
In particular, almost surely, $Z_r < \infty$, and $Z_r$ can be interpreted as the interior perimeter of $B_r^\bullet(\mathfrak{H})$ since we have:
\begin{equation}\label{Z:r:limi:Vol}
Z_r=\lim_{\varepsilon\downarrow 0}\varepsilon^{-2}\mathrm{Vol}_\infty\big(\{x\in\mathfrak{H}\setminus B_r^\bullet(\mathfrak{H}):D_\infty(x,\partial B_r^\bullet(\mathfrak{H}))\leq\varepsilon\}\big),
\end{equation}
where the limit holds almost surely. Identity \eqref{Z:r:limi:Vol} holds by the approximation of boundary measures  \cite[Theorem~9]{LG2022}, together with the
local coupling in \cite[proof of Theorem~1]{LGR6}. The importance of \eqref{Z:r:limi:Vol} is that it implies that  $Z_r$ can be read directly from $(\mathfrak{H},D_\infty,\mathrm{Vol}_\infty,\Lambda_\infty)$, without reference to the encoding by labeled trees, and is invariant under isometries of curve-decorated spaces. Proposition 4 in \cite{LGR6} also characterizes the law of $(Z_r,\beta_r,\gamma_r,\mathrm{Vol}_\infty(B_r^\bullet(\mathfrak{H})))$ through its Laplace transform. Namely, for every $\lambda,\nu_1,\nu_2,\mu\geq 0$, we have:
\begin{equation}\label{joint-Lap}
\mathbb{E}\Big[\exp\big(-\lambda Z_{r}-\nu_1 \beta_r-\nu_2 \gamma_r-\mu \mathrm{Vol}_\infty(B_r^\bullet(\mathfrak{H}))\big)\Big]=\frac{G_r(\mu,\nu_1)\cdot G_r(\mu,\nu_2)}{\frac{2}{3}\lambda\,r^2+\sqrt{2\mu}\,r^2\big(\coth((2\mu)^{\frac{1}{4}}r)^{2}-\frac{2}{3}\big)},
\end{equation}
where $G_r$ is the function:
\begin{equation}
\label{formu-G}
G_r(\mu,\nu):= r \exp\Big(-r\,\sqrt{\frac{2}{3}} \sqrt{\sqrt{2\mu}+\nu}\Big)\times\Big((2\mu)^{\frac{1}{4}}\coth((2\mu)^{\frac{1}{4}}r)+\sqrt{\frac{2}{3}}\sqrt{\sqrt{2\mu}+\nu}\,\Big).
\end{equation}
We mention that it is also possible to understand the whole process $(Z_r,\beta_r,\gamma_r,\mathrm{Vol}_\infty(B_r^\bullet(\mathfrak{H})))_{r\geq 0}$, and to connect it to other important objects such as horohull explorations see \cite{RR2026}.

\noindent For simplicity, we write
$$ Z_r = Z_r' + Z_r'', \quad Z_r' := \sum_{t_i< -\beta_r} L^r_\infty(\omega^i), \quad Z_r'' := \sum_{t_i> \gamma_r} L^r_\infty(\omega^i),$$
and introduce the function $\kappa(u)$, for $u\in[0,Z_r]$, defined by
$$\kappa(Z_r'-u) := \inf\{ s\in\mathbb{R} : L_s^{\mathfrak{H},r} \geq u\}, \quad 0<u\leq Z_r',$$
$$\kappa(Z_r'+u) := \inf\{ s\in\mathbb{R} : L_s^{\mathfrak{H},r} \geq Z_r - u\}, \quad 0\leq u< Z_r'',$$
and  $\kappa(Z_r) := \kappa(Z_r^-)$. Then the curve
$$\Lambda^{\bullet,r}_\infty(u) := \Pi_\infty( \mathcal{E}^\infty_{\kappa(u)}), \quad 0\leq u\leq Z_r,$$
is continuous and injective, and its range equals $\partial B_r^\bullet(\mathfrak{H})$, by the support property of exit local times and the special Markov property, we refer to \cite[Lemma~1]{LGR6} for more details. We set, for every $t\in [0,\gamma_r+\beta_r + Z_r]$,
$$\wh{\Lambda}^{\bullet,r}(t)=\left\{\begin{array}{ll}
\Lambda_\infty(-t)\quad&\hbox{if }0\leq t\leq \beta_r,\\
\Lambda^{\bullet,r}_\infty(t-\beta_r)&\hbox{if } \beta_r\leq t\leq \beta_r+Z_r,\\
\Lambda_\infty(\gamma_r+\beta_r + Z_r-t)&\hbox{if } \beta_r+Z_r\leq t\leq \gamma_r+\beta_r + Z_r.
\end{array}
\right.
$$
Remark that, by construction, $\wh{\Lambda}^{\bullet,r}$ is a simple loop starting and ending at $\mathbf{x}_0$ and with range $(\partial B_r^\bullet(\mathfrak{H}))\cup(\partial_\circ \mathfrak{H}\cap B_r^\bullet(\mathfrak{H}))$. The space
$$\big(B_r^\bullet(\mathfrak{H}),D_r^\bullet,\wh{\Lambda}^{\bullet,r}(0),\mathrm{Vol}_{r}^\bullet,\wh{\Lambda}^{\bullet,r}\big)$$
is a curve-decorated compact measure metric space. By construction, the vector  $(\beta_r,\gamma_r,Z_r)$ can be recovered from $\big(B_r^\bullet(\mathfrak{H}),D_r^\bullet,\wh{\Lambda}^{\bullet,r}(0),\mathrm{Vol}_{r}^\bullet,\wh{\Lambda}^{\bullet,r}\big)$, and even measurably from its equivalence class. Namely, recalling that
$$D_r^\bullet(\wh{\Lambda}^{\bullet,r}(0),\wh{\Lambda}^{\bullet,r}(t))=\sqrt{3}\,\mathcal{R}_{-t},\quad 0\leq t\leq\beta_r,$$
and
$$D_r^\bullet(\wh{\Lambda}^{\bullet,r}(0),\wh{\Lambda}^{\bullet,r}(t))=\sqrt{3}\,\mathcal{R}_{\gamma_r+\beta_r+Z_r-t},\quad \beta_r+Z_r\leq t\leq\gamma_r+\beta_r+Z_r,$$
we infer that $\beta_r$ is the first time $t\geq 0$
such that there exists $\ve>0$ verifying  
$$D_r^\bullet\big(\wh{\Lambda}^{\bullet,r}(0),\wh{\Lambda}^{\bullet,r}(t+s)\big)=r,\quad \text{ for all}~ s\in[0,\ve],$$
and an analogous representation holds for $\gamma_r$. Furthermore, $$Z_r=\inf\{t\geq 0:D_r^\bullet(\wh{\Lambda}^{\bullet,r}(0),\wh{\Lambda}^{\bullet,r}(\beta_r+t))\neq r\}.$$

Finally, removing the interior of the hull yields the subset $\mathfrak{H}_r := \mathfrak{H} \setminus (B_r^\bullet(\mathfrak{H}) \setminus \partial B_r^\bullet(\mathfrak{H}))$, which is homeomorphic to the closed half-plane. Its boundary curve is given by
\begin{equation}\label{bd-curve}
\Lambda^r_\infty(t) :=
\begin{cases}
\Lambda_\infty(t-\beta_r), & t \leq 0,\\
\Lambda^{\bullet,r}_\infty(t), & 0\leq t \leq Z_r,\\
\Lambda_\infty(\gamma_r + t - Z_r), & t \geq Z_r.
\end{cases}
\end{equation}
By \cite[Theorem~2]{LGR6}, the intrinsic distance on  $\mathfrak{H} \setminus B_r^\bullet(\mathfrak{H}) $has a unique continuous extension $D^r_\infty$ to  $\mathfrak{H}_r$. We also equip $\mathfrak{H}_r$  with the restriction of $\mathrm{Vol}_\infty$ to $\mathfrak{H}_r$, that we denote by  $\mathrm{Vol}_\infty^r$. According to \cite[Theorem~2]{LGR6}, the curve-decorated space $(\mathfrak{H}_r,D_\infty^r,\Lambda_\infty^r(0),\mathrm{Vol}_\infty^r,\Lambda_\infty^r)$ is again a Brownian half-plane in the sense that
\begin{equation}\label{eq:H:r}
\big[\mathfrak{H}_r,D_\infty^r,\Lambda_\infty^r(0),\mathrm{Vol}_\infty^r,\Lambda_\infty^r\big]\overset{(d)}{=}\big[\mathfrak{H},D_\infty,\Lambda_\infty(0),\mathrm{Vol}_\infty,\Lambda_\infty\big]. 
\end{equation}
Moreover, by \cite[Theorem~3]{LGR6}, the random variables
$$\big[B_r^\bullet(\mathfrak{H}),D_r^\bullet,\wh{\Lambda}^{\bullet,r}(0),\mathrm{Vol}_{r}^\bullet,\wh{\Lambda}^{\bullet,r}\big]\quad\text{and}\quad\big[\mathfrak{H}_r,D_\infty^r,\Lambda_\infty^r(0),\mathrm{Vol}_\infty^r,\Lambda_\infty^r\big]$$
are independent. We will refer to this result as \emph{the spatial Markov property of the Brownian half-plane}.

\subsection{Conditional versions of hulls}\label{sec:hull}
 The goal of this section is to give a regular version of the conditional law of the hull of radius $r>0$ with respect to $(\beta_r,Z_r,\gamma_r)$. In this direction, we fix $r>0$. Let $\mathcal{C}:=C(\mathbb{R}_+,\mathbb{R})$, endowed with the topology of uniform convergence on compact sets, and let $R$ be the canonical process on $\mathcal{C}$. We denote the law on $\mathcal{C}$ of a five-dimensional Bessel process started from $0$ by $\mathrm{P}^{(5)}_0$, and put
$$\tau_r:=\sup\{t\geq0:R_t\leq r\}.$$
Recall that the density of $\tau_r$ is given by \eqref{densite-bord}. For $a>0$, we write $\mathrm{P}^{(5)}_{0}(\cdot\mid\tau_r=a)$ for the law of the stopped path $(R_{t\wedge\tau_r})_{t\geq0}$ conditionally on $\tau_r=a$, we choose a conditional version such that $a\mapsto\mathrm{P}^{(5)}_{0}(\cdot\mid\tau_r=a)$ is continuous on $(0,\infty)$. Let $\mathsf{M}_p(\mathbb{R}_+\times\mathcal{S})$ be the space of point measures on $\mathbb{R}_+\times\mathcal{S}$, endowed with its usual $\sigma$-field, and denote the canonical process on $\mathcal{C}^2\times\mathsf{M}_p(\mathbb{R}_+\times\mathcal{S})^3$ by
$$(R^\ell,R^{\rm r},\mathcal{P}^\ell,\mathcal{P}^0,\mathcal{P}^{\rm r}).$$
For every $a,b,c>0$, let $\mathsf{P}_r^{a,b,c}$ be the probability measure on $\mathcal{C}^2\times\mathsf{M}_p(\mathbb{R}_+\times\mathcal{S})^3$ under which $R^\ell$ and $R^{\rm r}$ are independent with respective laws  $\mathrm{P}^{(5)}_{0}(\cdot |\tau_{r/\sqrt{3}}=a)$  and  $\mathrm{P}^{(5)}_{0}(\cdot |\tau_{r/\sqrt{3}}=c)$ , and, conditionally on $(R^\ell,R^{\rm r})$, the three point measures are independent Poisson point measures with respective intensities
\begin{equation*}
2\cdot \mathbf{1}_{[0,a]}(u)\,\dd u\, \N_{\sqrt{3} R^\ell_u}(\dd\omega\cap \{W_*(\omega)>0\}),
\quad
\mathbf{1}_{[0,b]}(u)\,\dd u\,\N_r(\dd\omega\cap \{W_*(\omega)>0\}),
\end{equation*}
and
$$ 2\cdot \mathbf{1}_{[0,c]}(u)\,\dd u\,\N_{\sqrt{3} R^{\rm r}_u}(\dd\omega\cap \{W_*(\omega)>0\}).$$ We write $ \mathsf{E}_{r}^{a,b,c}$ for the expectation with respect to  $\mathsf{P}_r^{a,b,c}$, and  we concatenate these three pieces as follows. Let $z:=a+b+c$ and, for $0\leq t\leq z$, set:
\begin{equation}\label{eq:conditional-label-function}
\mathrm{w}^{a,b,c}(t):=
\begin{cases}
\sqrt{3} R^\ell_t, & 0\leq t\leq a,\\
 r, & a\leq t\leq a+b,\\
\sqrt{3} R^{\rm r}_{a+b+c-t}, & a+b\leq t\leq a+b+c.
\end{cases}
\end{equation}
Finally, if
$$\mathcal{P}^\ell=\sum_{i\in I_\ell}\delta_{(u_i^\ell,\omega^{i,\ell})},\quad \mathcal{P}^0=\sum_{i\in I_0}\delta_{(u_i^0,\omega^{i,0})},\quad \mathcal{P}^{\rm r}=\sum_{i\in I_{\rm r}}\delta_{(u_i^{\rm r},\omega^{i,\rm r})},$$
we set
\begin{equation}\label{eq:conditional-point-measure}
\mathcal{P}^{a,b,c}:=
\sum_{i\in I_\ell}\delta_{(u_i^\ell,\omega^{i,\ell})}
+
\sum_{i\in I_0}\delta_{(a+u_i^0,\omega^{i,0})}
+
\sum_{i\in I_{\rm r}}\delta_{(a+b+c-u_i^{\rm r},\omega^{i,\rm r})}.
\end{equation}
For brevity, we write $\mathrm{w}:=\mathrm{w}^{a,b,c}$ and $\mathcal{P}:=\mathcal{P}^{a,b,c}$, under $ \mathsf{P}_{r}^{a,b,c}$. It is straightforward to verify from the construction that, under $ \mathsf{P}_{r}^{a,b,c}$,  $(\mathrm{w},\mathcal{P})$ is a coding pair in the sense of Section~\ref{Section:Coding:Pair}. In particular, under $ \mathsf{P}_{r}^{a,b,c}$, we can consider the associated decorated compact metric space $\Theta(\mathrm{w},\mathcal{P})$.

\begin{proposition}\label{prop:conditional-hull}
Fix $r>0$. For every nonnegative measurable function $G$ on $[0,\infty)^3$ and every nonnegative measurable function $F$ on $\mathbb{M}^{\mathrm{GHPU}}_\infty$, we have:
\begin{equation}\label{eq:conditional-hull-integral}
\mathbb{E}\Big[G(\beta_r,Z_r,\gamma_r)F\big(\big[B_r^\bullet(\mathfrak{H}),D_r^\bullet,\wh{\Lambda}^{\bullet,r}(0),\mathrm{Vol}_{r}^\bullet,\wh{\Lambda}^{\bullet,r}\big]\big)\Big]
=\mathbb{E}\Big[G(\beta_r,Z_r, \gamma_r)\mathsf{E}_{r}^{\beta_r,Z_r, \gamma_r}\Big[F\big(\big[\Theta(\mathrm{w},\mathcal{P})\big]\big)\Big]\Big].
\end{equation}
\end{proposition}
In words, under $\mathsf{P}_r^{a,b,c}$, $\Theta(\mathrm{w},\mathcal{P})$ is a conditional version of $\big[B_r^\bullet(\mathfrak{H}),D_r^\bullet,\wh{\Lambda}^{\bullet,r}(0),\mathrm{Vol}_{r}^\bullet,\wh{\Lambda}^{\bullet,r}\big]$ conditionally on $(\beta_r,Z_r,\gamma_r)=(a,b,c)$.
\begin{proof}
The proposition follows directly from the strategy of the proof of \cite[Theorem~23]{LGR4}, and for this reason we only give the main steps of the argument. We work under $\P$, and the general idea is to make \eqref{hull} precise in terms of a coding pair. To this end, fix $r>0$ and recall the notation $\mathcal{E}^\infty$ for the clockwise exploration of $\mathfrak{T}_\infty$. For every $i\in I$ with $t_i\notin[-\beta_r,\gamma_r]$, we write $[u_i,u_i+\sigma(\omega^i)]$ for the interval corresponding to $\omega^i$ in the time scale of $\mathcal{E}^\infty$, that is, $\mathcal{E}^\infty_s\in\mathcal{T}_{\omega^i}$ if and only if $s\in[u_i,u_i+\sigma(\omega^i)]$. Next, we introduce the process
$$\widetilde{L}_t:= \sum_{i\in I: t_i\notin[-\beta_r, \gamma_r]} L^r_{(t-u_i)\vee 0}(\omega^i),\quad t\in\mathbb{R}.$$
For every $i$ such that $t_i\notin[-\beta_r, \gamma_r]$, we
denote the excursions of $\omega^i$ outside $(r,\infty)$ by $(\omega^{i,k})_{k\in J_{i}}$
where $J_i$ is an appropriate indexing set (if the index $i$ is such that $W_*(\omega^i)\geq r$,
$J_i$ is the empty set). In the time scale of 
$\mathcal{E}^\infty$, the excursion $\omega^{i,k}$ corresponds to an interval
$(\alpha_{i,k}^{-},\alpha_{i,k}^{+})$ and we set $t_{i,k}=\widetilde{L}_{\alpha_{i,k}^{-}}=\widetilde{L}_{\alpha_{i,k}^{+}}$.
Finally, for $t\in[0, \beta_r+ Z_r+\gamma_r]$, we set
\begin{equation}\label{eq:conditional-label-function-prime}
\mathrm{w}^\prime(t):=
\begin{cases}
\sqrt{3}\,\mathcal{R}_{-t}, & 0\leq t\leq \beta_r,\\
r, & \beta_r\leq t\leq \beta_r+Z_r,\\
\sqrt{3}\,\mathcal{R}_{\beta_r+Z_r+\gamma_r-t}, & \beta_r+Z_r\leq t\leq \beta_r+Z_r+\gamma_r.
\end{cases}
\end{equation}
We also introduce the point measure
$$\mathcal{P}^\prime:=\sum_{i\in I:~t_i\in[-\beta_r,0]}\delta_{(-t_i,\omega^i)}+\sum_{i\in I:~t_i\notin[-\beta_r,\gamma_r]}\sum_{k\in J_i}\delta_{(\beta_r+t_{i,k},\omega^{i,k})}+\sum_{i\in I:~t_i\in[0,\gamma_r]}\delta_{(\beta_r+Z_r+\gamma_r-t_i,\omega^i)}.$$
By construction and the special Markov property, for every nonnegative measurable function $G$ on $[0,\infty)^3$ and every nonnegative measurable function $F$ on $C(\mathbb{R}_+,\mathbb{R}_+)\times\mathsf{M}_p(\mathbb{R}_+\times\mathcal{S})$, we have
$$\mathbb{E}\Big[G(\beta_r,Z_r,\gamma_r)F\big(\mathrm{w}^\prime, \mathcal{P}^\prime\big)\Big]
=\mathbb{E}\Big[G(\beta_r,Z_r, \gamma_r)\mathsf{E}_{r}^{\beta_r,Z_r, \gamma_r}\Big[F\big(\mathrm{w},\mathcal{P}\big)\Big]\Big], $$
where we extend $\mathrm{w}^\prime$ and $\mathrm{w}$ constantly outside this interval to obtain elements of  $C(\mathbb{R}_+,\mathbb{R}_+)$.
To conclude, it then suffices to establish that:
$$\big[B_r^\bullet(\mathfrak{H}),D_r^\bullet,\wh{\Lambda}^{\bullet,r}(0),\mathrm{Vol}_{r}^\bullet,\wh{\Lambda}^{\bullet,r}\big]=\big[\Theta(\mathrm{w}^\prime, \mathcal{P}^\prime)\big],$$
but this follows directly by applying, mutatis mutandis, the argument given at the end of the proof of \cite[Theorem~23]{LGR4}. We leave the details to the reader.
\end{proof}

\section{Peeling exploration of the Brownian half-plane}\label{sec:peeling}

The purpose of this section is to introduce the peeling process of the Brownian half-plane.  
Informally, we start with the Brownian half-plane $\mathfrak{H}$,  which has a  distinguished point $\mathbf{x}_0$ on $\partial_\circ \mathfrak{H}$, and we iterate the following operation:
\begin{itemize}
    \item First, we remove the hull of radius $r>0$ around $\mathbf{x}_0$, where $r$ is fixed.
    \item By the spatial Markov property of the Brownian half-plane, the closure of the remaining space $\mathfrak{H}^\prime$ is again distributed as a Brownian half-plane.
    \item Consequently, we may choose a new point $\mathbf{x}_1$ on $\partial_\circ \mathfrak{H}^\prime$ of this remaining half-plane and repeat the procedure.
\end{itemize}
By analogy with the peeling process on random planar maps \cite{Angel2003,CurienPeeling,CLG2017,Watabiki1995}, we call such an exploration a peeling exploration of the Brownian half-plane.
There are several ways to view this exploration. On the one hand, it can be seen as a sequence of boundary points $(\mathbf{x}_i)_{i\in\N}$ of $\mathfrak{H}$. On the other hand, one can focus on the increasing sequence of explored regions, given by the successive hulls. We will study these different aspects of the peeling process in the following sections.
\par 
We now formalize this procedure and introduce the quantities that we will study and we refer to Figure \ref{fig:peeling} for an illustration. To this end, we consider an auxiliary probability space $(\Omega^\prime, \mathcal{F}^\prime, P^\prime)$. Next we consider a measurable function: $$\mathcal{A}:\Big(\bigcup_{n\geq 1}(\mathbb{M}^{\mathrm{GHPU}}_\infty)^n\Big)\times \Omega^\prime\to [0,1].$$
Given an explored part, the function $\mathcal{A}$ will be used to choose the reference points from which to continue the exploration. Informally, $\mathcal{A}$ can be thought of as an exploration algorithm. In what follows, we work on the product space $\Omega\times \Omega^\prime$, equipped with the product $\sigma$–field and, for simplicity with a slight abuse of notation, we use the notation $\P$ for  the product measure $\P\otimes P^\prime$.

\medskip

\noindent\textbf{First step of the peeling.}
Fix $r>0$, which will be the radius of the hulls that are removed at each step, and $\omega^\prime\in \Omega^\prime$. We start from the Brownian half-plane
$$\boldsymbol{\mathfrak{H}}^{(0)}(r):=(\mathfrak{H},D_\infty,\Lambda_\infty(0),\mathrm{Vol}_\infty,\Lambda_\infty)$$
and set
$$\mathbf{x}_0(r):=\mathbf{x}_0=\Lambda_\infty(0).$$
We then define
$$K_{1}^\bullet(r):= B_r^{\bullet}(\mathfrak{H})\quad\text{and}\quad\mathbf{K}_1^\bullet(r):=\big(B_r^\bullet(\mathfrak{H}),D_r^\bullet,\wh{\Lambda}^{\bullet,r}(0),\mathrm{Vol}_{r}^\bullet,\wh{\Lambda}^{\bullet,r}\big),$$
so that $K_1^\bullet(r)$ is the hull of radius $r$ around $\mathbf{x}_0(r)$. Recall that, with the notation of Section~\ref{sec:bhp:markov}, the curve $[0,Z_r] \ni u \longmapsto \Lambda_\infty^{\bullet,r}(u)$ parametrizes the internal boundary $\partial B_r^\bullet(\mathfrak{H})$, while $[-\beta_r,\gamma_r]\ni t\longmapsto \Lambda_\infty(t)$ parametrizes the portion of the original boundary contained in the hull. We choose the next reference point  by setting
\begin{equation}\label{eq:algo}
    U_1(r):=\mathcal{A}\big([\mathbf{K}_1^\bullet(r)],\omega^\prime\big)\quad \text{ and }\quad\mathbf{x}_1(r):=\Lambda_\infty^{\bullet, r}\big(U_1(r) Z_r\big).
\end{equation}
Since $U_1(r)\in[0,1]$, we continue to explore on the internal boundary $\partial B_r^\bullet$. In order to proceed with the exploration, we now consider the unexplored part of the space. By the spatial Markov property and the invariance by shifting \eqref{invariance:shiflting}, outside a negligible set, the equivalence class of
$$\boldsymbol{\mathfrak{H}}^{(1)}(r):=(\mathfrak{H}_r,D_\infty^r,\Lambda_\infty^r(U_1(r)Z_r),\mathrm{Vol}_\infty^r,\Lambda_\infty^r(\cdot + U_1(r) Z_r))$$
is distributed as a curve-decorated Brownian half-plane, and it is independent of the explored hull.  For later use, we record the boundary lengths created and swallowed at this first step. We set
$$\big(X_1^{(\ell)}(r),X_1^{(\mathrm{r})}(r)\big) :=(\beta_r,\gamma_r)$$
and
$$\big(Y_1^{(\ell)}(r),Y_1^{(\mathrm{r})}(r)\big) :=\big(U_1(r) Z_r,(1-U_1(r)) Z_r\big).$$
Recall that we use the notation $\partial_\circ E$ for the set of points of $E$ with no neighborhood  homeomorphic to the open disk. Informally, $X_1^{(\ell)}(r)$ and $X_1^{(\mathrm{r})}(r)$ are the perimeters of the portions of $\partial_\circ\boldsymbol{\mathfrak{H}}^{(0)}(r)$  lying on the left and on the right of $\mathbf{x}_0(r)$ that are contained in $K_1^\bullet(r)$. Similarly, $Y_1^{(\ell)}(r)$ and $Y_1^{(\mathrm{r})}(r)$ are the lengths of $\partial_\circ\boldsymbol{\mathfrak{H}}^{(1)}(r)\setminus \partial_\circ\boldsymbol{\mathfrak{H}}^{(0)}(r)$ on the left and on the right of $\mathbf{x}_{1}(r)$ that are not part of the boundary of $\boldsymbol{\mathfrak{H}}^{(0)}(r)$, see Figure \ref{fig:hull}. We stress that the random variables $X_1^{(\ell)}(r),X_1^{(\mathrm{r})}(r)$ and $Z_r$ are independent, and that their respective laws are given by \eqref{densite-bord} and \eqref{densite-bord-Z}.

\medskip
\noindent\textbf{Iteration of the peeling.}
By the spatial Markov property, the equivalence class of $\boldsymbol{\mathfrak{H}}^{(1)}(r)$ is distributed as a curve-decorated Brownian half-plane and is independent of the explored hull. Moreover, all the quantities defined above (boundary lengths, etc.) depend only on the underlying metric measure space together with its boundary curve, and not on the particular encoding by labeled trees. Therefore we may iterate the previous construction.

Namely, we apply the same procedure to the space $\boldsymbol{\mathfrak{H}}^{(1)}(r)$  and obtain, with the obvious notation,
$$K_{2}^\bullet(r),\ \mathbf{K}_{2}^\bullet(r),\ U_2(r),\ \mathbf{x}_2(r),\ X_2^{(\ell)}(r),X_2^{(\mathrm{r})}(r),\ Y_2^{(\ell)}(r),Y_2^{(\mathrm{r})}(r),\ \boldsymbol{\mathfrak{H}}^{(2)}(r),$$
where 
$$U_2(r):=\mathcal{A}\big(([\mathbf{K}_{i}^\bullet(r)])_{i\leq 2}, \omega^\prime\big).$$
We then iterate, using at step $n$ the space $\boldsymbol{\mathfrak{H}}^{(n-1)}(r)$ and the variable 
$$U_n(r):=\mathcal{A}\big(([\mathbf{K}_{i}^\bullet(r)])_{i\leq n}, \omega^\prime\big).$$
In this way we construct, outside a negligible set, for every $n\geq 1$ the collection
$$K_{n}^\bullet(r),\ \mathbf{K}_{n}^\bullet(r),\ U_{n}(r),\ \mathbf{x}_n(r),\ X_n^{(\ell)}(r),X_n^{(\mathrm{r})}(r),\ Y_n^{(\ell)}(r),Y_n^{(\mathrm{r})}(r),\ \boldsymbol{\mathfrak{H}}^{(n)}(r).$$
For definiteness, we set  $\boldsymbol{\mathfrak{H}}^{(0)}:=\boldsymbol{\mathfrak{H}}$, $K_{0}^\bullet(r):=\{\mathbf{x}_0(r)\}$, $\mathbf{K}_{0}^\bullet(r):=(\{\mathbf{x}_0(r)\}, 0,\mathbf{x}_0(r), 0)$, and 
$$X_0^{(\ell)}(r)=X_0^{(\mathrm{r})}(r)=Y_0^{(\ell)}(r)=Y_0^{(\mathrm{r})}(r)=0.$$
We also define two processes $\big(\widetilde H_n^{(\ell)}(r),\widetilde H_n^{(\mathrm{r})}(r)\big)_{n\geq 0}$ by 
\begin{equation}\label{eq:processus:centré}
\widetilde{H}_n^{(\ell)}(r) := \sum_{i=1}^n\Big(\frac{r^2}{3}-X_i^{(\ell)}(r)\Big)\quad ;\quad 
\widetilde{H}_n^{(\mathrm{r})}(r) := \sum_{i=1}^n\Big(\frac{r^2}{3}-X_i^{(\mathrm{r})}(r)\Big).
\end{equation}
and set $$V_{n}(r):= \sum_{i=1}^n \mathrm{Vol}_\infty\big(K_{i}^\bullet(r)\big),\quad n\geq 0.$$
These processes capture the consecutive (centered) boundary lengths which are swallowed at each step of the peeling process, together with the volume of the discovered hulls. Observe that according to \eqref{densite-bord} and \eqref{densite-bord-Z}, the random variables $X_n^{(\ell)}(r),X_n^{(\mathrm{r})}(r)$ are heavy-tailed, while $Y_n^{(\ell)}(r),Y_n^{(\mathrm{r})}(r)$ have light tails. Hence, we expect the ``large'' elements of $(K_n^\bullet(r))_{n\in\N}$ to be associated to large jumps of $(\widetilde{H}_n^{(\ell)}(r))_{n\in\N}$ and $(\widetilde{H}_n^{(\mathrm{r})}(r))_{n\in\N}$. Similarly, using \eqref{joint-Lap}, one can check that the random variable $\mathrm{Vol}_\infty(K_n^\bullet(r))$ is also heavy tailed, and we expect that the big jumps of $(V_n(r))_{n\in\N}$ are closely related to those of $(\widetilde{H}_n^{(\ell)}(r))_{n\in\N}$ and $(\widetilde{H}_n^{(\mathrm{r})}(r))_{n\in\N}$. This discussion will become precise in Section~\ref{sec:boundary}.

For every $n\in\N$, we define the total boundary length of $K_n^\bullet(r)$ by
\begin{equation}\label{def:L:n}
L_n(r):=X_n^{(\ell)}(r)+X_n^{(\mathrm{r})}(r)+Y_n^{(\ell)}(r)+Y_n^{(\mathrm{r})}(r).
\end{equation}

\begin{lemma}\label{L:order}
Fix $r>0$. Almost surely, the random variables $(\exp(-r^3 n) L_n(r))_{n\geq1}$ are pairwise distinct and 
 \begin{equation}\label{eq:L:n:vers:0}
\exp(-r^3n)\cdot L_n(r)\longrightarrow 0,\quad\text{as }n\rightarrow\infty.
\end{equation}
\end{lemma}
\begin{proof}
By the spatial Markov property, the random variables $(L_n(r))_{n\geq 1}$ are independent. Moreover, each of them has a density since $Z_r$ has a density and is independent of $(X_1^{(\ell)}(r),X_1^{(\mathrm{r})}(r))$, by \eqref{densite-bord} and \eqref{densite-bord-Z}. This proves the first statement.
For the second statement, remark that, by \eqref{densite-bord} and \eqref{densite-bord-Z}, we have $\mathbb{E}[L_n(r)]=4r^{2}/3$, for $n\geq 1$. Thus, by Markov's inequality, we directly have $\sum_{n\geq 1} \mathbb{P}(L_n(r)>\delta\exp(r^3 n))<\infty$, for every $\delta>0$. The desired result follows by an application of the Borel--Cantelli lemma.
\end{proof}
For later purposes we consider the collection of hulls $\big(\mathbf{K}_n^\bullet(r)\big)_{n\geq 1}$ ordered by decreasing values of $e^{-r^3n}L_n(r)$. We denote this rearranged sequence by
\begin{equation}\label{eq:rearranged:seq}
    \mathcal{K}^\bullet(r)=\big(r^3 j(n),L_{j(n)}(r),\mathbf{K}_{j(n)}^\bullet(r)\big)_{n\geq 1}\quad \text{and}\quad [\mathcal{K}^\bullet(r)]=\big(r^3 j(n),L_{j(n)}(r),\big[\mathbf{K}_{j(n)}^\bullet(r)\big]\big)_{n\geq 1},
\end{equation}
where $(j(n),L_{j(n)}(r), K_{j(n)}^\bullet(r))_{n\geq 1}$ is the sequence $(n, L_n(r) , K_n^\bullet(r))_{n\geq 1}$ reordered in decreasing order of $\exp(-r^3n)L_n(r)$; remark that the latter is well defined by Lemma \ref{L:order}. We view $[\mathcal{K}^\bullet(r)]$ as a random variable with values in $$\big(\mathbb{R}_+\times\mathbb{R}_+\times\mathbb{M}_\infty^{\mathrm{GHPU}}\big)^\mathbb{N},$$ endowed with the product topology. Now remark that, even if the exploration depends on the precise form of the function $\mathcal{A}$, by construction and the spatial Markov property the law of the  random variable:
\begin{equation}\label{eq:definition:famille}
  \Big(
    \Big(\widetilde{H}_{\lfloor r^{-3} t\rfloor}^{(\ell)}(r),
        \widetilde{H}_{\lfloor r^{-3} t\rfloor}^{(\mathrm{r})}(r), V_{\lfloor r^{-3} t\rfloor}(r)\Big)_{t\geq 0}, [\mathcal{K}^\bullet(r)]
  \Big).
\end{equation}
does not depend on the particular choice of $\mathcal{A}$. Also remark that \eqref{eq:definition:famille} is compatible with the scaling property of the Brownian half-plane. 

The goal of Sections~\ref{sec:boundary} and \ref{sec:discovered:regions} is to determine the scaling limit in distribution of \eqref{eq:definition:famille} as $r\to0$. Next, we will focus our attention on a subclass of functions $\mathcal{A}$ where it is also possible to show that the sequence $(\mathbf{x}_n(r))_{n\geq 1}$ converges to a random function.  We also stress that related scaling limits of peeling explorations have been studied in the discrete setting, notably for the UIPQ in~\cite{CLG2017} and for percolation-decorated quadrangulations in~\cite{GM2021}.

\section{Scaling limit of \texorpdfstring{$(\widetilde{H}^{(\ell)},\widetilde{H}^{(\mathrm{r})},V)$}{}}
\label{sec:boundary}
The goal of this section is to describe the scaling limit of 
$$\big( \widetilde{H}_{\lfloor \varepsilon^{-3}t\rfloor}^{(\ell)}(\varepsilon), \widetilde{H}_{\lfloor \varepsilon^{-3}t\rfloor}^{(\mathrm{r})}(\varepsilon), V_{\lfloor \varepsilon^{-3}t\rfloor}(\varepsilon)\big)_{t\geq 0}. $$
This is essentially a standard consequence of \eqref{joint-Lap} and the spatial Markov property. Let us first introduce the candidate for the scaling limit. Up to enlarging the underlying probability space, we consider two independent totally asymmetric $3/2$-stable processes $$\big(\cS_t^{(\ell)},\cS_t^{(\mathrm{r})}\big)_{t\geq 0}$$ with no upward jumps and with Lévy measure
\begin{equation}\label{mesure:saut:Lévy}
\frac{1}{3\sqrt{6\pi}}\,|x|^{-5/2}\mathbf{1}_{\{x<0\}}\,\mathrm{d}x,
\end{equation}
started from $0$. Next, we consider $(t_i^{(\ell)})_{i\geq 1}$ and $(t_i^{(\mathrm{r})})_{i\geq 1}$, the jumping times of $\big(\cS_t^{(\ell)}\big)$ and $\big(\cS_t^{(\mathrm{r})}\big)$ (chosen in a measurable way with respect to $\big(\cS_t^{(\ell)}\big)$ and $\big(\cS_t^{(\mathrm{r})}\big)$), and $\xi_i^{(\ell)}$ and $\xi_i^{(\mathrm{r})}$, $i\geq 1$, i.i.d. (and independent of $\big(\cS_t^{(\ell)},\cS_t^{(\mathrm{r})}\big)$) with density
\begin{equation}\label{eq:density:disk}
\frac{\mathbf{1}_{u>0}}{\sqrt{2\pi u^5}} \cdot \exp\big(-\frac{1}{2u}\big).   
\end{equation}
Finally, we set
\begin{equation}\label{eq:def:volume}
    \cV^{(\ell)}_t=\sum_{t_i^{(\ell)}\leq t} (\Delta \cS_{t_i}^{(\ell)})^2\cdot \xi_i^{(\ell)}\quad \text{ ; }\quad \cV^{(\mathrm{r})}_t=\sum_{t_i^{(\mathrm{r})}\leq t} (\Delta \cS_{t_i}^{(\mathrm{r})})^2\cdot \xi_i^{(\mathrm{r})} \quad \text{ ; }\quad \cV_t=\cV^{(\ell)}_t+ \cV^{(\mathrm{r})}_t.
\end{equation}
Notice that the process $\big(\cS_t^{(\ell)},\ \cS_t^{(\mathrm{r})}, \cV_t\big)_{t\geq 0}$ is rcll. 
\begin{proposition}\label{prop:conv:processus}
We have the joint convergence
$$\big( \widetilde{H}_{\lfloor\varepsilon^{-3}t\rfloor}^{(\ell)}(\varepsilon),\ \widetilde{H}_{\lfloor \varepsilon^{-3}t\rfloor}^{(\mathrm{r})}(\varepsilon), V_{\lfloor \varepsilon^{-3}t\rfloor}(\varepsilon)\big)_{t\geq 0}\xrightarrow[\varepsilon\rightarrow 0]{(d)}\big(\cS_t^{(\ell)},\ \cS_t^{(\mathrm{r})}, \cV_t\big)_{t\geq 0}$$
in $\mathbb{D}(\mathbb{R}_+,\mathbb{R}^3)$ endowed with the local $J_1$ topology.
\end{proposition}
\begin{proof}
By the spatial Markov property and the scaling invariance of the Brownian half-plane, for every $\varepsilon>0$, we have:
$$\big( \widetilde{H}_{\lfloor \varepsilon^{-3}t\rfloor}^{(\ell)}(\varepsilon),\  \widetilde{H}_{\lfloor \varepsilon^{-3}t\rfloor}^{(\mathrm{r})}(\varepsilon), V_{\lfloor \varepsilon^{-3}t\rfloor}(\varepsilon)\big)_{t\geq 0}\overset{(d)}{=}\big(\varepsilon^{2}\cdot \widetilde{H}_{\lfloor \varepsilon^{-3}t \rfloor}^{(\ell)}(1),\ \varepsilon^{2}\cdot \widetilde{H}_{\lfloor \varepsilon^{-3} t\rfloor}^{(\mathrm{r})}(1), \varepsilon^{4}\cdot V_{\lfloor  \varepsilon^{-3} t\rfloor}(1)\big)_{t\geq 0} $$
and $$\Big(\widetilde{H}_n^{(\ell)}(1), \widetilde{H}_n^{(\mathrm{r})}(1), V_n(1)\Big)_{n\geq 1}$$
is a random walk.
Thus, by \cite[Theorem 16.14]{Kallenberg}, the proof of the proposition boils down to establishing that 
$$\Big(n^{-\frac{2}{3}}\cdot \widetilde{H}_n^{(\ell)}(1),n^{-\frac{2}{3}}\cdot  \widetilde{H}_n^{(\mathrm{r})}(1), n^{-\frac{4}{3}}\cdot V_n(1)\Big)$$ converges in distribution to $\big(\cS_1^{(\ell)},\ \cS_1^{(\mathrm{r})}, \cV_1\big)$, as $n\to \infty$. Therefore, it suffices to show that for $\lambda_1,\lambda_2,\mu >0$, that we fix for the rest of the proof, we have
$$\lim \limits_{n\to \infty} \E\Big[\exp\big(\lambda_1 n^{-\frac{2}{3}}\cdot \widetilde{H}_n^{(\ell)}(1) +\lambda_2 n^{-\frac{2}{3}}\cdot \widetilde{H}_n^{(\mathrm{r})}(1)- \mu n^{-\frac{4}{3}}\cdot  V_n(1)\big)\Big]=\E[\exp(\lambda_1 \cS^{(\ell)}_1+\lambda_2 \cS^{(\mathrm{r})}_1-\mu \cV_1)],$$
and all the quantities in the previous display are finite. In this direction, note that for  every integer $n>0$, the spatial Markov property entails that
\begin{align*}
&\E\Big[\exp\big(\lambda_1 n^{-\frac{2}{3}}\cdot \widetilde{H}_n^{(\ell)}(1) +\lambda_2 n^{-\frac{2}{3}}\cdot \widetilde{H}_n^{(\mathrm{r})}(1)- \mu n^{-\frac{4}{3}}\cdot  V_n(1)\big)\Big]\\
&=
\E\Big[\exp\big(-\lambda_1 n^{-\frac{2}{3}} \beta_1-\lambda_2 n^{-\frac{2}{3}}\gamma_1 -\mu n^{-\frac{4}{3}}\cdot  \mathrm{Vol}_\infty(B_1^\bullet)\big)\Big]^{n} \cdot \exp(\frac{\lambda_1+\lambda_2}{3} n^{\frac{1}{3}}),
\end{align*}
For $\mu>0$, by \eqref{joint-Lap}, the previous display is equal to
\begin{equation}
\Big(\frac{F_n(\mu,\lambda_1)\,F_n(\mu,\lambda_2)}{\sqrt{2\mu}\,\big(\coth((2\mu)^{\frac{1}{4}} n^{-\frac{1}{3}})^{2}-\frac{2}{3}\big)}\Big)^{n} \cdot \exp(\frac{\lambda_1+\lambda_2}{3}n^{\frac{1}{3}}),
\end{equation}
where $F_n$ is the function:
\begin{equation*}
F_n(\mu,\lambda):= \exp\Big(-\,\sqrt{\frac{2}{3}} \sqrt{\sqrt{2\mu}+\lambda} \cdot n^{-\frac{1}{3}}\Big)\cdot\Big((2\mu)^{\frac{1}{4}}\coth((2\mu)^{\frac{1}{4}}n^{-\frac{1}{3}})+\sqrt{\frac{2}{3}}\sqrt{\lambda+\sqrt{2\mu}}\,\Big).
\end{equation*}
 A straightforward computation yields that:
\begin{align*}
&\lim \limits_{n\to \infty}\E\Big[\exp\big(\lambda_1 n^{-\frac{2}{3}}\cdot \widetilde{H}_n^{(\ell)}(1)+ \lambda_2 n^{-\frac{2}{3}}\cdot \widetilde{H}_n^{(\mathrm{r})}(1)- \mu n^{-\frac{4}{3}} \cdot V_n(1)\big)\Big]\\
&=\exp\Big(\frac{(2/3)^{\frac{3}{2}}}{3}\Big((\lambda_1-\sqrt{\mu/2})\sqrt{\lambda_1+\sqrt{2\mu}}+(\lambda_2-\sqrt{\mu/2})\sqrt{\lambda_2+\sqrt{2\mu}}\,\Big)\Big).
\end{align*}
Therefore, to conclude, it suffices to show that $\E[\exp(\lambda_1 \cS^{(\ell)}_1+\lambda_2 \cS^{(\mathrm{r})}_1-\mu \cV_1)]$ equals the right-hand side of the previous display. Now remark that, by construction and independence of $\cS^{(\ell)}$ and $\cS^{(\mathrm{r})}$, we have
$$\E[\exp(\lambda_1 \cS^{(\ell)}_1+\lambda_2 \cS^{(\mathrm{r})}_1-\mu \cV_1)]= \E[\exp(\lambda_1 \cS^{(\ell)}_1-\mu \cV_1^{(\ell)})]\cdot  \E[\exp(\lambda_2 \cS^{(\mathrm{r})}_1-\mu \cV_1^{(\mathrm{r})})]$$
and hence, by symmetry, it remains to show that
\begin{equation}\label{eq:laplace:levy}
    \E[\exp(\lambda_1 \cS^{(\ell)}_1-\mu \cV_1^{(\ell)})]=\exp\Big(\frac{(2/3)^{\frac{3}{2}}}{3}(\lambda_1-\sqrt{\mu/2})\sqrt{\lambda_1+\sqrt{2\mu}}\,\Big),
\end{equation}
for $\lambda_1,\mu\geq 0$, that we fix for the rest of the proof. To this end, remark that by independence and  \eqref{eq:density:disk}, the left-hand side of \eqref{eq:laplace:levy} is equal to
$$\E\Big[\Big( \prod_{t_i^{(\ell)}\leq  1} (1+\sqrt{2\mu} |\Delta  \cS^{(\ell)}_{t_i^{(\ell)}}|) \exp(-\sqrt{2\mu} |\Delta  \cS^{(\ell)}_{t_i^{(\ell)}}|)\Big) \exp(\lambda_1  \cS^{(\ell)}_{1})\Big].$$
For simplicity, set $f(x)=-\sqrt{2\mu} x+\log(1+\sqrt{2\mu} x) $, and remark that the latter display can be rewritten in the form:
\begin{equation}\label{eq:fin:preuve:convergence:simple}
\E\Big[ \exp\Big(\lambda_1  \cS^{(\ell)}_{1}+ \sum_{t_i^{(\ell)}\leq 1} f\big(\big|\Delta \cS^{(\ell)}_{t_i^{(\ell)}}\big|\big) \Big)\Big]. 
\end{equation}
By the exponential formula for Poisson point processes, using the compensated representation of the stable process and the fact that
$f(x)=O(x^2)$ as $x\to 0$,  and the definition of the associated Lévy measure \eqref{mesure:saut:Lévy}, it follows that  the process
$(M_t)_{t\geq 0}$ defined by
$$M_{t}:=\exp\Big(\lambda_1  \cS^{(\ell)}_{t}+ \sum_{t_i^{(\ell)}\leq t} f(-\Delta \cS^{(\ell)}_{t_i^{(\ell)}}) -\frac{t}{3\sqrt{6\pi}}\int_{-\infty}^{0} \frac{\mathrm{d} x}{|x|^{\frac{5}{2}}} (\exp(\lambda_1 x+f(-x))-1-\lambda_1 x)\Big),\quad t\geq 0, $$
is a martingale (for the canonical filtration of $\cS^{(\ell)}$). Therefore \eqref{eq:fin:preuve:convergence:simple} is equal to
$$\exp\Big(\frac{1}{3\sqrt{6\pi}} \int_{-\infty}^{0} \frac{\mathrm{d} x}{|x|^{\frac{5}{2}}} (\exp((\lambda_1+\sqrt{2\mu}) x)(1-\sqrt{2\mu}x)-1-\lambda_1 x) \Big). $$
A standard computation gives that $\int_{-\infty}^{0} \mathrm{d} x ~|x|^{-5/2} (\exp((\lambda_1+\sqrt{2\mu}) x)(1-\sqrt{2\mu}x)-1-\lambda_1 x)$ equals
\begin{align*}
&\int_{-\infty}^{0} \frac{\mathrm{d} x}{|x|^{\frac{5}{2}}} \big(\exp((\lambda_1+\sqrt{2\mu}) x)-1-(\lambda_1+\sqrt{2\mu}) x\big)-\sqrt{2\mu}\int_{-\infty}^{0} \frac{\mathrm{d} x}{|x|^{\frac{3}{2}}} \big(1-\exp((\lambda_1+\sqrt{2\mu}) x)\big)\\
&=\frac{4\sqrt{\pi}}{3}\cdot \big(\lambda_1-\sqrt{\mu/2}\big)\cdot \sqrt{\lambda_1+\sqrt{2\mu}}
\end{align*}
and the desired result follows.
\end{proof}

\begin{lemma}
For every $t,\delta>0$, we have
$$\lim\limits_{\varepsilon\to 0}\P\Big(\max_{k\leq \lfloor\varepsilon^{-3}t\rfloor}\big(Y_{k}^{(\ell)}(\varepsilon)+Y_{k}^{(\mathrm{r})}(\varepsilon)\big)\geq \delta\Big)=0.$$
\end{lemma}

\begin{proof}
 Since each random variable $Y_k^{(\ell)}(\varepsilon)+Y_k^{(\mathrm{r})}(\varepsilon)$, $k\geq 1$, has the same law as $Z_\varepsilon$, a union bound directly gives:
$$\P\Big(\max_{k\leq \lfloor\varepsilon^{-3}t\rfloor}\big(Y_{k}^{(\ell)}(\varepsilon)+Y_{k}^{(\mathrm{r})}(\varepsilon)\big)\geq \delta\Big)\leq \varepsilon^{-3}t\,\P(Z_\varepsilon\geq\delta)=\varepsilon^{-3}t\exp\big(-\frac{3\delta}{2\varepsilon^2}\big).$$
The right-hand side converges to $0$ as $\varepsilon\to0$.
\end{proof}

\section{Convergence of the discovered regions towards Brownian disks}\label{sec:discovered:regions}
The purpose of this section is to study the random variables $(\mathcal{K}^\bullet(\varepsilon))$. More precisely, we will prove that as $\varepsilon$ goes to $0$, the sequences $(\mathcal{K}^\bullet(\varepsilon))_{\varepsilon>0}$ converge in distribution toward a family of Brownian disks. To do so, we will start by showing that, under some conditioning, the hull $B_\varepsilon^\bullet(\mathfrak{H})$ converges in distribution towards a Brownian disk as $\varepsilon\downarrow0$, and then extend this result to the whole family $(\mathcal{K}^\bullet(\varepsilon))_{\varepsilon>0}$.

\subsection{Preliminary results on Bessel processes}
This section is devoted to the proof of several results on Bessel processes, which will play a crucial role  to prove convergence results on hulls and Brownian disks.

We start with some notation and basic results on Bessel processes. Recall that we denote the canonical process on $\mathcal{C}(\R_+,\R)$ by $\big(R_t:t\geq 0\big)$. Since we will be dealing with several variants of Brownian motions and Bessel processes in the next sections, we will endow this space with various probability measures. Namely, for every $\alpha\in\R$ and $x\geq 0$, let $\mathrm{P}^{(\alpha)}_x$ be the probability measure under which the canonical process is a Bessel process of dimension $\alpha$ starting from $x$. For $\alpha<2$, the Bessel process of dimension $\alpha$ is understood to be absorbed at $0$ after hitting $0$. Also, for any $x,y\geq 0$ and $r>0$, let $\mathrm{P}^{(\alpha),r}_{x\to y}$ be the probability measure under which the canonical process is a Bessel bridge of duration $r$ between $x$ and $y$ of dimension $\alpha$. We use the notation $\mathrm{E}^{(\alpha)}_x$ and $\mathrm{E}^{(\alpha),r}_{x\to y}$ respectively for the expectations under these measures. For $r\geq 0$, we set
$$T_r=\inf\{t\geq 0:R_t=r\}\quad \text{ and }\quad \tau_r=\sup\{t\geq 0:R_t=r\}.$$
We shall only use the following standard relations between Bessel processes of dimensions $5$ and $-1$. First, we have the following standard absolute continuity relation:
\begin{equation}\label{eq:absolute:continuity:Bessel}
\mathrm{E}^{(5)}_x
\Big[F\big((R_t)_{t\in[0,s]}\big)\Big] =\mathrm{E}^{(-1)}_x\Big[ \Big(\frac{R_s}{x}\Big)^3F\big((R_t)_{t\in[0,s]}\big)\Big],
\end{equation}
for every  $s,x>0$ and every non-negative measurable function $F$ on $\mathcal{C}([0,s],\R_+)$. In particular, for every $0<x<y$ and $u>0$, we have
\begin{equation}\label{eq:ac:stopped:bessel}
\mathrm{P}^{(5)}_x\big(T_y\leq u\big)=\Big(\frac{y}{x}\Big)^3\cdot \mathrm{P}^{(-1)}_x\big(T_y\leq u\big).
\end{equation}
Second,  by  Nagasawa’s time-reversal
theorem, for every $y>0$, we have
\begin{equation}\label{eq:time:reversal:zero}
\big((R_{\tau_y-t})_{t\in[0,\tau_y]}:\text{ under }\mathrm{P}^{(5)}_0\big)\overset{(d)}{=}\big((R_t)_{t\in[0,T_0]}:\text{ under }\mathrm{P}^{(-1)}_y\big).
\end{equation}

Finally, we recall that the density of $T_0$ under $\mathrm{P}^{(-1)}_x$ is
\begin{equation}\label{eq:last:passage:Bessel}
\bar{p}_t(x,0)=\frac{x^3}{\sqrt{2\pi t^5}}\exp\Big(-\frac{x^2}{2t}\Big)\mathbf{1}_{\{t>0\}},
\end{equation}
see e.g. \cite[Equation 15]{GY2003}. By convention, we set $\bar{p}_t(x,0)=0$ for $t\leq 0$. Equivalently, by \eqref{eq:time:reversal:zero}, $\bar{p}_t(y,0)$ is also the density of $\tau_y$ under $\mathrm{P}^{(5)}_0$.

Our first result deals with the supremum of Bessel processes under some conditioning.

\begin{lemma}\label{lemme:tech:hull:disk}
Consider $z>0$ and two functions $f,g:\R_+\to\R_+^*$ satisfying
$$\lim_{\varepsilon\downarrow 0}f(\varepsilon)=z>0\quad\text{and}\quad g(\varepsilon)=O(\varepsilon^2)\quad\text{ as }\varepsilon\downarrow 0.$$
Then, for any $\delta>0$,
$$\limsup_{\varepsilon\downarrow 0}\mathrm{P}^{(-1)}_{\varepsilon}\big(\sup_{t\geq 0}R_t\geq \delta\big| T_0=g(\varepsilon)\big)=0,$$
and
$$\limsup_{\eta\downarrow 0}\limsup_{\varepsilon\downarrow 0}\mathrm{P}^{(-1)}_{\varepsilon}\big(\sup_{0\leq t\leq \eta}R_t\geq \delta\big| T_0=f(\varepsilon)\big)=0.$$
\end{lemma}

\begin{proof}
For $\varepsilon<\delta$, the event $\{\sup_{t\geq 0}R_t\geq \delta\}$ is equal to $\{T_\delta<\infty\}$. By the strong Markov property, we have
\begin{align*}
&\mathrm{P}^{(-1)}_{\varepsilon}\big(\sup_{t\geq 0}R_t\geq \delta\big| T_0=g(\varepsilon)\big) =
\mathrm{E}^{(-1)}_{\varepsilon}\Big[\mathbf{1}_{\{T_\delta<\infty\}}\frac{\bar{p}_{g(\varepsilon)-T_\delta}(\delta,0)}{\bar{p}_{g(\varepsilon)}(\varepsilon,0)}\Big].
\end{align*}
For fixed $x>0$, the function $t\mapsto \bar{p}_t(x,0)$ is increasing on $(0,x^2/5]$. Since $g(\varepsilon)=O(\varepsilon^2)$, we have $g(\varepsilon)\leq \delta^2/5$ for all sufficiently small $\varepsilon$. Hence, using \eqref{eq:last:passage:Bessel}, we obtain that for $\varepsilon>0$ small enough
$$\mathrm{P}^{(-1)}_{\varepsilon}\big(\sup_{t\geq 0}R_t\geq \delta\big| T_0=g(\varepsilon)\big)\leq
\frac{\bar{p}_{g(\varepsilon)}(\delta,0)}{\bar{p}_{g(\varepsilon)}(\varepsilon,0)}=\Big(\frac{\delta}{\varepsilon}\Big)^3\cdot\exp\Big(-\frac{\delta^2-\varepsilon^2}{2g(\varepsilon)}\Big).$$
Since $g(\varepsilon)=O(\varepsilon^2)$, the last quantity converges to $0$ as $\varepsilon\downarrow 0$. This proves the first assertion. For the second assertion, the strong Markov property similarly gives, for every $\eta>0$,
\begin{align*}
&\mathrm{P}^{(-1)}_{\varepsilon}\big(\sup_{0\leq t\leq \eta}R_t\geq \delta\big| T_0=f(\varepsilon)\big) =
\mathrm{E}^{(-1)}_{\varepsilon}\Big[\mathbf{1}_{\{T_\delta\leq \eta\}}\frac{\bar{p}_{f(\varepsilon)-T_\delta}(\delta,0)}
{\bar{p}_{f(\varepsilon)}(\varepsilon,0)}\Big].
\end{align*}
Using $f(\varepsilon)\to z>0$ and  \eqref{eq:last:passage:Bessel}, it is straightforward to verify that there exists $\varepsilon_0>0$ and a constant $C>0$, such that:
$$\mathrm{P}^{(-1)}_{\varepsilon}\big(\sup_{0\leq t\leq \eta}R_t\geq \delta\big| T_0=f(\varepsilon)\big)
\leq C\varepsilon^{-3}\mathrm{P}^{(-1)}_{\varepsilon}\big(T_\delta\leq \eta\big),$$
for every $\varepsilon\in(0,\varepsilon_0)$ and $\eta\in(0,z/2)$. To conclude it then suffices to show that:
$$\limsup_{\eta\downarrow 0}\limsup_{\varepsilon\downarrow 0}\varepsilon^{-3}\mathrm{P}^{(-1)}_{\varepsilon}\big(T_\delta\leq \eta\big)=0.$$
In this direction, remark that by scaling and \eqref{eq:ac:stopped:bessel}, we have
\begin{align*}
\varepsilon^{-3}
\mathrm{P}^{(-1)}_{\varepsilon}\big(T_\delta\leq \eta\big)=\varepsilon^{-3}\mathrm{P}^{(-1)}_{\varepsilon/\delta}\Big(T_1\leq \frac{\eta}{\delta^2}\Big) 
=\delta^{-3}\mathrm{P}^{(5)}_{\varepsilon/\delta}\Big(T_1\leq \frac{\eta}{\delta^2}\Big).
\end{align*}
Under $\mathrm{P}^{(5)}_0$, the process tends to $\infty$ as $t\to\infty$, and in particular, by the strong Markov property at $T_{\varepsilon/\delta}$, we have
$$\mathrm{P}^{(5)}_{\varepsilon/\delta}\Big(T_1\leq\frac{\eta}{\delta^2}\Big)=\mathrm{P}^{(5)}_0\Big(T_1-T_{\varepsilon/\delta}\leq\frac{\eta}{\delta^2}\Big).$$
By continuity, under $\mathrm{P}^{(5)}_0$, $T_{\varepsilon/\delta}$ goes to $0$ as $\varepsilon\downarrow 0$, which entails that:
$$\lim_{\varepsilon\downarrow0}\varepsilon^{-3}\mathrm{P}^{(-1)}_{\varepsilon}\big(T_\delta\leq \eta\big)=\delta^{-3}\mathrm{P}^{(5)}_0\Big(T_1\leq \frac{\eta}{\delta^2}\Big).$$
Finally, since $\mathrm{P}^{(5)}_0\big(T_1<u\big)\to 0$ as $u\to0$, it follows that
$$\limsup_{\eta\downarrow 0}\limsup_{\varepsilon\downarrow 0}\varepsilon^{-3}\mathrm{P}^{(-1)}_{\varepsilon}\big(T_\delta\leq \eta\big)=0.$$
This completes the proof.
\end{proof}

Next, we prove a convergence in total variation between a Bessel process conditioned on $\{\tau_\varepsilon=f(\varepsilon)\}$ and a Bessel bridge. This result will be used to couple both processes in Section~\ref{subsection:convergence}.

\begin{lemma}\label{lemme:convergence:total:variation}
Fix $z>0$ and consider a function $f:\mathbb{R}_+\to \mathbb{R}_+$ such that $f(\varepsilon)\to z$ as $\varepsilon\downarrow 0$. Then, for every $0<\delta<z$, we have
$$\Big((R_t)_{t\leq z-\delta}:\quad\mathrm{P}^{(5)}_0\big(\cdot\big|\tau_\varepsilon=f(\varepsilon)\big)\Big)\xrightarrow[\varepsilon\downarrow 0]{\mathrm{TV}}\Big((R_t)_{t\leq z-\delta}:\quad\mathrm{P}^{(5),z}_{0\to 0}\Big),$$
where the exponent $\mathrm{TV}$ means that the convergence takes place in total variation.
\end{lemma}
\begin{proof}
Set
$$a=z-\delta\quad\text{ and } \quad h(\varepsilon)=f(\varepsilon)-a,$$
and remark that $h(\varepsilon)\to\delta$ as $\varepsilon\downarrow 0$, and in particular $h(\varepsilon)>0$ for $\varepsilon$ small enough. Without loss of generality we may assume that $f(\varepsilon)>a$ for every $\varepsilon>0$. Let $F$ be a measurable function on $\mathcal{C}([0,a],\R_+)$ such that $0\leq F\leq 1$. By the Markov property, we have:
\begin{align*}
&\mathrm{E}^{(5)}_0\Big[F\big((R_t)_{t\leq a}\big)\Big| \tau_\varepsilon=f(\varepsilon)\Big] =\int_{\R_+} \mathrm{P}^{(5)}_0\big(R_a\in \mathrm{d}x \big| \tau_\varepsilon=f(\varepsilon)\big)\mathrm{E}^{(5),a}_{0\to x}\Big[F\big((R_t)_{t\leq a}\big)\Big].
\end{align*}
By \eqref{eq:time:reversal:zero}, conditionally on $\tau_\varepsilon=f(\varepsilon)$, the reversed path has the law of a Bessel process of dimension $-1$ started from $\varepsilon$ and conditioned on $T_0=f(\varepsilon)$. Since $R_a$ corresponds to time $h(\varepsilon)=f(\varepsilon)-a$ for the reversed path, by \eqref{eq:time:reversal:zero}, the Markov property gives
\begin{align*}
&\mathrm{P}^{(5)}_0\big(R_a\in \mathrm{d}x\big| \tau_\varepsilon=f(\varepsilon)\big) =\mathrm{P}^{(-1)}_{\varepsilon}\big(R_{h(\varepsilon)}\in \mathrm{d}x\big| T_0=f(\varepsilon)\big) =\frac{p^{(-1)}_{h(\varepsilon)}(\varepsilon,x)\bar{p}_a(x,0)}{\bar{p}_{f(\varepsilon)}(\varepsilon,0)} \mathrm{d}x,
\end{align*}
here $p_t^{(-1)}$ stands for the transition density on $(0,\infty)$ of the killed Bessel process of dimension $-1$, which  is given by
$$p_t^{(-1)}(r_0,r)=\frac{r}{t}\Big(\frac{r_0}{r}\Big)^{3/2}\exp\Big(-\frac{r^2+r_0^2}{2t}\Big)I_{3/2}\Big(\frac{rr_0}{t}\Big),\quad r_0,r,t>0,$$
where we use the standard notation
$$I_{3/2}(u)=\Big(\frac{u}{2}\Big)^{\frac{3}{2}}\sum_{k=0}^{\infty}\frac{(u^2/4)^k}{k!\Gamma(\frac{3}{2}+k+1)},\quad u\geq 0,$$
for the modified Bessel function of the first kind with index $3/2$, see \cite[Equations~(2.2)--(2.4)]{MatsumotoYor2005}.
 A direct computation gives
\begin{align}\label{eq:density:conditioned:last:passage}
&\mathrm{P}^{(5)}_0\big(R_a\in \mathrm{d}x\big| \tau_\varepsilon=f(\varepsilon)\big) \nonumber \\
&\quad =\frac{1}{2^{3/2}\Gamma(5/2)}x^4\Big(\frac{f(\varepsilon)}{a h(\varepsilon)}\Big)^{5/2}\exp\Big(-\frac{f(\varepsilon)}{2a h(\varepsilon)}x^2-\frac{a}{2h(\varepsilon)f(\varepsilon)}\varepsilon^2\Big)\widetilde I_{3/2}\Big(\frac{\varepsilon x}{h(\varepsilon)}\Big)\mathrm{d}x,
\end{align}
where to simplify notation we write:
$$\widetilde I_{3/2}(u)=\Gamma(5/2)\Big(\frac{u}{2}\Big)^{-3/2}I_{3/2}(u),\quad u>0,$$
with the convention $\widetilde I_{3/2}(0)=1$.

We now compute the one-dimensional marginal of the Bessel bridge. The transition density of a Bessel process of dimension $5$ is
$$p_t^{(5)}(0,x)=\frac{1}{2^{3/2}\Gamma(5/2)}\frac{x^4}{t^{5/2}}\exp\Big(-\frac{x^2}{2t}\Big),\quad x>0,$$
and, for $x,y>0$,
$$p_t^{(5)}(x,y)=\frac{1}{2^{3/2}\Gamma(5/2)}\frac{y^4}{t^{5/2}}\exp\Big(-\frac{x^2+y^2}{2t}\Big)\widetilde I_{3/2}\Big(\frac{xy}{t}\Big).$$
It follows that, for $x>0$,
\begin{align}\label{eq:density:R:pont}
\mathrm{P}^{(5),z}_{0\to 0}\big(R_a\in \mathrm{d}x\big)&=\lim_{y\downarrow 0} \frac{p_a^{(5)}(0,x)p_{\delta}^{(5)}(x,y)}{p_z^{(5)}(0,y)}\mathrm{d}x=\frac{1}{2^{3/2}\Gamma(5/2)}x^4\Big(\frac{z}{\delta a}\Big)^{5/2}\exp\Big(-\frac{z}{2\delta a}x^2\Big)\mathrm{d}x.
\end{align}
Combining \eqref{eq:density:conditioned:last:passage} and \eqref{eq:density:R:pont} with the Markov property, we obtain
$$\mathrm{E}^{(5)}_0\Big[F\big((R_t)_{t\leq a}\big)\Big| \tau_\varepsilon=f(\varepsilon)\Big]=\mathrm{E}^{(5),z}_{0\to 0}\Big[A_\varepsilon(R_a)F\big((R_t)_{t\leq a}\big)\Big],$$
where
$$A_\varepsilon(x)=\Big(\frac{\delta f(\varepsilon)}{zh(\varepsilon)}\Big)^{5/2}\exp\Big(-\frac{a}{2h(\varepsilon)f(\varepsilon)}\varepsilon^2+\frac{z h(\varepsilon)-\delta f(\varepsilon)}{2\delta a h(\varepsilon)}x^2\Big)\widetilde I_{3/2}\Big(\frac{\varepsilon x}{h(\varepsilon)}\Big).$$
Therefore, we derive that:
\begin{align*}
&\Big|\mathrm{E}^{(5)}_0\Big[F\big((R_t)_{t\leq a}\big)\Big| \tau_\varepsilon=f(\varepsilon)\Big]-\mathrm{E}^{(5),z}_{0\to 0}\Big[F\big((R_t)_{t\leq a}\big)\Big]\Big| \leq\mathrm{E}^{(5),z}_{0\to 0}\Big[\big|A_\varepsilon(R_a)-1\big|\Big].
\end{align*}
Taking the supremum over all measurable functions $F$ such that $0\leq F\leq 1$, it follows that to conclude it suffices to show that the right-hand side converges to $0$ as $\varepsilon \to 0$. To this end, note that, for every fixed $x\geq 0$, we have $A_\varepsilon(x)\to 1$ as $\varepsilon\downarrow 0$, and the series expansion of $\widetilde I_{3/2}$ gives
$$\widetilde I_{3/2}(u)\leq \exp(u^2),\quad u\geq 0.$$
Since
$$\frac{z h(\varepsilon)-\delta f(\varepsilon)}{2\delta a h(\varepsilon)}=\frac{f(\varepsilon)-z}{2\delta h(\varepsilon)}\longrightarrow 0, \quad \text{ as }\: \varepsilon \to 0,$$
we infer that there exists a constant $C>0$ such that, for $\varepsilon$ small enough, we have
$$A_\varepsilon(R_a)\leq C\cdot \exp\Big(\frac{z}{4\delta a}R_a^2\Big).$$
 By \eqref{eq:density:R:pont}, the random variable on the right-hand side is integrable under $\mathrm{P}^{(5),z}_{0\to 0}$. Therefore, by dominated convergence,
$$\mathrm{E}^{(5),z}_{0\to 0}\Big[\big|A_\varepsilon(R_a)-1\big|\Big]\longrightarrow 0.$$
This completes the proof of the lemma.
\end{proof}

\subsection{Convergence of a hull towards a Brownian disk}\label{subsection:convergence}
The goal of this section is to show that under suitable conditioning, a hull $B_\varepsilon^\bullet(\mathfrak{H})$ converges in distribution towards a Brownian disk as $\varepsilon \downarrow 0$. To do so, we will compare the two constructions of these objects presented  in Sections \ref{sec:bhp:markov} and \ref{sec:hull}, and couple the different processes involved so that they are equal with high probability.

\textit{For the rest of this section, fix three positive functions $\mathfrak{z},f,g$ from $\R_+^*$ to $\R_+^*$ and $z>0$ such that } 
\begin{equation}\label{eq:assumption:sequences}
\mathfrak{z}(\varepsilon)\rightarrow z,\quad \lim \limits_{\varepsilon\downarrow 0}f(\varepsilon)=0\quad\text{ and }\quad g(\varepsilon)=O(\varepsilon^2).
\end{equation}
Recall the notation of Sections \ref{Section:Coding:Pair} and \ref{sec:hull}. In particular, recall that under $\P_z$, 
$$(\D, D,  \Pi(0),\mathrm{Vol} ,(\Pi(t))_{t\in[0,z]})$$
is a curved decorated free Brownian disk with perimeter $z$, and that, under $\mathsf{P}^{\mathfrak{z}(\varepsilon),f(\varepsilon),g(\varepsilon)}_\varepsilon$ we write $(\mathrm{w},\cP)$ for $(\mathrm{w}^{\mathfrak{z}(\varepsilon),f(\varepsilon),g(\varepsilon)},\cP^{\mathfrak{z}(\varepsilon),f(\varepsilon),g(\varepsilon)})$. The rest of the section is devoted to the proof of the following theorem. 
\begin{theorem}\label{theo:convergence:hull:}
    We have the convergence 
$$\Big(\big[\Theta(\mathrm{w},\cP)\big]:\,\text{ under }\mathsf{P}^{\mathfrak{z}(\varepsilon),f(\varepsilon),g(\varepsilon)}_\varepsilon\Big) \xrightarrow[\varepsilon\rightarrow 0]{(d)}\Big(\big[(\D, D,  \Pi(0),\mathrm{Vol} ,(\Pi(t))_{t\in[0,z]})\big]:\,\text{ under }\P_z\Big)$$
in distribution on $\mathbb{M}_\infty^{\mathrm{GHPU}}$.
\end{theorem}
By Proposition \ref{prop:conditional-hull},  the previous theorem can be understood as the convergence of the hull of radius $\varepsilon$, conditionally on $(\beta_{\varepsilon},Z_{\varepsilon},\gamma_{\varepsilon})=(\mathfrak{z}(\varepsilon),f(\varepsilon),g(\varepsilon))$.\par
The idea of the proof is to couple large portions of the hull and the disk in such a way that they are isometric with high probability, and to show that the remaining parts are small enough. To do so, we introduce some notation. For definiteness, under $\mathsf{P}^{\mathfrak{z}(\varepsilon),f(\varepsilon),g(\varepsilon)}_\varepsilon$, we write $\mathcal{P}:=\sum_{i\in I^\prime}\delta_{t_i, \omega^{i}}$, and $\mathcal{T}_{i}:=\mathcal{T}_{\omega^{i}}$. Next,  for $\varepsilon,\eta>0$ and under $\mathsf{P}^{\mathfrak{z}(\varepsilon),f(\varepsilon),g(\varepsilon)}_\varepsilon$, we decompose $\Theta(\mathrm{w},\cP)$ into two parts $\Pi(A_1(\varepsilon, \eta))$ and  $\Pi(A_2(\varepsilon,\eta))$ where
    \begin{align*}
    A_1(\varepsilon,\eta)&=[0,(z-\eta)_+\wedge\mathfrak{z}(\varepsilon)]\cup\bigcup_{i\in I^\prime,t_i\in[0,(z-\eta)_+\wedge \mathfrak{z}(\varepsilon)]}\cT_i,\\
    A_2(\varepsilon,\eta)&=\cT\setminus \Big([0,(z-\eta)_+\wedge\mathfrak{z}(\varepsilon)]\cup\bigcup_{i\in I^\prime,t_i\in[0,(z-\eta)_+\wedge\mathfrak{z}(\varepsilon)]}\cT_i\Big).
    \end{align*}
 Similarly, and with the notation of Section \ref{Section:Coding:Pair}, working under $\P_z$, we also decompose $\D$ into two parts $\Pi(A_1(\eta))$ and $\Pi(A_2(\eta))$ as follows:
    \begin{align*}         
    A_1(\eta)&=[0,(z-\eta)_+]\cup\left(\bigcup_{i\in J,t_i\in[0,(z-\eta)_+]}\cT_i\right),\\        
    A_2(\eta)&=[(z-\eta)_+, z]\cup\left(\bigcup_{i\in J,t_i\in[(z-\eta)_{+},z]}\cT_i\right).
    \end{align*}
Our goal now is to control the diameter of $\Pi(A_2(\varepsilon,\eta))$ and $\Pi(A_2(\eta))$ under the respective measures. 
\begin{lemma}\label{diam:vers:0}
    For every $\delta>0$,  we have
    \begin{equation}\label{diam:0:1}
    \limsup_{\eta \to 0}\P_z\big(\mathrm{diam}(\Pi(A_2(\eta)))>\delta\big)=0
    \end{equation}
 and    
\begin{equation}\label{diam:vers:O:2}
  \limsup_{\eta \to 0}\limsup_{\varepsilon \to 0}\mathsf{P}^{\mathfrak{z}(\varepsilon),f(\varepsilon),g(\varepsilon)}_\varepsilon\Big(\mathrm{diam}(\Pi(A_2(\varepsilon,\eta)))>\delta)\Big)=0.   
\end{equation}
\end{lemma}

\begin{proof}
We fix $\delta >0$. Recall that, working either under $\mathsf{P}^{\mathfrak{z}(\varepsilon),f(\varepsilon),g(\varepsilon)}_\varepsilon$ or $\P_z$, the labels represent distances to $\Pi(0)$. Hence, for any $A\subset\mathfrak{T}$, it holds that 
\begin{equation}\label{eq:borne:diametre}
\{\Pi(A)\subset B_{\delta/2}(\Pi(0))\}=\{\forall x\in A:\,\ell_x\leq\delta/2\}\subset \{\mathrm{diam}(\Pi(A))\leq \delta\}.
\end{equation}
Therefore, to control the diameter of a subset, it suffices to control the labels on this subset.

First, to prove \eqref{diam:0:1}, under $\P_z$ note that the function $\mathfrak{L}:t\mapsto\ell_{\cE_t}$ from $[0,\Delta]$ to $\R$  , where $\Delta$ is the lifetime of the exploration process $\mathcal{E}$, is continuous. Moreover, we have $\mathfrak{L}(0)=\mathfrak{L}(\Delta)=0$. Since $\cE_\Delta=z$, it follows that
$$\lim_{\eta\to 0}\sup_{x\in A_2(\eta)}\ell_x=0,\quad \P_z\text{--a.s.}$$
and the result follows from \eqref{eq:borne:diametre}. Let us now proceed with the proof of \eqref{diam:vers:O:2}. In this direction,  for $\varepsilon,\eta>0$ and under $\mathsf{P}^{\mathfrak{z}(\varepsilon),f(\varepsilon),g(\varepsilon)}_\varepsilon$, we decompose $A_2(\varepsilon,\eta)$ into three parts:
    \begin{align*}
    C_1(\varepsilon,\eta)&=\Pi\Big([(z-\eta)_+\wedge \mathfrak{z}(\varepsilon), \mathfrak{z}(\varepsilon)]\cup\Big(\bigcup_{i\in I^\prime,t_i\in[(z-\eta)_+\wedge \mathfrak{z}(\varepsilon), \mathfrak{z}(\varepsilon)]}\cT_i\Big)\Big),\\
    C_2(\varepsilon)&=\Pi\Big([\mathfrak{z}(\varepsilon),\mathfrak{z}(\varepsilon)+f(\varepsilon)]\cup\bigcup_{i\in I^\prime,t_i\in[\mathfrak{z}(\varepsilon),\mathfrak{z}(\varepsilon)+f(\varepsilon)]}\cT_i\Big),\\
    C_3(\varepsilon)&=\Pi\Big([\mathfrak{z}(\varepsilon)+f(\varepsilon),\mathfrak{z}(\varepsilon)+f(\varepsilon)+g(\varepsilon)]\cup\Big(\bigcup_{i\in I^\prime,t_i\in [\mathfrak{z}(\varepsilon)+f(\varepsilon),\mathfrak{z}(\varepsilon)+f(\varepsilon)+g(\varepsilon)]}\cT_i\Big)\Big).
    \end{align*}
Let us now control the probability for each of these parts to be small, starting with $C_2(\varepsilon)$. In this case, using the special Markov property, we get 
\begin{align*}
    \mathsf{P}^{\mathfrak{z}(\varepsilon),f(\varepsilon),g(\varepsilon)}_\varepsilon\big(C_2(\varepsilon)\subset B_{\delta/2}(\Pi(0))\big)&=\mathsf{P}^{\mathfrak{z}(\varepsilon),f(\varepsilon),g(\varepsilon)}_\varepsilon\big(\forall i\in I'\cup J',\,\sup \widehat{W}_s^{(i)}\leq \delta/2\big)\\
    &=\mathsf{E}^{\mathfrak{z}(\varepsilon),f(\varepsilon),g(\varepsilon)}\big[\exp\big(-f(\varepsilon)\cdot \N_\varepsilon( \sup \widehat{W}> \delta/2;~ W_*>0)\big)\big]\\
    &\geq\exp\big(-f(\varepsilon)\cdot \N_0(\sup W>\delta/2-\varepsilon)\big).
\end{align*}
Since $f(\varepsilon)\to 0$, it follows from \eqref{eq:borne:diametre} that
$$\liminf_{\varepsilon\to 0} \mathsf{P}^{\mathfrak{z}(\varepsilon),f(\varepsilon),g(\varepsilon)}_\varepsilon\big(\mathrm{diam}(C_2(\varepsilon))<\delta\big)\geq   \liminf_{\varepsilon\to 0} \mathsf{P}^{\mathfrak{z}(\varepsilon),f(\varepsilon),g(\varepsilon)}_\varepsilon\big(C_2(\varepsilon)\subset B_{\delta/2}(\Pi(0))\big)=1.$$
The cases of $C_1(\varepsilon,\eta)$ and $C_{3}(\varepsilon)$ are similar, but now we need to control the variations of the labels along the spine, and for this we will use Lemma \ref{lemme:tech:hull:disk}. For $C_3(\varepsilon)$, by definition, we have 
\begin{align*}
       \mathsf{P}^{\mathfrak{z}(\varepsilon),f(\varepsilon),g(\varepsilon)}_\varepsilon\big(C_3(\varepsilon)\subset B_{\delta/2}(\Pi(0))\big)&=\mathsf{P}^{\mathfrak{z}(\varepsilon),f(\varepsilon),g(\varepsilon)}_\varepsilon\big(\forall x\in C_3(\varepsilon),\,\ell_x\leq \delta/2\big)
\end{align*}
which equals:
\begin{align*}
\mathsf{E}_0^{(5)}\Big[\exp\Big(-2\int_0^{\tau_{\varepsilon/\sqrt{3}}}&\N_{\sqrt{3} R_{t}}(\sup\widehat{W_s}\geq \delta/2,W_*\geq0)dt\Big)\,\Big|\,\tau_{\varepsilon/\sqrt{3}}=g(\varepsilon)\Big]\\
      &\geq\mathsf{E}_0^{(5)}\Big[\exp\Big(-2\int_0^{\tau_{\varepsilon/\sqrt{3}}}\N_{\sqrt{3}R_{t}}(\sup\widehat{W_s}\geq \delta/2)dt\Big)\,\Big|\,\tau_{\varepsilon/\sqrt{3}}=g(\varepsilon)\Big].
\end{align*}
Now remark, that the previous display is bounded above by:
$$\mathsf{E}_0^{(5)}\Big[\1_{\sup_{t\in[0,\tau_{\varepsilon/\sqrt{3}}]}\sqrt{3}R_{t}\leq \delta/3}\exp\Big(-3\int_0^{\tau_{\varepsilon/\sqrt{3}}}\frac{1}{\big(\sqrt{3} R_{t}-\delta/2\big)^2}dt\Big)\,\Big|\,\tau_{\varepsilon/\sqrt{3}}=g(\varepsilon)\Big]$$
Hence, we infer that:
$$\liminf_{\varepsilon\to 0} \mathsf{P}^{\mathfrak{z}(\varepsilon),f(\varepsilon),g(\varepsilon)}_\varepsilon(\mathrm{diam}(C_3(\varepsilon))<\delta)\geq   \liminf_{\varepsilon\to 0} \mathsf{P}_0^{(5)}(\sup_{t\in[0,\tau_{\varepsilon/\sqrt{3}}]}\sqrt{3} R_{t}\leq \delta/3\,\big|\,\tau_{\varepsilon/\sqrt{3}}=g(\varepsilon)).$$
By duality of Bessel processes \eqref{eq:time:reversal:zero}, we have
$$\mathsf{P}_0^{(5)}\big(\sup_{t\in[0,\tau_{\varepsilon/\sqrt{3}}]}\sqrt{3} R_{t}\leq \delta/3\,\big|\,\tau_{\varepsilon/\sqrt{3}}=g(\varepsilon)\big)=\mathsf{P}^{(-1)}_{\varepsilon/\sqrt{3}}\big(\sup R< \delta/(3\sqrt{3}) \, \big|\, T_0=g(\varepsilon) \big)$$
and an application of Lemma \ref{lemme:tech:hull:disk} gives
$$\liminf_{\varepsilon\to 0} \mathsf{P}^{\mathfrak{z}(\varepsilon),f(\varepsilon),g(\varepsilon)}_\varepsilon\big(\mathrm{diam}(C_3(\varepsilon))<\delta\big)\geq  1.$$    
It remains to control $C_1(\varepsilon,\eta)$, for which we argue similarly. In this direction, notice that:
$$ \mathsf{P}^{\mathfrak{z}(\varepsilon),f(\varepsilon),g(\varepsilon)}_\varepsilon(\mathrm{diam}(C_1(\varepsilon,\eta))<\delta)\geq \mathsf{P}^{\mathfrak{z}(\varepsilon),f(\varepsilon),g(\varepsilon)}_\varepsilon\big(\forall x\in C_1(\varepsilon,\eta),\,\ell_x\leq \delta/2\big),$$
which equals
\begin{align*}
&\mathsf{E}_0^{(5)}\Big[\exp\Big(-2\int_{(z -\eta)_+\wedge \tau_{\varepsilon/\sqrt{3}}}^{\tau_{\varepsilon/\sqrt{3}}}\N_{\sqrt{3} R_{t}}(\sup\widehat{W_s}\geq \delta/2,W_*\geq0)\mathrm{d}t\Big)\,\Big|\,\tau_{\varepsilon/\sqrt{3}}=\mathfrak{z}(\varepsilon)\Big]\\
&\geq \mathsf{E}_0^{(5)}\Big[\1_{\sup_{t\in[(z -\eta)_+\wedge \tau_{\varepsilon/\sqrt{3}}\ ,\ \tau_{\varepsilon/\sqrt{3}}]}\sqrt{3} R_{t}\leq \delta/3}\exp\Big(-3\int_{(z -\eta)_+\wedge \tau_{\varepsilon/\sqrt{3}}}^{\tau_{\varepsilon/\sqrt{3}}}\frac{\mathrm{d}t}{\big(\sqrt{3}R_{t}-\delta/2\big)^2}\Big)\,\Big|\,\tau_{\varepsilon/\sqrt{3}}=\mathfrak{z}(\varepsilon)\Big].
\end{align*}   
Thus, by \eqref{eq:borne:diametre}, it suffices then to establish that:
$$\limsup_{\eta\to 0}\limsup_{\varepsilon\to 0}\P_0^{(5)}\big(\sup_{t\in[(z -\eta)\wedge \tau_{\varepsilon/\sqrt{3}},\tau_{\varepsilon/\sqrt{3}}]}R_{t}> \delta/(3\sqrt{3})\:\big| \:\tau_{\varepsilon/\sqrt{3}}=\mathfrak{z}(\varepsilon)\big)=0. $$
But once again, this is a consequence of Lemma \ref{lemme:tech:hull:disk}, since by duality \eqref{eq:time:reversal:zero} and monotonicity:
$$ \mathsf{P}_0^{(5)}\big(\sup_{t\in[(z -\eta)\wedge \tau_{\varepsilon/\sqrt{3}},\tau_{\varepsilon/\sqrt{3}}]}R_{t}> \delta/(3\sqrt{3})\:\big| \:\tau_{\varepsilon/\sqrt{3}}=\mathfrak{z}(\varepsilon)\big)\leq \mathsf{P}^{(-1)}_{\varepsilon/\sqrt{3}}\big(\sup_{[0,\eta]} R> \delta/(3\sqrt{3})\, |\, T_0=\mathfrak{z}(\varepsilon)\big).$$
This concludes the proof.
\end{proof}
Next, we need to control the volumes.

\begin{lemma}\label{lem:volume:decoupage}
        For every $\delta>0$,  we have
\begin{equation}
\limsup_{\eta \to 0}\P_z\big(\mathrm{Vol}(\Pi(A_2(\eta)))>\delta\big)=0
\end{equation}        
and        
    \begin{equation}
\limsup_{\eta \to 0}\limsup_{\varepsilon \to 0}\mathsf{P}^{\mathfrak{z}(\varepsilon),f(\varepsilon),g(\varepsilon)}_\varepsilon\Big(\mathrm{Vol}(\Pi(A_2(\varepsilon,\eta)))>\delta)\Big)=0,
     \end{equation}
     where under $\mathsf{P}^{\mathfrak{z}(\varepsilon),f(\varepsilon),g(\varepsilon)}_\varepsilon$ we write $\mathrm{Vol}$ for the measure of $\Theta(\mathrm{w},\cP)$.
\end{lemma}
\begin{proof}
We only give the proof under $\P_z$, the other case follows exactly by the same method.
Observe that by removing the conditions on the Poisson point measures, one can see that the volume of $\Pi(A_2(\eta))$ (under $\P_z$) is stochastically dominated by the first hitting time of $-\eta$ by a standard Brownian motion $(\mathfrak{W}_t)_{t\geq 0}$ started from $0$. Using the scaling invariance of the Brownian motion, we get:
$$ \P_z(\mathrm{Vol}(\Pi(A_2(\eta))\geq \delta)\leq \P\big(\inf_{t\in[0,\delta]} \mathfrak{W}_t>-\eta\big)=\P\big(\inf_{t\in[0,\delta \eta^{-2}]} \mathfrak{W}_t>-1\big)$$
The last term  converges to $0$ as $\eta\downarrow0$ since $\liminf_{t\geq 0}\mathfrak{W}_t=-\infty$. This completes the proof.
\end{proof}

\begin{proof}[Proof of Theorem \ref{theo:convergence:hull:}]
First, remark that by Lemmas \ref{diam:vers:0} and \ref{lem:volume:decoupage}, for every $\delta>0$ with the notation of Section  \ref{Section:Coding:Pair}, we have:
\begin{align*}
\limsup_{\eta\to 0}\limsup_{\varepsilon\to 0}\mathsf{P}^{\mathfrak{z}(\varepsilon),f(\varepsilon),g(\varepsilon)}_\varepsilon\Bigg(\mathrm{d}_{\mathrm{GHPU}}\Big(
(&\Pi(A_1(\varepsilon,\eta)),D|_{\Pi(A_1(\varepsilon,\eta))},\Pi(0),\mathrm{Vol}|_{\Pi(A_1(\varepsilon,\eta))},(\Pi(t))_{0\leq t\leq \mathfrak{z}(\varepsilon)\wedge (z-\eta)_+})\\
&\Theta(\mathrm{w},\cP)\Big)>\delta\Bigg)=0    
\end{align*}
and
\begin{align*}
\lim_{\eta\to 0}\P_z\Bigg(\mathrm{d}_{\mathrm{GHPU}}\Big((&\Pi(A_1(\eta)),D|_{\Pi(A_1(\eta))},\Pi(0),\mathrm{Vol}_{\Pi(A_1(\eta))},(\Pi(t))_{0\leq t\leq  (z-\eta)_+}),\\
&(\D,D,\Pi(0),\mathrm{Vol},(\Pi(t))_{0\leq t\leq z})\Big)>\delta\Bigg)=0.    
\end{align*}
Hence, to obtain the desired convergence, it suffices to show that for every $\eta>0$ small enough, we have 
\begin{multline*}
    \left((\Pi(A_1(\varepsilon,\eta)),D|_{\Pi(A_1(\varepsilon,\eta))},\Pi(0),\mathrm{Vol}|_{\Pi(A_1(\varepsilon,\eta))},(\Pi(t))_{0\leq t\leq (z-\eta)_+}):\text{ under }\mathsf{P}^{\mathfrak{z}(\varepsilon),f(\varepsilon),g(\varepsilon)}_\varepsilon\right)\\
    \xrightarrow[\varepsilon\rightarrow 0]{(d)}\left(\Pi(A_1(\eta)),D|_{\Pi(A_1(\eta))},\Pi(0),\mathrm{Vol}_{\Pi(A_1(\eta))},(\Pi(t))_{0\leq t\leq (z-\eta)_+}) :\text{ under }\P_z\right).
\end{multline*}
To prove this convergence, we will rely on a coupling argument. In this direction, let $(\mathrm{w}_\varepsilon,\mathcal{P}_\varepsilon)$ and $(\mathrm{w},\mathcal{P})$ be  coding pairs with respective laws under $\mathsf{P}^{\mathfrak{z}(\varepsilon),f(\varepsilon),g(\varepsilon)}_\varepsilon$ and $\P_z$. Next we denote the associated trees, labels, (pre)pseudo-distances, volume measures,  and projections by the transparent notation 
$$\big(\mathcal{T}_\varepsilon,\ell^\varepsilon, D_\varepsilon^\circ, D_\varepsilon, \mathrm{Vol}_\varepsilon,\Pi_\varepsilon\big)\quad \text{and}\quad \big(\mathcal{T},\ell, D^\circ, D, \mathrm{Vol},\Pi\big), $$
respectively. Now remark that  by Lemma \ref{lemme:convergence:total:variation}, for every $\delta>0$ and for $\eta,\varepsilon>0$  with $\mathfrak{z}(\varepsilon)>z-\eta>0$, we can couple $\mathrm{w}_\varepsilon$ and $\mathrm{w}$ such that the event:
$$E_{\varepsilon,\eta}:=\big\{\mathrm{w}_\varepsilon(t)=\mathrm{w}(t):~\text{ for all}~t\in[0,z-\eta]\big\}$$
has probability at least $1-\delta$. For the remaining of the proof we always consider $\eta,\varepsilon>0$ satisfying $\mathfrak{z}(\varepsilon)>z-\eta>0$. Next remark that, under $E_{\varepsilon,\eta}$, the conditional densities of $\mathcal{P}_\varepsilon$ and $\mathcal{P}$ agree on $[0,z-\eta]\times \mathcal{S}$. Therefore, we may assume that this coupling also satisfies:  
$$\cP_\varepsilon|_{[0,z-\eta]\times\cS}=\cP|_{[0,z-\eta]\times\cS}.$$
on the event $E_{\varepsilon,\eta}$.  We use the same notation $A_1(\varepsilon,\eta)$ and $A_1(\eta)$ for the subsets obtained by applying the definitions above to the coding pairs $(\mathrm{w}_\varepsilon,\mathcal{P}_\varepsilon)$ and $(\mathrm{w},\mathcal{P})$  respectively. In particular, on the event $E_{\varepsilon,\eta}$, the canonical isometry $\phi:A_1(\varepsilon, \eta)\rightarrow A_1(\eta)$  preserves labels, volumes and intervals. Moreover, $\phi$ fixes $[0,z-\eta]$ pointwise. Therefore,   to conclude it suffices to establish that:
\begin{equation}\label{eq:D:varepsilon:coupling}
D_\varepsilon(a,b)= D(\phi(a),\phi(b)),\quad\text{for every $a,b\in A_1(\varepsilon, \eta)$,}
\end{equation}
since this will imply that the induced function $\Pi_\varepsilon(a)\to \Pi(\phi(a))$ is well defined and  is an isometry preserving the roots, the volume measures and the distinguished curves. Let us proceed with the proof of \eqref{eq:D:varepsilon:coupling}. In this direction, remark that for every $a,b\in A_1(\varepsilon,\eta)$ such that $[|a,b|]_{\mathcal{T}_\varepsilon}\subset A_1(\varepsilon,\eta)$, we have 
$$\inf_{u\in[|a,b|]_{\mathcal{T}_\varepsilon}}\ell_u\geq\inf_{u\in[|a,b|]_{\mathcal{T}_\varepsilon}}\ell_u=0.$$
If $[|a,b|]_{\mathcal{T}_\varepsilon}$ is not contained in $A_1(\varepsilon,\eta)$, then this interval meets a point with label $0$, and this part of the interval can be removed in the definition of $D^\circ$. Therefore, since $\phi$ preserves intervals and labels, we have:
$$D^\circ_\varepsilon(a,b)= D^\circ(\phi(a),\phi(b)) $$
for every $a,b\in A_1(\varepsilon, \eta)$. This leads us to define 
\begin{equation}\label{eq:new:distance}
    \widetilde{D}(a,b):=\inf_{\underset{a_0,...,a_p\in A_1(\eta)}{a_0=a,...,a_p=b}} \sum_{i=1}^p D^\circ(a_{i-1},a_i),
\end{equation}
for every $a,b\in A_1(\eta)$. Since $\phi$ preserves $D^\circ$, identity \eqref{eq:D:varepsilon:coupling} boils down to prove that, for every $a,b\in A_1(\varepsilon, \eta)$, we have
$$ \widetilde{D}(\phi(a),\phi(b))= D_\varepsilon(a,b) \quad \text{ and }\quad \widetilde{D}(\phi(a),\phi(b))=D(\phi(a),\phi(b)).$$

 We only give the details for the second equality since the first follows by the same argument, with a slight modification of notation. First, observe that since the infimum of \eqref{eq:new:distance} is taken over a subset of $\cT$, for every $a,b\in A_1(\eta)$, we have
 $$D(a,b)\leq\widetilde{D}(a,b).$$
 To prove the converse inequality, it suffices to show that for every $n\geq 1$ and $a_0,a_1, \dots, a_n\in \cT$ with $a_0, a_n\in A_1(\eta)$ we have:
$$\widetilde{D}(a_0,a_n)\leq \sum_{k=0}^{n-1}D^\circ(a_k,a_{k+1}).  $$
Fix such a sequence $a_0,a_1, \dots, a_n$. If all the $a_k$ are in $A_1(\eta)$ there is nothing to do. Otherwise let 
$$i:=\min \{k\geq 0:~a_k\notin A_1(\eta)\} \quad \text{ and } \quad j:=\max \{k\geq 0:~a_k\notin A_1(\eta)\}.$$
Let $\cE:[0,\Delta]\to \mathcal{T}$ be the exploration process of $\cT$ ( see Section \ref{Section:Coding:Pair}, and, to simplify notation, we write $\Delta$ for the lifetime of $\cE$). Fix $t_0,t_i,t_j\in[0,\Delta]$,  such that 
$$\cE_{t_0}=z-\eta,\quad\quad\cE_{t_i}=a_i,\quad\quad\cE_{t_j}=a_j.$$
Next, for every $t\in[0,t_0]$, set 
$$\kappa(t):=\inf\{s\geq t:\,\ell_{\cE_s}=\inf_{r\in[t,t_0]}\ell_{\cE_r}\}.$$
Set $a_i^*=\cE_{\kappa(t_i)}$ and $a_j^*=\cE_{\kappa(t_j)}$ and remark that $a_i^*,a_j^*\in A_1(\eta)$. Then, by construction and the definition of $D^\circ$, we have:
$$D^{\circ}(a_{i-1}, a_{i})= D^{\circ}(a_{i-1}, a_i^*)+D^\circ(a_i^*,a_{i})\quad \text{ and } \quad D^{\circ}(a_{j}, a_{j+1})= D^{\circ}(a_j, a_j^*)+D^\circ(a_j^*,a_{j+1}).$$
In particular, we get:
\begin{align*}
\sum_{k=i-1}^{j}D^\circ(a_k,a_{k+1})&=D^{\circ}(a_{i-1}, a_i^*)+D^\circ(a_i^*,a_{i})+\sum_{k=i}^{j-1}D^\circ(a_k,a_{k+1})+ D^{\circ}(a_j, a_j^*)+D^\circ(a_j^*,a_{j+1})\\
&\geq D^{\circ}(a_{i-1}, a_i^*)+D(a_i^*,a_{j}^*)+D^\circ(a_j^*,a_{j+1}).\end{align*}
Hence, we infer that 
$$ \sum_{k=0}^{n-1}D^\circ(a_k,a_{k+1})\geq \widetilde{D}(a_0,a_i^{*})+  D(a_i^*,a_{j}^*)+ \widetilde{D}(a_j^*,a_n). $$
Finally, by definition of $a_i^*$ and $a_j^*$ and by \eqref{inegalité:label}, we have
$$ D(a_i^*,a_{j}^*)=|\ell_{a_i^*}-\ell_{a_j^*}|=\widetilde{D}(a_i^*,a_{j}^*),$$
and consequently:
$$ \sum_{k=0}^{n-1}D^\circ(a_k,a_{k+1})\geq \widetilde{D}(a_0,a_i^{*})+  \widetilde{D}(a_i^*,a_{j}^*)+ \widetilde{D}(a_j^*,a_n)\geq \widetilde{D}(a_0,a_n), $$
as wanted.
\end{proof}

\subsection{Convergence of the sequences \texorpdfstring{$\mathcal{K}^\bullet(\varepsilon)$}{}}
In this section, we use Theorem \ref{theo:convergence:hull:} to obtain the scaling limit of the random variables $([\mathcal{K}^\bullet(\varepsilon)])_{\varepsilon>0}$, jointly with $\left(\widetilde{H}_{\lfloor \varepsilon^{-3}t\rfloor}^{(\ell)}(\varepsilon),\widetilde{H}_{\lfloor \varepsilon^{-3}t\rfloor}^{(\mathrm{r})}(\varepsilon),V_{\lfloor \varepsilon^{-3}t\rfloor}(\varepsilon)\right)_{t\geq0}$.

\begin{theorem}\label{th:conv:familledisque}
We have the joint convergence
$$\left(\left(\widetilde{H}_{\lfloor \varepsilon^{-3}t\rfloor}^{(\ell)}(\varepsilon),\widetilde{H}_{\lfloor \varepsilon^{-3}t\rfloor}^{(\mathrm{r})}(\varepsilon),V_{\lfloor\varepsilon^{-3}t\rfloor}(\varepsilon)\right)_{t\geq 0},[\mathcal{K}^\bullet(\varepsilon)]\right)\xrightarrow[\varepsilon\rightarrow0]{(d)}\left(\left(\cS_t^{(\ell)},\cS_t^{(\mathrm{r})},\cV_t\right)_{t\geq0},[\mathcal{K}^\bullet]\right),$$
in the product of $\mathbb{D}(\mathbb{R}_+,\mathbb{R}^3)$, endowed with the local $J_1$ topology, and $\big(\mathbb{R}_+^2\times\mathbb{M}_\infty^{\mathrm{GHPU}}\big)^\mathbb{N}$, endowed with the product topology, where:
\begin{itemize}[label=\textbullet]
    \item $(\cS^{(\ell)},\cS^{(\mathrm{r})},\cV)$ is defined as in Proposition \ref{prop:conv:processus};
    \item if $(t_i)_{i\geq1}$ denotes the union of the jumping times of $\cS^{(\ell)}$ and $\cS^{(\mathrm{r})}$ and
    $$\ell_i=-\Delta\cS_{t_i}^{(\ell)}-\Delta\cS_{t_i}^{(\mathrm{r})},$$
    then, conditionally on $(\cS^{(\ell)},\cS^{(\mathrm{r})})$, $[\mathcal{K}^\bullet]$ is a sequence $(t_i,\ell_i,[\D_i])_{i\geq1}$, where the spaces $[\D_i]$ are independent free Brownian disks with respective perimeters $\ell_i$, pointed at a uniform boundary point and equipped with their standard boundary curves. The sequence is rearranged in decreasing order of $(e^{-t_i}\ell_i)_{i\geq1}$;
    \item almost surely, for every $t\geq0$,
    $$\cV_t=\sum_{i:t_i\leq t}\mathrm{Vol}_i([\D_i]),$$
    where $\mathrm{Vol}([\D_i])$ stands for the total mass of the volume measure of any representative of $[\D_i]$.\footnote{ We stress that this quantity does not depend on the choice of  a particular representative, since the isometries defining the equivalence relation preserve the volume measure.}
\end{itemize}
\end{theorem}
\begin{proof}
The proof relies on Theorem \ref{theo:convergence:hull:}, together with Proposition \ref{prop:conv:processus}. First, we prove the result without the volume process. For every $\varepsilon>0$, write 
$$\mathcal{K}^\bullet(\varepsilon):=\big(\varepsilon^3j_\varepsilon(i), L_{j_\varepsilon(i)}(\varepsilon), \mathbf{K}_{j_\varepsilon(i)}^\bullet(\varepsilon) \big)_{i\geq 1}.$$
 Recall that, since the processes $\cS^{(\ell)}$ and $\cS^{(\mathrm{r})}$ are independent, almost surely they do not have simultaneous jumping times. Moreover, the random variables $\varepsilon^{-2}Y_j^{(\ell)}(\varepsilon)$ and $\varepsilon^{-2} Y_j^{(\mathrm{r})}(\varepsilon)$ have moments of all orders. Combining these observations with Proposition \ref{prop:conv:processus} and the standard convergence of the jumps for the $J_1$ topology yields
\begin{align*}
\Big(&\varepsilon^3j_\varepsilon(i),L_{j_\varepsilon(i)}(\varepsilon),
\max\big(X_{j_\varepsilon(i)}^{(\ell)}(\varepsilon),X_{j_\varepsilon(i)}^{(\mathrm{r})}(\varepsilon)\big),
\min\big(X_{j_\varepsilon(i)}^{(\ell)}(\varepsilon),X_{j_\varepsilon(i)}^{(\mathrm{r})}(\varepsilon)\big),
Y_{j_\varepsilon(i)}^{(\ell)}(\varepsilon)+Y_{j_\varepsilon(i)}^{(\mathrm{r})}(\varepsilon)\Big)_{1\leq i\leq k}\\
&\xrightarrow[\varepsilon\rightarrow0]{(d)}\big(t_i,\ell_i,\ell_i,0,0\big)_{1\leq i\leq k},
\end{align*}
for every $k\geq 1$. To justify the rearrangement over all times, we also use that, for every $\delta>0$,
$$\lim_{T\to\infty}\limsup_{\varepsilon\to 0}\mathbb{P}\left(\sup_{n\geq\lfloor\varepsilon^{-3}T\rfloor}e^{-\varepsilon^3n}L_n(\varepsilon)>\delta\right)=0,$$
as a consequence of the $3/2$-stable tail of $X_n^{(\ell)}(\varepsilon)$ and $X_n^{(\mathrm{r})}(\varepsilon)$, together with the moment bounds for $Y_n^{(\ell)}(\varepsilon)$ and $Y_n^{(\mathrm{r})}(\varepsilon)$. The corresponding estimate for the jumps of $\cS^{(\ell)}$ and $\cS^{(\mathrm{r})}$ follows directly from the intensity measure of the jumps \eqref{mesure:saut:Lévy}.

Now, by the spatial Markov property of the Brownian half-plane, conditionally on their boundary lengths, the elements of $[\mathcal{K}^\bullet(\varepsilon)]$ are independent hulls. Up to exchange the left and right sides, it is straightforward to verify that the assumptions of Theorem \ref{theo:convergence:hull:} hold, and therefore an application of Theorem \ref{theo:convergence:hull:} entails that:
$$\Big(\varepsilon^3j_\varepsilon(i),L_{j_\varepsilon(i)}(\varepsilon),\big[\mathbf{K}_{j_\varepsilon(i)}^\bullet(\varepsilon)\big]\Big)_{1\leq i\leq k}\xrightarrow[\varepsilon\rightarrow 0]{(d)}\big(t_i,\ell_i,[\D_i]\big)_{1\leq i\leq k},$$
for every $k\geq 1$.
The result without the volume process follows since $k$ is arbitrary and the space of sequences is endowed with the product topology. Let us now deduce the result with the volume process. Recall that the volume process $V(\varepsilon)$, for $\varepsilon>0$, is defined by 
$$V_{n}(\varepsilon):= \sum_{i=1}^n \mathrm{Vol}_\infty\big(K_{i}^\bullet(\varepsilon)\big),\quad n\geq 0,$$
with $V_0(\varepsilon):=0$. By Proposition \ref{prop:conv:processus} and the first part of the proof, it follows that the sequence
\begin{equation}\label{pre:tight}
    \left(\left(\widetilde{H}_{\lfloor \varepsilon^{-3}t\rfloor}^{(\ell)}(\varepsilon),\widetilde{H}_{\lfloor \varepsilon^{-3}t\rfloor}^{(\mathrm{r})}(\varepsilon), V_{\lfloor\varepsilon^{-3}t\rfloor}(\varepsilon)\right)_{t\geq0},[\mathcal{K}^\bullet(\varepsilon)]\right)_{\varepsilon>0}
\end{equation}
is asymptotically  tight in distribution as $\varepsilon\to 0$. Any adherence point can be obtained as the law of some variable of the form  $\big(\cS^{(\ell)},\cS^{(\mathrm{r})},\mathcal{V}^\prime,[\cK^\bullet]\big)$. To conclude it then suffices to establish that:
$$\mathcal{V}^\prime_t=\sum_{t_i\leq t}\mathrm{Vol}_i([\D_i])$$
for every $t\geq 0$. By Skorokhod representation theorem, and with a slight abuse of notation, we can suppose that the convergence
$$ \left(\left(\widetilde{H}_{\lfloor \varepsilon^{-3}t\rfloor}^{(\ell)}(\varepsilon),\widetilde{H}_{\lfloor \varepsilon^{-3}t\rfloor}^{(\mathrm{r})}(\varepsilon), V_{\lfloor \varepsilon^{-3}t\rfloor}(\varepsilon)\right)_{t\geq0},\big[\mathcal{K}^\bullet(\varepsilon)\big]\right)\xrightarrow[\varepsilon\rightarrow0]{}\left(\left(\cS_t^{(\ell)},\cS_t^{(\mathrm{r})},\cV^\prime_t\right)_{t\geq0},\big[\mathcal{K}^\bullet\big]\right)$$
 holds almost surely along a subsequence. In particular, for fixed $0\leq s< t$, the convergence of the volume and Fatou's lemma yields
$$\mathcal{V}^\prime_t\geq \liminf_{\varepsilon\downarrow0}\sum_{i=1}^{\lfloor \varepsilon^{-3}s \rfloor} \mathrm{Vol}_\infty\!\left( K_i^{\bullet}(\varepsilon) \right) \geq \sum_{i=1}^{\lfloor \varepsilon^{-3}s \rfloor}\liminf_{\varepsilon\downarrow0} \mathrm{Vol}_\infty\!\left( K_i^{\bullet}(\varepsilon) \right) =\sum_{t_i\leq s}\mathrm{Vol}_i([\D_i]),
$$
where to obtain the last equality we use that almost surely  $s$ is not a jumping time of $\cS^{(\ell)}$ or $\cS^{(\mathrm{r})}$. Hence, since the processes are rcll,  we deduce that $\mathcal{V}' \geq \mathcal{V}$ a.s. However, using Proposition~\ref{prop:conv:processus} and the explicit distribution of the volume of free Brownian disks  \cite[Section~1.5]{BM2017}, it follows that $\mathcal{V}'$ and $\mathcal{V}$ have the same distribution. This implies that  $\mathcal{V}' = \mathcal{V}$ and concludes the proof.
\end{proof}

\section{Nearest-neighbour peeling explorations}\label{sec:nearest}
In this section, we introduce a family of peeling explorations that we call \textit{nearest-neighbour peeling explorations}. The rest of the paper will be devoted to the study of these processes.

Fix $r>0$, which is the radius of the hulls we reveal. Recall from Section \ref{sec:peeling} that a peeling exploration of the Brownian half-plane depends on an algorithm $\cA$, see \eqref{eq:algo}, which is used to choose the successive points around which we reveal the hulls. To simplify notation, for $n\geq 0$, we write:
$$\mathfrak{H}^{(n)}(r):=\big(\mathfrak{H}^{(n)}(r),D_\infty^{(n,r)},\mathbf{x}_n(r), \mathrm{Vol}_\infty^{(n,r)},\Lambda_\infty^{(n,r)}\big).$$
From now on, we will focus our study on random algorithms which choose the next point on the exposed boundary of the last discovered hull, in a way which only depends on its boundary size and on $r$. Formally, this means that $\cA$ takes values in $[0,1]$. In this situation, the algorithm may just be replaced by a sequence of random variables $(U_n(r))_{n\in\N}$ on $[0,1]$, by letting $\mathbf{x}_{n+1}(r):=\Lambda^{(n,r)}_\infty(U_n(r)\cdot (Y_n^{(\ell)}(r)+Y_{n}^{\mathrm{r}}))$.
Since two consecutive distinguished points $(\mathbf{x}_n(r),\mathbf{x}_{n+1}(r))$, $n\geq 0$, are ``as close as possible'', we say that the peeling processes associated to these algorithms are \textit{nearest-neighbour}.

\begin{figure}[htbp!]
   \centering
   \includegraphics[width=15cm]{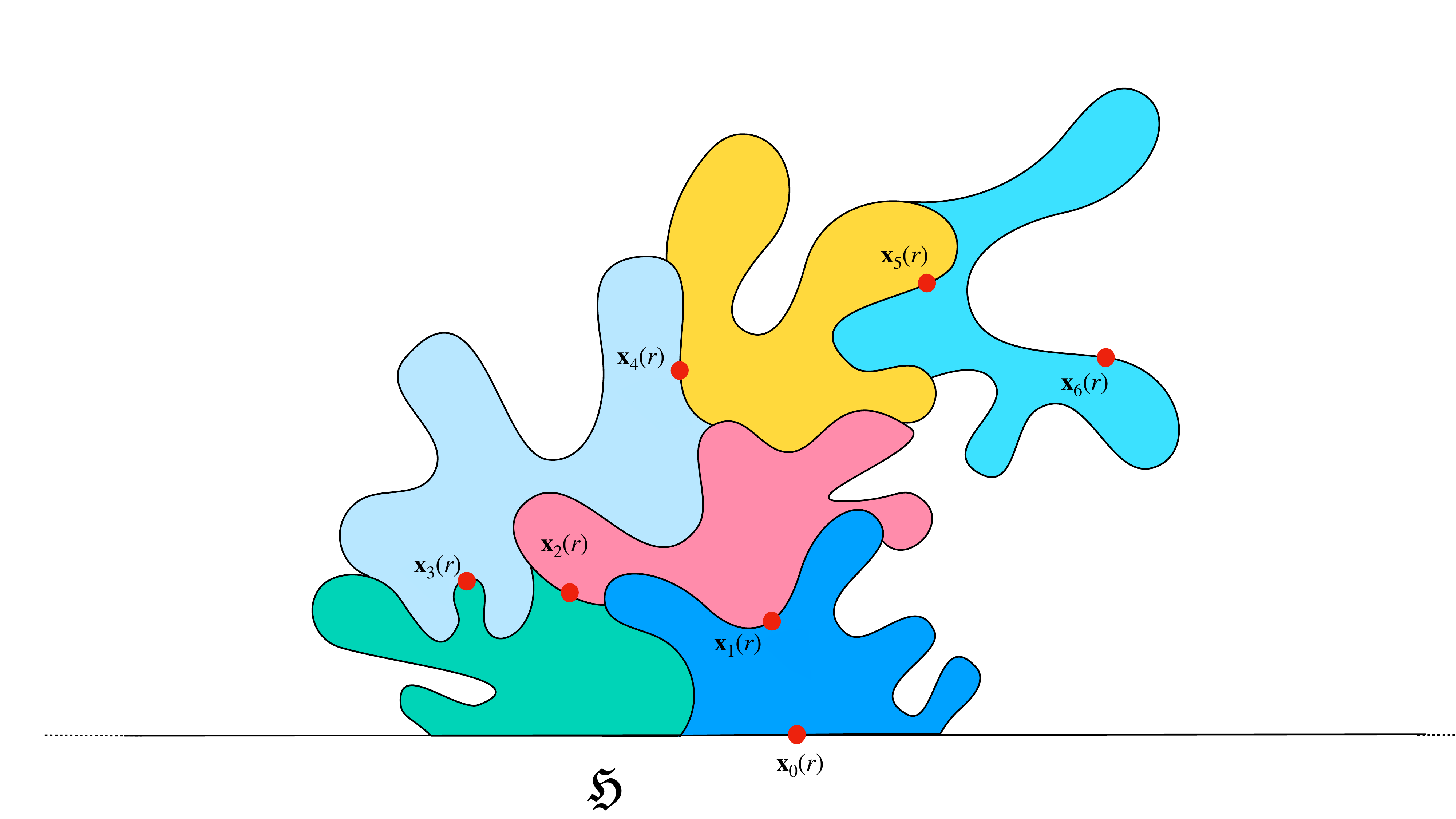}
   \caption{Illustration of the first six steps of a nearest-neighbour peeling exploration.}
\end{figure}

We can associate to any nearest-neighbour peeling algorithm a continuous curve and a pair of functions as follows. For every $n\geq 0$, let $\eta_{n}^{(r)}$ be a geodesic (for $D_\infty^{(n,r)}$) connecting $\mathbf{x}_{n}(r)$ and $\mathbf{x}_{n+1}(r)$. Since $\mathbf{x}_n(r)$ lies on the exposed boundary of $K_n^\bullet(r)$, $\eta_n^{(r)}$ can be seen as a continuous map
$$\eta_n^{(r)} : [0,r] \longrightarrow K_{n+1}^\bullet(r).$$
We now agglomerate these quantities to define processes, which will be the objects of our study. Specifically, we concatenate the geodesics $(\eta_n^{(r)})_{n\geq 1}$ to obtain a path $\Gamma^{(r)}:[0,\infty)\to \mathfrak{H}$, defined by 
$$\Gamma^{(r)}(t):=\eta_{\lfloor t\rfloor}^{(r)}\big(r\cdot (t-\lfloor t\rfloor)\big),\quad t\geq 0.$$
In particular we have $\Gamma^{(r)}(n)=\mathbf{x}_n(r)$, for every $n\geq 0$.  See Figure \ref{Fig:W:M} for an illustration. 

 The goal of the following sections is to study the limits in distribution of:
 \begin{equation}\label{goal:section:7}
\Big[\mathfrak{H},D_\infty, \mathrm{Vol}_\infty, \Lambda_\infty, \big(\Gamma^{(r)}(r^{-(3\wedge(2+\vartheta))} t)\big)_{t\geq 0}\Big]
\end{equation}
as $r\to 0$. Let us mention that we choose to interpolate between the points $\mathbf{x}_n(r)$ using geodesics since geodesics give the simplest control on distances, however the methods that we  implement should work for other interpolations (for example using the parametrization of the boundaries). 
\\
\\
In order to study \eqref{goal:section:7}, for every $n\in\N$, we define the \textit{horodistance processes} by $H_0^{(\ell)}(r)=H_0^{(\mathrm{r})}(r)=0$ and 
\begin{equation}
    H_n^{(\ell)}(r)=\sum_{i=1}^n \big(Y_i^{(\ell)}(r)-X_i^{(\ell)}(r)\big)\quad\text{and}\quad H_n^{(\mathrm{r})}(r)=\sum_{i=1}^n \big(Y_i^{(\mathrm{r})}(r)-X_i^{(\mathrm{r})}(r)\big).
\end{equation}
These processes have already been considered in \cite{C2015}, and record the variation of the boundary on each side of the distinguished points $(\mathbf{x}_n(r))_{n\in\N}$. These processes turn out to contain several geometric informations on the peeling process. For instance, for $n\geq 0$, set
\begin{equation}\label{def:W}
W^{(\ell)}_n(r):=\inf\big\{t\geq 0:\Lambda_\infty(-t)\in\mathfrak{H}^{(n)}(r)\big\}\quad\text{ ; }\quad W^{(\mathrm{r})}_n(r):=\inf\big\{t\geq 0:\Lambda_\infty(t)\in\mathfrak{H}^{(n)}(r)\big\}    
\end{equation}
and
\begin{equation}\label{def:M}
M^{(\ell)}_n(r):=\inf\{t\geq 0:\Lambda_\infty^{(n,r)}(-t)\in\Lambda_\infty(\mathbb{R})\}\quad\text{ ; }\quad M^{(\mathrm{r})}_n(r):=\inf\{t\geq0:\Lambda_\infty^{(n,r)}(t)\in\Lambda_\infty(\mathbb{R})\},
\end{equation}
where we recall that $\partial_\circ \mathfrak{H}^{(n)}(r)$ is parameterized by the standard boundary curve
$\Lambda_\infty^{(n,r)}:\R\longrightarrow \partial_\circ \mathfrak{H}^{(n)}(r)$, obtained from the Brownian half-plane $\boldsymbol{\mathfrak{H}}^{(n)}(r)$. The quantities $W^{(\ell)}_n(r)$ and $W^{(\mathrm{r})}_n(r)$ record the portion of $\partial_\circ \mathfrak{H}$ that has been swallowed by the peeling process after $n$ steps.  Similarly, $M^{(\ell)}_n(r)$ and $M^{(\mathrm{r})}_n(r)$ correspond to left and right the perimeter of $\partial_\circ \boldsymbol{\mathfrak{H}}^{(n)}(r)\setminus \partial_\circ \boldsymbol{\mathfrak{H}}(r) $ with respect to $\mathbf{x}_n(r)$, see Figure \ref{Fig:W:M}. These two families of processes are closely related to $H^{(\ell)}(r)$ and $H^{(\mathrm{r})}(r)$, as explained in the following proposition.
\begin{figure}[htbp!]
   \centering
   \includegraphics[width=15cm,height=8cm]{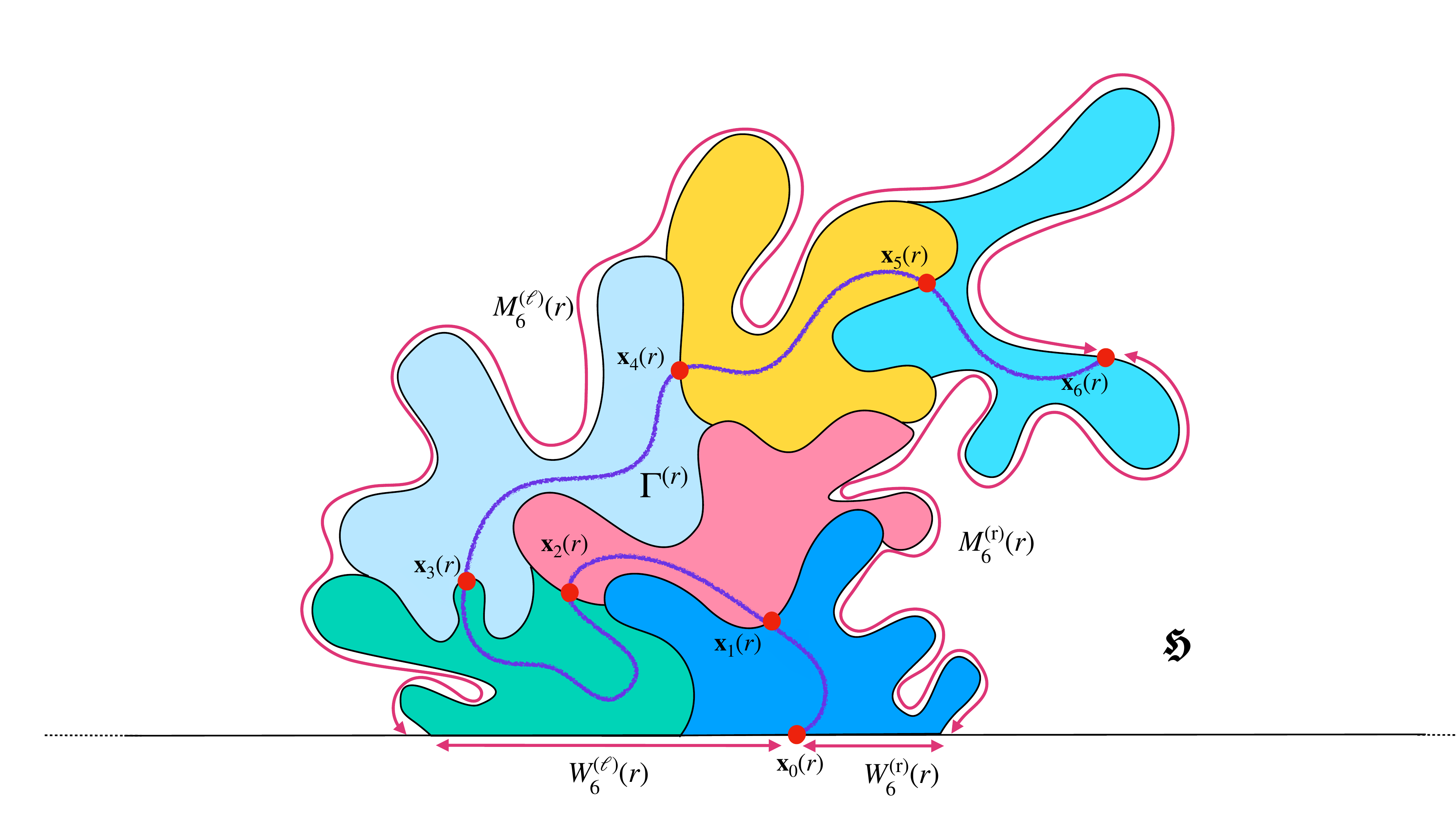}
   \caption{Illustration of the range of $\Gamma^{(r)}$ during the first six steps of the exploration (in blue), together with the variables $W^{(\ell)}_n(r)$, $W^{(\mathrm{r})}_n(r)$, $M^{(\ell)}_n(r)$, and $M^{(\mathrm{r})}_n(r)$ (in red), for $n=6$.}\label{Fig:W:M}
\end{figure}
\begin{lemma}\label{prop:formule:bord}
    For every $n\geq 0$, we have
$$W_n^{(\ell)}(r)=\sup_{j\in[\![0,n]\!]}\big(-H_j^{(\ell)}(r)+Y_j^{(\ell)}(r)\big)\quad\text{ and }\quad M_n^{(\ell)}(r)=H_n^{(\ell)}(r)+W_n^{(\ell)}(r),$$
    and the same holds with the exponent $(\ell)$ replaced by $(\mathrm{r})$.
\end{lemma}
\begin{proof}
To simplify notation we give the proof for $r=1$ for the exponent $(\ell)$. Exactly the same argument gives the result for general $r>0$ and for the exponent $(\mathrm{r})$. We proceed by induction. For $n=1$, there is nothing to show since the formulas directly give $W_1^{(\ell)}(1)=X_1^{(\ell)}(1)$ and $M_1^{(\ell)}(1)=Y_1^{(\ell)}(1)$.   Suppose that the result holds for the first $n$ steps of the peeling exploration.  By the spatial Markov property, observe that $W_{n+1}^{(\ell)}(1)>W_{n}^{(\ell)}(1)$ if and only  if $X_{n+1}^{(\ell)}(1)>M_{n}^{(\ell)}(1)$. Thus we always have 
$$W_{n+1}^{(\ell)}(1)=W_{n}^{(\ell)}(1)
+\big(X_{n+1}^{(\ell)}(1)-M_{n}^{(\ell)}(1)\big)\vee 0.$$    
Now remark that  the induction hypothesis yields 
$$ W_{n}^{(\ell)}(1)
+X_{n+1}^{(\ell)}(1)-M_{n}^{(\ell)}(1)=-H_{n}^{(\ell)}(1)+X_{n+1}^{(\ell)}(1)=-H_{n+1}^{(\ell)}(1)+Y_{n+1}^{(\ell)}(1).$$
We infer that:
\begin{equation}\label{eq:W:n:1:proof:induc}
    W_{n+1}^{(\ell)}(1)=\max\big(-H_{n+1}^{(\ell)}(1)+Y_{n+1}^{(\ell)}(1), W_{n}^{(\ell)}\big)= \sup_{j\in [\![0,n]\!]}\big(-H_j^{(\ell)}(1)+Y_j^{(\ell)}(1)\big),
\end{equation}
where in the last equality we used the induction hypothesis. Similarly, if $X_{n+1}^{(\ell)}(1)>M_{n}^{(\ell)}(1)$, then we must have $M_{n+1}^{(\ell)}(1)=Y_{n+1}^{(\ell)}(1)$ which coincides with 
$$H_{n+1}^{(\ell)}(1)+W_{n+1}^{(\ell)}(1)=H_{n+1}^{(\ell)}(1)-H_{n+1}^{(\ell)}(1)+Y_{n+1}^{(\ell)}(1)= Y_{n+1}^{(\ell)}(1),  $$
where to obtain the first equality we use again that if $X_{n+1}^{(\ell)}(1)>M_{n}^{(\ell)}(1)$ then $W_{n+1}^{(\ell)}(1)>W_{n}^{(\ell)}(1)$ combined with \eqref{eq:W:n:1:proof:induc}. If $X_{n+1}^{(\ell)}(1)\leq M_{n}^{(\ell)}(1)$,  the induction hypothesis gives
$$M_{n+1}^{(\ell)}(1)=M_{n}^{(\ell)}(1)+(Y_{n+1}^{(\ell)}(1)-X_{n+1}^{(\ell)}(1))=W_{n+1}^{(\ell)}(1)+H_{n+1}^{(\ell)}(1),$$
which completes the proof.
\end{proof}

From now on, we will consider nearest-neighbour peeling processes satisfying, for every $r\in(0,1)$ small enough, the following conditions:
\begin{itemize}
    \item  the sequence $U_k(r)$, $k\geq 1$, is i.i.d, and independent of $\boldsymbol{\mathfrak{H}}$;
    \item there exists $d\in\R$ and $\vartheta\geq 0$ 
\begin{equation}
   \E[U_1(r)]=1/2+d\cdot  r^\vartheta.
\end{equation}
If $d=0$, we take $\vartheta=1$.
\end{itemize}

Observe that if $\vartheta=0$ and since the support of $U_1(r)$ is included in $[0,1]$, we necessarily have $d\in[-1/2,1/2]$. As we will see, the behaviour of the peeling process heavily depends on the value of $\vartheta$, with a phase transition occurring at $\vartheta=1$. An intuitive understanding of this phase diagram is given by the following proposition, which compares the scaling limit of the process $(H^{(\ell)},H^{(\mathrm{r})})$ with the one of $(\widetilde{H}^{(\ell)},\widetilde{H}^{(\mathrm{r})})$ given in Theorem \ref{th:conv:familledisque}, where we  recall that $\widetilde{H}^{(\ell)},\,\widetilde{H}_n^{(\mathrm{r})}$ were defined in \eqref{eq:processus:centré}. We consider:
\begin{equation}\label{retour:du:vecteur}
    \left(\left(\cS_t^{(\ell)},\cS_t^{(\mathrm{r})},\cV_t\right)_{t\geq0},[\mathcal{K}^\bullet]\right)
\end{equation}
a random variable distributed as in the statement of Theorem \ref{th:conv:familledisque}.

\begin{proposition}\label{different:limite:beta:d}
For $\varepsilon >0$, set:
$$\gimel(\varepsilon):= 
\Big(\Big(H_{\lfloor \varepsilon^{-(3\wedge(2+\vartheta))}t\rfloor}^{(\ell)}(\varepsilon),H_{\lfloor\varepsilon^{-(3\wedge(2+\vartheta))}t\rfloor}^{(\mathrm{r})}(\varepsilon),\widetilde{H}_{\lfloor \varepsilon^{-(3\wedge(2+\vartheta))}t\rfloor}^{(\ell)}(\varepsilon),\widetilde{H}_{\lfloor\varepsilon^{-(3\wedge(2+\vartheta))}t\rfloor}^{(\mathrm{r})}(\varepsilon),V_{\lfloor\varepsilon^{-(3\wedge(2+\vartheta))}t\rfloor}(\varepsilon)\Big)_{t\geq 0}, [\mathcal{K}^\bullet(\varepsilon)]\Big).$$
We have the following convergences :
\begin{itemize}
    \item if $\vartheta\in(0,1)$ and $d\neq 0$,
    $$\gimel(\varepsilon)\xrightarrow[\varepsilon\downarrow 0]{(d)}\Big(\Big(\frac{2d}{3} t,-\frac{2d}{3} t,0,0,0\Big)_{t\geq 0},[\mathcal{K}^\bullet]\Big),$$
    \item if $\vartheta=1$ and $d\neq 0$, 
    $$\gimel(\varepsilon)\xrightarrow[\varepsilon\downarrow 0]{(d)}\Big(\Big(\cS_t^{(\ell)}+\frac{2d t}{3},\cS^{(\mathrm{r})}_t-\frac{2d t}{3},\cS_t^{(\ell)},\cS^{(\mathrm{r})}_t,\cV_t\Big)_{t\geq 0},[\mathcal{K}^\bullet]\Big),$$
      \item if $\vartheta>1$ or $d=0$,
    $$\gimel(\varepsilon)\xrightarrow[\varepsilon\downarrow 0]{(d)}\Big(\left(\cS_t^{(\ell)},\cS^{(\mathrm{r})}_t,\cS_t^{(\ell)},\cS^{(\mathrm{r})}_t,\cV_t\right)_{t\geq 0},[\mathcal{K}^\bullet]\Big).$$
\end{itemize}
\end{proposition}
\begin{proof}
To simplify notation we set $\vartheta^\prime=3\wedge(2+\vartheta)$. By Theorem \ref{th:conv:familledisque} and Slutsky lemma, it suffices to establish that:
$$\Big(H_{\lfloor \varepsilon^{-\vartheta^\prime}t\rfloor}^{(\ell)}(\varepsilon)-\widetilde{H}_{\lfloor \varepsilon^{-\vartheta^\prime}t\rfloor}^{(\ell)}(\varepsilon),H_{\lfloor\varepsilon^{-\vartheta^\prime}t\rfloor}^{(\mathrm{r})}(\varepsilon)-\widetilde{H}_{\lfloor\varepsilon^{-\vartheta^\prime}t\rfloor}^{(\mathrm{r})}(\varepsilon)\Big)_{t\geq 0}$$ 
converges in distribution, as $\varepsilon\to 0$, towards 
$$
\begin{cases}
\big(\frac{2d}{3}t, -\frac{2d}{3}t\big)_{t\geq 0}, & \text{if }~  \vartheta\in(0,1]~\text{ and }~d\neq 0\\
(0,0)_{t\geq 0}, & \text{if }~  \vartheta>1~\text{ and }~d=0.
\end{cases}
$$
In this direction,  recall that, for $\varepsilon>0$, we have:  
$$\left({Y}_n^{(\ell)}(\varepsilon),{Y}_n^{(\mathrm{r})}(\varepsilon)\right)_{n\in\N}\overset{(d)}{=}\Big({U}_n(\varepsilon)Z_n(\varepsilon),\big(1-U_n(\varepsilon)\big)Z_n(\varepsilon)\Big)_{n\in\N}~,$$ 
where the random variables $(Z_n(\varepsilon))_{n\in\N}$ are i.i.d exponential random variables with mean $2\varepsilon^2/3$, independent  of $(U_n(\varepsilon))_{n\geq 0}$. Now, let us  write:
\begin{align*}
\Big(H_{n}^{(\ell)}(\varepsilon)-\widetilde{H}_{n}^{(\ell)}(\varepsilon),H_{n}^{(\mathrm{r})}(\varepsilon)-\widetilde{H}_{n}^{(\mathrm{r})}(\varepsilon)\Big)_{n\in\mathbb{N}}\overset{(d)}{=}\Big(R_n^{(\ell)}(\varepsilon)+D_n^{(\ell)}(\varepsilon),R_n^{(\mathrm{r})}(\varepsilon)+D_n^{(\mathrm{r})}(\varepsilon)\Big)_{n\in \mathbb{N}}
\end{align*}
where 
\begin{align*}
    &\left(R_n^{(\ell)}(\varepsilon),R_n^{(\mathrm{r})}(\varepsilon)\right)_{n\in \mathbb{N}}=\Big(\sum_{i=1}^n\Big(({U}_i(\varepsilon)-d\varepsilon^\vartheta)Z_i(\varepsilon)-\frac{\varepsilon^2}{3}\Big),\sum_{i=1}^n\Big((1-U_i(\varepsilon)+d\varepsilon^\vartheta)Z_i(\varepsilon)-\frac{\varepsilon^2}{3}\Big)\Big)_{n\in \mathbb{N}}
\end{align*}
and
\begin{equation}
        \left(D_n^{(\ell)}(\varepsilon),D_n^{(\mathrm{r})}(\varepsilon)\right)_{n\in \mathbb{N}}=\left(\sum_{i=1}^n d\varepsilon^\vartheta Z_i(\varepsilon),\sum_{i=1}^n -d\varepsilon^\vartheta Z_i(\varepsilon)
        \right)_{n\in\N}.
\end{equation}
Let us find the scaling limit of these processes. Observe that by scaling, it follows that
$$  \left(D_{\lfloor \varepsilon^{-\vartheta^\prime}t\rfloor}^{(\ell)}(\varepsilon),D_{\lfloor \varepsilon^{-\vartheta^\prime}t\rfloor}^{(\mathrm{r})}(\varepsilon)\right)_{t\geq 0}\overset{(d)}{=}\left(\varepsilon^{2+\vartheta} D_{\lfloor \varepsilon^{-\vartheta^\prime}t\rfloor}^{(\ell)}(1),\varepsilon^{2+\vartheta} D_{\lfloor \varepsilon^{-\vartheta^\prime}t\rfloor}^{(\mathrm{r})}(1)\right)_{t\geq0}.$$
Since $(D_n(1))_{n\geq 1}$ is a random walk, by the law of large numbers and since $\E[Z_1(1)]=2/3$, we get that
$$  \left(D_{\lfloor \varepsilon^{-\vartheta^\prime}t\rfloor}^{(\ell)}(\varepsilon),D_{\lfloor \varepsilon^{-\vartheta^\prime}t\rfloor}^{(\mathrm{r})}(\varepsilon)\right)_{t\geq 0}\xrightarrow[\varepsilon\downarrow0]{(d)}\left(\frac{2d}{3}t,-\frac{2d}{3}t\right)_{t\geq0}.$$
Similarly, we have 
$$    \left(R_n^{(\ell)}(\varepsilon),R_n^{(\mathrm{r})}(\varepsilon)\right)_{n\in \mathbb{N}}\overset{(d)}{=}\left(\varepsilon^2\sum_{i=1}^n\Big( ({U}_i(\varepsilon)-d\varepsilon^\vartheta)Z_i(1)-\frac{1}{3}\Big),\varepsilon^2\sum_{i=1}^n\Big((1-U_i(\varepsilon)+d\varepsilon^\vartheta)Z_i(1)-\frac{1}{3}\Big)\right)_{n\in\N}.$$
Recall that for every $\varepsilon>0$, $\E[({U}_1(\varepsilon)-d\varepsilon^\vartheta)Z_1(1)]=1/3$. Hence, the previous display is just a centred random walk. Using invariance principles for triangular arrays, we get that:

$$\left(R_{\lfloor\varepsilon^{-\vartheta^\prime)}t\rfloor}^{(\ell)}(\varepsilon),R_{\lfloor\varepsilon^{-\vartheta^\prime}t\rfloor}^{(\mathrm{r})}(\varepsilon)\right)_{t\geq0}\xrightarrow[\varepsilon\downarrow0]{(d)}\left(0,0\right)_{t\geq0}.$$
The desired result now follows by Slutsky lemma.
\end{proof}

\section{Intermediate estimates of the discovered region and modulus of continuity}\label{sec:estimates}

\centerline{We fix $\vartheta\in \mathbb{R}_+$ and $d\in \mathbb{R}$, with $\vartheta=1$ if $d=0$, throughout the section.}
\vspace{0.5cm}
For $m \geq 1$ and $\varepsilon > 0$, we define the explored region after $m$ peeling steps by
\begin{equation}\label{eq:zone:explored}
    \mathcal{D}_m^\bullet(\varepsilon) := \bigcup_{n=0}^m K_n^\bullet(\varepsilon),
\end{equation}
where we use the notation $\mathcal{D}$ as a reference to \textit{the discovered region}. This defines an increasing sequence of compact subsets of $\mathfrak{H}$.  By construction,  $\mathcal{D}_m^\bullet(\varepsilon)$ has the topology of the closed unit disk in the complex plane and is therefore simply connected. The goal of this section is to control the diameter and volume of $\mathcal{D}_m^\bullet(\varepsilon)$. In this direction, remark that since the boundaries of hulls  have $0$ mass for the measure $\mathrm{Vol}_{\infty}$, we have
\begin{equation}
    \mathrm{Vol}_\infty\big(\mathcal{D}_m^\bullet(\varepsilon)\big)=V_{m}(\varepsilon).
\end{equation}
Therefore, to control the volume of the discovered region boils down to the study of a random walk. The diameter  is more delicate to analyze. The strategy is to introduce a notion of generalized perimeter that is simple to study using horodistances and that allows us to control metric distances. We stress that the strategy we follow is strongly inspired by \cite[Section 5]{GM2021}, however we will rely on more volume arguments in order to obtain more explicit controls (in particular these estimates will play a crucial role in Section \ref{sec:peeling:scaling}). To this end, we introduce some notation. The topological boundary of $\mathcal{D}_m^\bullet(r)$ is given by:
$$\partial \mathcal{D}_m^\bullet(r):=\mathcal{D}_m^\bullet(r) \cap \mathfrak{H}^{(m)}(r),$$
and, if as usual we write $\partial_\circ \mathcal{D}_m^\bullet(\varepsilon)$ for  the set  of all points with no neighborhood in $\mathcal{D}_m^\bullet(\varepsilon)$ homeomorphic to the open disk, we have:
$$\partial_\circ \mathcal{D}_m^\bullet(r)=\partial \mathcal{D}_m^\bullet(r)\cup\big(\mathcal{D}_m^\bullet(r)\cap\partial_\circ \mathfrak{H}\big).$$
Recalling the notation \eqref{def:W} and \eqref{def:M}, we set:
$$\big|\partial \mathcal{D}_m^\bullet(r)\big|=M_{m}^{(\ell)}(r)+M_{m}^{(\mathrm{r})}(r)\quad \text{ and } \big|\partial_\circ \mathcal{D}_m^\bullet(r)\big|=M_{m}^{(\ell)}(r)+M_{m}^{(\mathrm{r})}(r)+W_{m}^{(\ell)}(r)+W_{m}^{(\mathrm{r})}(r),$$
that we interpret as the generalized perimeter or length of $\partial \mathcal{D}_m^\bullet(r)$ and $\partial_\circ \mathcal{D}_m^\bullet(r)$ respectively,  see Figure \ref{Fig:W:M}.
\noindent Finally, for  every $\delta,\varepsilon>0$, we set:
$$T_\delta^{(\varepsilon)}:=\inf\big\{n>0:~X_n^{(\ell)}(\varepsilon)\vee X_n^{(\mathrm{r})}(\varepsilon)>\delta ~\text{ or }~ \mathrm{Vol}(K_n^{\bullet})>\delta^2\big\}\wedge 
\lfloor \varepsilon^{-(3\wedge (2+\vartheta))}\cdot \delta^{\frac{3}{2}} \rfloor.$$
Formally, our goal is to control the probability to be large for the variables:
$$\mathrm{Vol}_\infty\big(\mathcal{D}_{T_\delta^{(\varepsilon)}-1}^{\bullet}\big)\quad  ;\quad \big|\partial_\circ \mathcal{D}_{T_\delta^{(\varepsilon)}-1}^{\bullet} \big|  \quad ;\quad \mathrm{diam}\big(\partial_\circ \mathcal{D}_{T_\delta^{(\varepsilon)}-1}^{\bullet}\big),$$
uniformly in $\varepsilon$ and $\delta$, where here and in all the section we write $\mathrm{diam}_{\mathfrak{H}}(E)$ for the diameter in $(\mathfrak{H},D_\infty)$ for a subset $E$ of $\mathfrak{H}$. We start in Section \ref{sub:uni:contr:1} by controlling the first two variables which will follow easily from standard random walk techniques, and we derive in Section \ref{sub:uni:contr:2} the bounds for the diameter. These controls will then be used in Section \ref{sec:tight:curve} to control the modulus of continuity of $\Gamma^{(\varepsilon)}$.

\subsection{Uniform control of volumes and boundary lengths}\label{sub:uni:contr:1}
We start by  studying  the volume which is the simpler to obtain.

\begin{proposition}\label{prop:control:Vol:uniform}
There exist $c,C\in (0,\infty)$, such that for every $\delta\in (0,1]$:
$$\sup_{\varepsilon\in(0,\sqrt{\delta})}\P\big(\mathrm{Vol}_\infty(\mathcal{D}^\bullet_{T_\delta^{(\varepsilon)}-1}(\varepsilon))>x\big)\leq C \cdot \exp\big(-c \cdot \frac{x}{\delta^2}\big),\quad \text{ for }~x\geq 0.$$
\end{proposition}
\noindent We start with a Tauberian lemma which will also be useful later:
\begin{lemma}\label{lem:control:Vol:uniform}
There exist $C\in (0,\infty)$, such that for every $a,\varepsilon>0$:
$$\P\big(\mathrm{Vol}_\infty(K_1^{\bullet}(\varepsilon))>a\big)\leq C\cdot \varepsilon^{3}\cdot a^{-3/4}. $$
\end{lemma}

\begin{proof}
Introduce the function
$$F(t):=\exp(-\sqrt{\frac{8}{3}}t)\cdot  \frac{\coth(t)+\sqrt{2/3}}{\coth(t)-\sqrt{2/3}},\quad \text{ for } ~t> 0.$$
By \eqref{joint-Lap}, we have $\E[\exp(-\mu \mathrm{Vol}(K_1^{\bullet}(\varepsilon)))]=F((2\mu)^{1/4} \varepsilon)$. Now remark that:
$$C:=\sup_{t>0}\: (1-F(t))\cdot  t^{-3} <\infty.$$
In particular, for every $a>0$, we have:
$$\P\big(\mathrm{Vol}_\infty(K_1^{\bullet}(\varepsilon))>a\big)\leq\frac{1- \E\big[\exp(-a^{-1}\cdot \mathrm{Vol}(K_1^{\bullet}(\varepsilon)))\big]}{1-\exp(-1)}\leq \frac{ 2^{\frac{3}{4}}C}{1-\exp(-1)} \cdot \varepsilon^{3}\cdot a^{-\frac{3}{4}}.$$
\end{proof}
We now have all we need to prove Proposition \ref{prop:control:Vol:uniform}.
\begin{proof}[Proof of Proposition \ref{prop:control:Vol:uniform}]
Fix $x>0$ and $\delta,\varepsilon>0 $ with $\delta\in (0,1]$ and $\varepsilon\leq \sqrt{\delta}$. By Lemma \ref{lem:control:Vol:uniform}, there exists $C_1>0$, such that 
$$\E\Big[( \mathrm{Vol}_\infty(K_1^\bullet(\varepsilon))\wedge \delta^2)\Big] =\int_0^{\delta^{2}}\mathrm{d} a ~\P( \mathrm{Vol}(K_1^{\bullet}(\varepsilon))>a)\leq C_1\cdot \varepsilon^{3}\cdot\delta^{\frac{1}{2}}.$$
Now remark that
\begin{align*}
\P\big(\mathrm{Vol}_\infty(\mathcal{D}^\bullet_{T_\delta^{(\varepsilon)}-1}(\varepsilon))>x\big)&\leq \P\Big(\sum_{n\leq \lfloor \varepsilon^{-3}\delta^{3/2}\rfloor} (\mathrm{Vol}_\infty(K_n^\bullet(\varepsilon))\wedge \delta^2)>x\Big)\\
&\leq \exp(-\frac{x}{\delta^2})\cdot \mathbb{E}\Big[\exp(\delta^{-2} ( \mathrm{Vol}_\infty(K_1^\bullet(\varepsilon))\wedge \delta^2))\Big]^{\lfloor \varepsilon^{-3}\delta^{3/2}\rfloor}. 
\end{align*}
Since by convexity, for every $y\in[0,\delta^{2}]$, it holds that $\exp(\delta^{-2} y)\leq 1+(\exp(1)-1)\delta^{-2} y$, we get
\begin{align*}
 \mathbb{E}\big[\exp(\delta^{-2} ( \mathrm{Vol}_\infty(K_1^\bullet(\varepsilon))\wedge \delta^2))\big]&\leq 1+\delta^{-2}\cdot(\exp(1)-1)\cdot \E\big[( \mathrm{Vol}_\infty(K_1^\bullet(\varepsilon))\wedge \delta^2)\big]\\
 &\leq 1+ C_2\cdot \varepsilon^{3}\cdot \delta^{-\frac{3}{2}}\leq \exp( C_2\cdot \varepsilon^{3}\cdot \delta^{-\frac{3}{2}}),
 \end{align*}
 where $C_2:=(\exp(1)-1) C_1$. Hence, we have
$$\P\big(\mathrm{Vol}_\infty(\mathcal{D}^\bullet_{T_\delta^{(\varepsilon)}-1}(\varepsilon))>x\big)\leq \exp(C_2)\cdot \exp(-\frac{x}{\delta^2}). $$
\end{proof}

The rest of the section is devoted to establish a similar control on the boundary length of $\mathcal{D}^\bullet_{T_\delta^{(\varepsilon)}-1}(\varepsilon)$.

\begin{proposition}\label{prop:control:perimetre}
There exists $c,C\in (0,\infty)$, such that for every $x>0$ and $\delta\in(0,1]$, we have:
$$\sup_{\varepsilon\in(0,\sqrt{\delta})}\P\Big(|\partial_\circ(\mathcal{D}^\bullet_{T_\delta^{(\varepsilon)}-1}(\varepsilon))|>x\Big)\leq C \cdot \exp\big(-c \cdot \frac{x}{\delta}\big).$$
\end{proposition}
Again we start with a simple preliminary  lemma.
\begin{lemma}\label{lem:control:perimetre}
There exists $C,c>0$, such that for every $\delta\in(0,1)$ we have:
$$\sup_{\varepsilon\in(0,\sqrt{\delta})} \P\Big(\varepsilon^{\vartheta}\cdot \sum_{k=1}^{T_{\delta}^{\varepsilon}} (Y_k^{(\ell)}(\varepsilon)+Y_k^{(\mathrm{r})}(\varepsilon))> x \Big)\leq C \cdot \exp\big(-c \cdot \frac{x}{\delta^{\frac{3}{2}}}\big),\quad \text{ for }~x\geq 0.$$
\end{lemma}
\begin{proof}
First, by scaling invariance, we have:
$$\P\Big(\varepsilon^{\vartheta}\cdot \sum_{k=1}^{T_{\delta}^{\varepsilon}} (Y_k^{(\ell)}(\varepsilon)+Y_k^{(\mathrm{r})}(\varepsilon))> x \Big)\leq\P\Big( \sum_{k=1}^{\lfloor \varepsilon^{-(3\wedge (2+\vartheta))}\cdot \delta^{\frac{3}{2}} \rfloor} (Y_k^{(\ell)}(1)+Y_k^{(\mathrm{r})}(1))> \varepsilon^{-(2+\vartheta)}\cdot x \Big)$$
By the spatial Markov property and Markov inequality (multiplying both sides by $\lfloor \varepsilon^{-(3\wedge (2+\vartheta))}\cdot \delta^{\frac{3}{2}} \rfloor^{-1}$), the right-hand side term of the previous display is bounded above by:
\begin{align*}
&\exp\big(-\frac{x}{\delta^{\frac{3}{2}}}\big) \cdot \E\Big[\exp\Big(\lfloor \varepsilon^{-(3\wedge (2+\vartheta))}\cdot \delta^{\frac{3}{2}} \rfloor^{-1}\cdot  (Y_1^{(1)}(\varepsilon)+Y_1^{(\mathrm{r})}(1))\Big)\Big]^{\lfloor \varepsilon^{-(3\wedge (2+\vartheta))}\cdot \delta^{\frac{3}{2}} \rfloor}\\
&\leq \exp\big(-\frac{x}{\delta^{\frac{3}{2}}}\big) \cdot \E[\exp(Y_1^{(1)}(1)+Y_1^{(\mathrm{r})}(1))],     
\end{align*}
where to obtain the second inequality we used Jensen inequality. The lemma follows directly since $\E[\exp(Y_1^{(1)}(\varepsilon)+Y_1^{(\mathrm{r})}(1))]<\infty$ by \eqref{densite-bord-Z}.
\end{proof}

We can now proceed with the proof of Proposition \ref{prop:control:perimetre}

\begin{proof}[Proof of Proposition \ref{prop:control:perimetre}]
The idea is to use Lemma \ref{prop:formule:bord} to control $|\partial_\circ\mathcal{D}^\bullet_{T_\delta^{(\varepsilon)}-1}(\varepsilon)|$ in terms of the processes $H^{(\ell)}(\varepsilon)$ and $H^{(\mathrm{r})}(\varepsilon)$, so that the desired bound can be obtained by controlling random walks. Namely, observe that by Lemma \ref{prop:formule:bord}, for every $n\in\N$, we have:
$$\Big|\big|\partial_\circ\mathcal{D}^\bullet_n(\varepsilon)\big|-\big(H_n^{(\ell)}(\varepsilon)+H_n^{(\mathrm{r})}(\varepsilon)-2\inf_{j\in[1,n]}H_j^{(\ell)}(\varepsilon)-2\inf_{j\in[1,n]}H_j^{(\mathrm{r})}(\varepsilon)\big)\Big|\leq 4\sup_{j\in[1,n]} \big(Y_j^{(\ell)}(\varepsilon)\vee Y_j^{(\mathrm{r})}(\varepsilon)\big).$$
Then, remark that, for $x,\delta,\varepsilon>0$ with $\varepsilon<\sqrt{\delta}$, we have:
\begin{align*}
\P\big(\sup_{i\leq \lfloor \varepsilon^{-3}\delta^{3/2}\rfloor} Y_i^{(\ell)}(\varepsilon)\vee Y_{i}^{(\mathrm{r})}(\varepsilon)>x\big)  
\leq 2\lfloor \varepsilon^{-3}\delta^{3/2}\rfloor\cdot  \P\big(Y_1^{(\ell)}(\varepsilon)>x\big)
=2\lfloor \varepsilon^{-3}\delta^{3/2}\rfloor\cdot  \exp\big(-\frac{3 x}{2\varepsilon^{2}}\big),
\end{align*}
where in the last line we used \eqref{densite-bord-Z}. 
Hence, by Lemma \ref{prop:formule:bord} and symmetry, it then suffices to show that there exists $c,C\in(0,\infty)$ such that:
\begin{equation}\label{eq:exp:tension:1}
\sup_{\varepsilon\in(0,\sqrt{\delta})}\P\Big(\sup_{k< T_{\delta}^{\varepsilon}} H_k^{(\ell)}(\varepsilon)>x\Big)\leq C\cdot \exp\big(-c\cdot\frac{x}{\delta}\big)
\end{equation}
and
\begin{equation}\label{eq:exp:tension:2}
\sup_{\varepsilon\in(0,\sqrt{\delta})}\P\Big(\inf_{k< T_{\delta}^{\varepsilon}} H_k^{(\ell)}(\varepsilon)<-x\Big)\leq C\cdot  \exp\big(-c\cdot \frac{x}{\delta}\big).
\end{equation}
Moreover by Lemma \ref{lem:control:perimetre}, the fact that $\delta \in(0,1]$, and the spatial Markov property it follows that, up to replace, for every $\varepsilon>0$, the variables $(U_k(\varepsilon))_{k\geq 1}$ by 
$$\Big(\frac{U_k(\varepsilon)-d\varepsilon^{\vartheta}+|d\varepsilon^{\vartheta}|}{1+2d\varepsilon^{\vartheta}}\Big)_{k\geq 1},$$
it suffices to establish \eqref{eq:exp:tension:1} and \eqref{eq:exp:tension:2} in the special case $d=0$. In this direction, for the rest of the proof we assume that $d=0$.
\par
Let us start by proving \eqref{eq:exp:tension:1}. To this end, note that:
$$\P\Big(\sup_{k< T_{\delta}^{\varepsilon}} H_k^{(\ell)}(\varepsilon)>x\Big)\leq\P\Big(\sup_{k\leq \delta^{3/2}\varepsilon^{-3}} H_k^{(\ell)}(\varepsilon)>x\Big).$$
Since $\E[H_1^{(\ell)}(\varepsilon)]=0$, by the spatial Markov property, the sequence $$\exp\Big(\frac{H_{k}^{(\ell)}(\varepsilon)}{2\delta}\Big), \quad k\geq 0, $$
is a $\sigma(H_0^{(\ell)}(\varepsilon),\dots, H_k^{(\ell)}(\varepsilon) )$, $k\geq 0$, submartingale, which yields that 
\begin{align*}
\P\Big(\sup_{k\leq \lfloor \varepsilon^{-3}\delta^{3/2}\rfloor} H_k^{(\ell)}(\varepsilon)>x\Big)&\leq \exp(-\frac{x}{2\delta})\cdot  \E\big[\exp(\frac{H_{\lfloor \varepsilon^{-3}\delta^{3/2}\rfloor}^{(\ell)}(\varepsilon)}{2\delta})\big]\\
&=\exp(-\frac{x}{2\delta})\cdot  \E\big[\exp(\frac{H_{1}^{(\ell)}(\varepsilon)}{2\delta})\big]^{\lfloor \varepsilon^{-3}\delta^{3/2}\rfloor}.
\end{align*}
By \eqref{densite-bord}, we have
$$\mathbb{E}\big[\exp\big(\frac{H_1^{(\ell)}(\varepsilon)}{2\delta}\big)\big]=\Big(1+\varepsilon/\sqrt{3\delta} \Big)\cdot \exp\big(-\varepsilon/\sqrt{3\delta}\big) \cdot \mathbb{E}\big[\exp(\frac{U_1(\varepsilon) Z_{\varepsilon}}{2\delta})\big].$$
By convexity: 
$$ \mathbb{E}\big[\exp(\frac{ U_1(\varepsilon) Z_{\varepsilon}}{2\delta})\big]\leq \mathbb{E}\Big[(1-U_1(\varepsilon))+U_1(\varepsilon)\exp(\frac{ Z_\varepsilon}{2\delta})\Big]=\frac{1+\mathbb{E}[\exp(\frac{ Z_\varepsilon}{2\delta})]}{2}=\frac{1-\varepsilon^2 /(6\delta)}{1-\varepsilon^{2}/(3\delta)}.$$
We infer that:
\begin{align*}
\sup_{\varepsilon\in(0,\sqrt{\delta})}\P\Big(\sup_{k\leq T_{\delta}^{\varepsilon}} H_k^{(\ell)}(\varepsilon)>x\Big)&\leq \exp(-\frac{x}{2\delta})\cdot\sup_{\varepsilon\in(0,\sqrt{\delta})}\Big( \frac{1-\varepsilon^2/ (6\delta)}{1-\varepsilon/\sqrt{3\delta}}\cdot \exp\big(-\varepsilon/\sqrt{3\delta}\big) \Big)^{\lfloor \varepsilon^{-3}\delta^{3/2}\rfloor}  \\
&\leq \exp(-\frac{x}{2\delta})\cdot \sup_{0<a<3^{-\frac{1}{2}}}  h(a),
\end{align*}
where 
$$ h(a):=\exp\Big(-a^{-3}\big(a -\log(1-a^{2}/2)+\log(1-a)\big)\Big).$$
It is straightforward to verify that $ h(a)<2$, for every $a\in(0,3^{-1/2})$, and we obtain \eqref{eq:exp:tension:1}. Let us now establish \eqref{eq:exp:tension:2}. To this end, we introduce the process:
$$H_n^{(\ell),\delta}(\varepsilon):=\sum_{k=1}^n\Big(Y_k^{(\ell)}(\varepsilon)-X_k^{(\ell)}(\varepsilon)\wedge \delta\Big).$$
By definition, for every $x>0$, we have 
$$\P\Big(\inf_{k< T_\delta^{(\varepsilon)} } H_k^{(\ell)}(\varepsilon)<-x\Big)\leq \P\Big(\inf_{k>0} H_k^{(\ell),\delta}(\varepsilon)<-x\Big).$$
We are going to conclude the proof of \eqref{eq:exp:tension:2} by standard ruin probability arguments. First we claim that there exists  $\lambda>0$, which does not depend on $\varepsilon$ and $\delta$, such that
$$\E\big[\exp(-\frac{\lambda}{\delta} H_1^{(\ell),\delta}(\varepsilon) )\big]<1. $$
To see it remark that, for $\lambda>0$, we have:
\begin{align*}
\mathbb{E}\Big[\exp(-\frac{\lambda}{\delta} H_1^{(\ell),\delta}(\varepsilon) )\Big]&= \mathbb{E}\Big[\exp(-\lambda \frac{U_1 Z_\varepsilon}{\delta} )\Big]\cdot \mathbb{E}\Big[\exp(\lambda\frac{\beta_\varepsilon\wedge \delta}{\delta} ) \Big].
\end{align*}    
Again by convexity we have 
$$\mathbb{E}[\exp(-\lambda\frac{U_1 Z_\varepsilon}{\delta} )]\leq \frac{1+\E[\exp(-\lambda Z_\varepsilon/\delta )]}{2}=\frac{1+ \lambda \varepsilon^2/(3\delta)}{1+ 2\lambda \varepsilon^2/(3\delta)}.$$
If we restrict to the case $\lambda\in(0,1)$, we can use the inequality $\exp(\lambda y)\leq 1+\lambda y+ \lambda^{2} y^2$, for $y\in [0,1]$, to infer that
$$ \mathbb{E}\big[\exp(\lambda \frac{\beta_\varepsilon\wedge \delta}{\delta} ) \big]\leq 1+\frac{\lambda}{\delta}\cdot  \E\big[\beta_\varepsilon \wedge \delta\big]+ \frac{\lambda^2}{\delta^2}\cdot \E\big[(\beta_\varepsilon \wedge \delta)^2\big].$$
Applying \eqref{densite-bord}, it is straightforward to derive that we can find $c_1,c_2>0$ not depending on $\varepsilon,\delta,\lambda$ such that
$$ \E\big[\beta_\varepsilon \wedge \delta\big]\leq \frac{\varepsilon^2}{3}-c_1 \frac{\varepsilon^3}{\delta^{\frac{1}{2}}}\quad \text{ and }\quad \E\big[(\beta_\varepsilon \wedge \delta)^2\big]\leq c_2 \delta^{\frac{1}{2}}\varepsilon^3.$$
We infer that:
$$\mathbb{E}\Big[\exp(-\frac{\lambda}{\delta} H_1^{(\ell),\delta}(\varepsilon) )\Big]\leq \sup_{a\in(0,1/3)}\Big(\frac{1+\lambda\cdot a}{1+2\lambda\cdot a} \cdot \Big(1+\lambda\cdot a-  3^{\frac{3}{2}}(c_1\lambda-c_2\lambda^2) a^{\frac{3}{2}}  \Big)\Big)$$
and thus we can find $\lambda\in(0,1)$ such that the supremum in the right-hand side is equal to $1$. We fix such a $\lambda>0$. 
To conclude now remark that the sequence  $\exp(-\lambda\frac{H_{k}^{(\ell),\delta}(\varepsilon)}{\delta})$, $k\geq 0, $
is a $\sigma(H_0^{(\ell),\delta}(\varepsilon),\dots, H_k^{(\ell),\delta}(\varepsilon) )$, $k\geq 0$, supermartingale. Next, we set $S=\inf\{k\geq 0:~H_{k}^{(\ell),\delta}(\varepsilon)<-x\}$, and an application of the optimal stopping  theorem entails:
$$\exp(\frac{\lambda x}{\delta}) \cdot \P(S\leq n)\leq \E\big[\exp(-\lambda \frac{H_{S\wedge n}^{(\ell),\delta}(\varepsilon)}{\delta})\big] \leq 1,$$
for every $n\geq 1$. We deduce that $\P(S\leq n)\leq \exp(-\lambda \frac{x}{\delta})$, and \eqref{eq:exp:tension:2} follows by taking $n\to \infty$.
\end{proof}

\subsection{Uniform control of the diameter}\label{sub:uni:contr:2}

The goal of this section is to prove:
\begin{proposition}\label{Proposition:inclusion:hull}
There exists $c,C>0$, such that for every $A,\delta,\varepsilon>0$, with $\delta<1$ and  $\varepsilon \leq \sqrt{\delta}$, we have:
$$ \P\big(\mathcal{D}^\bullet_{T_\delta^{(\varepsilon)}-1}(\varepsilon) \not\subset B_A^{\bullet}(\mathfrak{H})\big)\leq \P\big(\mathrm{diam}_{\mathfrak{H}}(\partial_\circ \mathcal{D}^\bullet_{T_\delta^{(\varepsilon)}-1}(\varepsilon))\geq A\big) \leq C\cdot \exp\big(-c\cdot \frac{A}{\sqrt{\delta}}\big). $$
\end{proposition}

In order to prove this result, we will need the following lemma to control the distances along $\partial_\circ \mathfrak{H}$.

\begin{lemma}\label{lem:FKP}
There exists $c,C>0$ such that for every $A,\delta>0$, we have:
$$ \sup_{\varepsilon>0} \P\big(\sup_{t\in[-A\sqrt{\delta}, A\sqrt{\delta}]} D_\infty(\Lambda_\infty(0), \Lambda_\infty(t))> A\big| \beta_\varepsilon \vee \gamma_\varepsilon\vee \mathrm{Vol}_\infty(B_\varepsilon^\bullet(\mathfrak{H}))^{\frac{1}{2}}>\delta\big)\leq C\cdot \exp(-c\cdot \frac{A}{\sqrt{\delta}}). $$ 
\end{lemma}

\begin{proof}
We use the notation of Section \ref{Brownian half-plane}, and we set:
$$V^{-}(\varepsilon):=\sum \limits_{i\in I : t_i\in [-\beta_\varepsilon,0)}\sigma(\omega^i)+ \sum \limits_{i\in I : t_i<- \beta_\varepsilon}\big(\sigma(\omega^i)-\sigma(\mathrm{tr}_{\varepsilon}(\omega^i))\big)$$
and
$$V^{+}(\varepsilon):=\sum \limits_{i\in I : t_i\in (0,\gamma_\varepsilon]}\sigma(\omega^i) +\sum \limits_{i\in I : t_i>\gamma_\varepsilon}\big(\sigma(\omega^i)-\sigma(\mathrm{tr}_\varepsilon(\omega^i))\big).$$
We stress that since, under $\mathbb{P}$,   $t_i\neq0 $, for every $i\in I$, and the definition of $\mathrm{Vol}_\infty$ and \eqref{hull}, we have:
$$\mathrm{Vol}_{\infty}\big(B_\varepsilon^\bullet(\mathfrak{H})\big)=V^{-}(\varepsilon)+V^{+}(\varepsilon).$$
For simplicity, we set $N_\varepsilon^{-}:=\beta_\varepsilon\vee V^{-}(\varepsilon)^{\frac{1}{2}}$ and $N_\varepsilon^{+}:=\gamma_\varepsilon\vee V^{+}(\varepsilon)^{\frac{1}{2}}.$ By \eqref{distance:to:the:root}, it suffices to find $c,C>0$ such that:
\begin{equation}\label{equation:sup:cond:N}
\sup_{\varepsilon>0} \P\big(\sup_{t\in[-A\sqrt{\delta}, A\sqrt{\delta}]} \mathcal{R}_t> A\big| N_\varepsilon^{-}\vee N_\varepsilon^{+}>\delta\big)\leq C\cdot \exp(-c\cdot \frac{A}{\sqrt{\delta}}).
\end{equation}
Next remark that by symmetry, it holds that:
\begin{align*}
\P\big(\sup_{t\in[-A\sqrt{\delta}, A\sqrt{\delta}]} \mathcal{R}_t> A \:;\:  N_\varepsilon^{+}\vee N_\varepsilon^{-}>\delta\big)   \leq  2\cdot \P(\sup_{t\in[-A\sqrt{\delta}, A\sqrt{\delta}]} \mathcal{R}_t> A; N_\varepsilon^{+}>\delta).
\end{align*}
Moreover the latter is bounded above by 
\begin{align*}
&2\cdot\Big( \P(\sup_{t\in[0, A\sqrt{\delta}]} \mathcal{R}_t> A; N_\varepsilon^{+}>\delta)+\P(\sup_{t\in[- A\sqrt{\delta},0]} \mathcal{R}_t> A; N_\varepsilon^{+}>\delta)\Big) \\
&\leq 2\cdot\Big( \P(\sup_{t\in[0, A\sqrt{\delta}]} \mathcal{R}_t> A; N_\varepsilon^{+}>\delta)+\P(\sup_{t\in[- A\sqrt{\delta},0]} \mathcal{R}_t> A)\cdot  \P(N_\varepsilon^{+}>\delta)\Big) 
\end{align*}
where in the second line we used that $\mathcal{R}_{t}$, $t\leq 0$, and $N_\varepsilon^{+}$ are independent. Hence, the left-hand side term of \eqref{equation:sup:cond:N} is bounded above by:
$$2\cdot  \sup_{\varepsilon>0}\P(\sup_{t\in[0, A\sqrt{\delta}]} \mathcal{R}_t> A | N_\varepsilon^{+}>\delta)+ 2\cdot \P(\sup_{t\in[0,A\sqrt{\delta}]} \mathcal{R}_t> A).$$
By Lemma \ref{control:distance:bord}, we already know that there exists $c, C>0$ such that $\P\big(\sup_{t\in[0, A\sqrt{\delta}]}  \mathcal{R}_t> A/(2\sqrt{3})\big)\leq C \exp(-c\frac{A}{\sqrt{\delta}})$, and then to conclude it suffices to show that
$$\P\big(\sup_{t\in[0, A\sqrt{\delta}]}  \mathcal{R}_t> A\big|  N_\varepsilon^{+} >\delta)\leq  \P\big(\sup_{t\in[0, A\sqrt{\delta}]}  \mathcal{R}_t> A).$$
This inequality can be seen as a negative correlation between the events $\{\sup_{t\in[0, A\sqrt{\delta}]}  \mathcal{R}_t> A\}$ and $\{N_\varepsilon^{+}  >\delta\}$, and this is a direct consequence of the FKG inequality of Bessel processes established in \cite[Corollary 1.10]{AlexandreLeGrand}.
\end{proof}
Let us now proceed with the proof of Proposition \ref{Proposition:inclusion:hull}.
\begin{proof}[Proof of Proposition \ref{Proposition:inclusion:hull}.]
Recall the notation $\partial \mathcal{D}_{T_\delta^{(\varepsilon)}-1}^\bullet(\varepsilon)$ for the topological boundary of $\mathcal{D}_{T_\delta^{(\varepsilon)}-1}^\bullet$.  By planarity, we have:
$$\P(\mathcal{D}_{T_\delta^{(\varepsilon)}-1}^\bullet \not\subset B_A^{\bullet})\leq \P\big(\mathrm{diam}_{\mathfrak{H}}(\partial_\circ \mathcal{D}_{T_\delta^{(\varepsilon)}-1})>A\big),$$
and, to simplify notation, we write:
$$\partial^{\prime }\mathcal{D}_{T_\delta^{(\varepsilon)}-1}^\bullet(\varepsilon) :=\mathcal{D}_{T_\delta^{(\varepsilon)}-1}^\bullet(\varepsilon)\cap \partial_\circ \mathfrak{H}.$$
Remark that we have
$$\partial_\circ\mathcal{D}_{T_\delta^{(\varepsilon)}-1}^\bullet(\varepsilon) = \partial \mathcal{D}_{T_\delta^{(\varepsilon)}-1}^\bullet(\varepsilon) \cup\partial^{\prime }\mathcal{D}_{T_\delta^{(\varepsilon)}-1}^\bullet(\varepsilon) \quad \text{ and }\quad \partial \mathcal{D}_{T_\delta^{(\varepsilon)}-1}^\bullet(\varepsilon) \cap\partial^{\prime }\mathcal{D}_{T_\delta^{(\varepsilon)}-1}^\bullet(\varepsilon)\neq \varnothing,$$
and consequently it holds that:
$$ \P\Big(\mathrm{diam}_{\mathfrak{H}}(\partial_\circ \mathcal{D}_{T_\delta^{(\varepsilon)}-1})>A\Big)\leq \P\Big(\mathrm{diam}_{\mathfrak{H}}(\partial \mathcal{D}_{T_\delta^{(\varepsilon)}-1})>\frac{A}{2}\Big)+\P\Big(\mathrm{diam}_{\mathfrak{H}}(\partial^\prime \mathcal{D}_{T_\delta^{(\varepsilon)}-1})>\frac{A}{2}\Big).$$
Let us control each term separately. First, note that
$$\P\big(\mathrm{diam}_{\mathfrak{H}}(\partial^\prime\mathcal{D}_{T_\delta^{(\varepsilon)}-1})>A/2\big) \leq \P\big(|\partial_\circ\mathcal{D}_{T_\delta^{(\varepsilon)}-1}|>A\sqrt{\delta}\big)+\P\big(\sup_{t\in [-A\sqrt{\delta},A\sqrt{\delta}]} D_\infty(\Lambda_\infty(0), \Lambda_\infty(t))>A/4\big).$$
By Lemma \ref{control:distance:bord} and Proposition \ref{prop:control:perimetre}, there exists $c_1,C_1>0$ universal constant such that the latter display is bounded above by $C_1\cdot \exp(-c_1\frac{A}{\sqrt{\delta}})$.  The idea to control $\P\big(\mathrm{diam}_{\mathfrak{H}}(\partial \mathcal{D}_{T_\delta^{(\varepsilon)}-1})>A/2\big)$ is
similar. We start by noticing that $\P\big(\mathrm{diam}_{\mathfrak{H}}(\partial \mathcal{D}_{T_\delta^{(\varepsilon)}-1})>A/2\big)$ is bounded above by:
\begin{align*}
&\P\Big(|\partial_\circ \mathcal{D}_{T_\delta^{(\varepsilon)}-1} |>A\sqrt{\delta}\Big)+ \P\Big(\mathrm{diam}_{\mathfrak{H}}(\partial\mathcal{D}_{\lfloor \varepsilon^{-\widetilde{\vartheta}}\cdot \delta^{\frac{3}{2}} \rfloor-1})>\frac{A}{2}, |\partial_\circ \mathcal{D}_{\lfloor \varepsilon^{-\widetilde{\vartheta}}\delta^{3/2}\rfloor-1} |<A\sqrt{\delta}\Big)\\
&+\P\Big(\mathrm{diam}_{\mathfrak{H}}(\partial \mathcal{D}_{T_\delta^{(\varepsilon)}-1})>\frac{A}{2}, |\partial\mathcal{D}_{T_\delta^{(\varepsilon)}-1} |<A\sqrt{\delta}, T_\delta^\varepsilon<\lfloor \varepsilon^{-\widetilde{\vartheta}}\cdot \delta^{\frac{3}{2}} \rfloor\Big),
\end{align*}
where to simplify notation we write $\widetilde{\vartheta}:=3\wedge (2+\vartheta)$. Now remark that an application of the spatial Markov property entails
$$\P\Big(\mathrm{diam}_{\mathfrak{H}}(\partial\mathcal{D}_{\lfloor \varepsilon^{-\widetilde{\vartheta}}\cdot \delta^{\frac{3}{2}} \rfloor-1})>\frac{A}{2}, |\partial_\circ \mathcal{D}_{\lfloor \varepsilon^{-3}\delta^{3/2}\rfloor-1} |<A\sqrt{\delta}\Big)\leq \P\big(\sup_{t\in [-A\sqrt{\delta},A\sqrt{\delta}]} D_\infty(\Lambda_\infty(0), \Lambda_\infty(t))>\frac{A}{4} \big)$$
and 
 \begin{align*}
     &\P\Big(\mathrm{diam}_{\mathfrak{H}}(\partial \mathcal{D}_{T_\delta^{(\varepsilon)}-1})>\frac{A}{2}, |\partial\mathcal{D}_{T_\delta^{(\varepsilon)}-1} |<A\sqrt{\delta}, T_\delta^\varepsilon<\lfloor \varepsilon^{-\widetilde{\vartheta}}\delta^{3/2}\rfloor\Big)\\
     &\leq  \P\Big(\sup_{t\in[-A\sqrt{\delta}, A\sqrt{\delta}]} D_\infty(\Lambda_\infty(0), \Lambda_\infty(t))> \frac{A}{4}~\Big| \beta_\varepsilon \vee \gamma_\varepsilon\vee \mathrm{Vol}(B_\varepsilon^\bullet(\mathfrak{H}))^{\frac{1}{2}}>\delta\Big).
 \end{align*}
Combining Lemmas \ref{control:distance:bord} and   \ref{lem:FKP} with Proposition  \ref{prop:control:perimetre} yields that we can find $c_2,C_2>0$ universal constants such that $\P\big(\mathrm{diam}_{\mathfrak{H}}(\partial \mathcal{D}_{T_\delta^{(\varepsilon)}-1})>A/2\big)$ is bounded above by $C_2\cdot \exp(-c_2\frac{A}{\sqrt{\delta}})$. 
\end{proof}

\subsection{Control on the modulus of continuity}\label{sec:tight:curve}
The goal of this section is to control the modulus of continuity of $\Gamma^{(\varepsilon)}$. In this direction, for every $t,\delta,\varepsilon>0$, we set:
\begin{equation}\label{eq:modulus:continuity}
\varpi^{(\varepsilon)}\big(\delta, t\big):=\sup\Big\{D_\infty\big(\Gamma^{(\varepsilon)}( \varepsilon^{-(3\wedge (2+\vartheta))}\cdot s_1),\Gamma^{(\varepsilon)}( \varepsilon^{-(3\wedge (2+\vartheta))}\cdot s_2)\big):~s_1,s_2\in[0,t]~\text{ with }~|s_1-s_2|\leq \delta^{\frac{3}{2}}\Big\}.
\end{equation}
\begin{proposition}\label{prop:Tightness}
 For every $\kappa,t>0$, we have
$$\lim_{A\to \infty}\limsup_{\varepsilon\to 0}\mathbb{P}\Big(\sup_{\delta\in[\varepsilon^2,1]}\delta^{-\frac{1}{2}+\kappa}\cdot\varpi^{(\varepsilon)}\big(\delta, t\big)\geq A\Big)=0.$$
\end{proposition}

\begin{proof}
To simplify notation we treat the case  $t=1$, the proof is exactly the same for arbitrary $t>0$, and we set:
$$\widetilde{\vartheta}:=3\wedge (2+\vartheta).$$
Let us start by introducing some notation.  For every $\varepsilon,\delta>0$, we set $T_\delta^{(\varepsilon, 0)}:=0$, and for $k\geq 0$, we let
$$T_\delta^{(\varepsilon), k+1}=\inf\big\{n>T_\delta^{(\varepsilon, k)}:~X_n^{(\ell)}(\varepsilon)\vee X_n^{(\mathrm{r})}(\varepsilon)>\delta  ~\text{ or }~ \mathrm{Vol}(K_n^{\bullet}(\varepsilon))>\delta^2\big\}\wedge \big( T_\delta^{(\varepsilon), k}+ \lfloor \varepsilon^{-\widetilde{\vartheta}}\cdot \delta^{\frac{3}{2}} \rfloor\big).$$
By the spatial Markov property, the random variables
$$\big(T_\delta^{(\varepsilon), k+1}-T_\delta^{(\varepsilon), k}\big),\quad k\geq 0, $$
are i.i.d distributed as the minimum between  $\lfloor \varepsilon^{-(3\wedge (2+\vartheta))}\cdot \delta^{\frac{3}{2}} \rfloor$ and a geometric random variable with parameter $p(\delta,\varepsilon)$,  where:
\begin{equation}\label{eq:p:parameter:geo}
p(\delta,\varepsilon):=\P\Big(X_1^{(\ell)}(\varepsilon)\vee X_1^{(\mathrm{r})}(\varepsilon)\vee \mathrm{Vol}(B^\bullet_\varepsilon)^{\frac{1}{2}}>\delta \Big).    
\end{equation}
Remark that by \eqref{densite-bord}, \eqref{densite-bord-Z} and Lemma \ref{lem:control:Vol:uniform}, there exists $C>0$, such that for every $\delta,\varepsilon>0$ with $\varepsilon^2\leq \delta\leq  1$, we have:
\begin{equation}\label{eq:bound:p(delta,varepsilon)}
C^{-1} \cdot \delta^{-\frac{3}{2}}\cdot  \varepsilon^{3} \leq p(\delta,\varepsilon)\leq C \cdot \delta^{-\frac{3}{2}}\cdot  \varepsilon^{3}.    
\end{equation}
Fix $\kappa_1\in(0,1/2)$ and from now on we work with $\varepsilon,\delta\in(0,1)$,  with $\delta\geq \varepsilon^2$,  and $A>1$. We introduce the event:
$$\bigA(\delta,\varepsilon, A):= \bigcup_{0\leq i\leq \lfloor \delta^{-3/2}\rfloor} \big\{\exists s_1,s_2\in [i\delta^{3/2}, (i+1)\delta^{3/2}] :  D_\infty\big(\Gamma^{(\varepsilon)}(\varepsilon^{-\widetilde{\vartheta}}s_1),\Gamma^{(\varepsilon)}(\varepsilon^{-\widetilde{\vartheta}}s_2)\big) >  A\cdot  \delta^{\frac{1}{2}-\kappa_1}\big\},$$
and remark that it  suffices to show that:
$$\lim_{A\to \infty}\limsup_{\varepsilon\to 0}\P\Big(\bigcup_{m\in \mathbb{N}, 2^{-m}\geq \varepsilon^{2}}\bigA(2^{-m},\varepsilon,A)\Big)=0. $$
For $\varepsilon,\delta\in(0,1)$,  with $\delta\geq \varepsilon^2$,  and $A>1$,  and $0\leq i\leq \lfloor \delta^{-3/2}\rfloor$, we  set
$$\mathcal{H}_{i}(\delta,\varepsilon):=\bigcup_{\varepsilon^{\widetilde{\vartheta}}\cdot T_\delta^{(\varepsilon),k}\in [i\delta^{3/2}, (i+1)\delta^{3/2}]} \Big( \mathrm{cl}\Big(\mathcal{D}_{T_{\delta}^{(\varepsilon),k+1}-1}^{\bullet}(\varepsilon)\setminus \mathcal{D}_{T_{\delta}^{(\varepsilon),k}}^{\bullet}(\varepsilon) \Big) \cup \Gamma^{(\varepsilon)}\big([ (T_\delta^{(\varepsilon),k}-1), T_\delta^{(\varepsilon),k}]\big)\Big),$$
where we write $\mathrm{cl}(E)$ for the closure of a set $E$,  and we also consider the event:
$$
\bigB_1(\delta,\varepsilon,A):=\Big\{\max_{i\leq\lfloor \delta^{-3/2}\rfloor} \mathrm{diam}_{\mathfrak{H}}(\mathcal{H}_i(\delta, \varepsilon)\big) \geq A\cdot \delta^{1/2-\kappa_1}\Big\}.
$$
Since $\mathrm{diam}_{\mathfrak{H}}(\Gamma^{(\varepsilon)}([i\delta^{\frac{3}{2}},(i+1)\delta^{\frac{3}{2}}]))\leq   \mathrm{diam}_{\mathfrak{H}}(\mathcal{H}_{i-1}(\delta, \varepsilon)\big)+\mathrm{diam}_{\mathfrak{H}}(\mathcal{H}_i(\delta,\varepsilon)\big)$, for every $i\leq\lfloor \delta^{-3/2}\rfloor$, it is enough to establish that 
\begin{equation}\label{B:1:limsup:limsup}
   \lim_{A\to \infty}\limsup_{\varepsilon\to 0}\P\big(\bigcup_{m\in \mathbb{N}}\bigB_1(\delta,\varepsilon,A)\big)=0.
\end{equation}
To this end fix $0<\kappa_3<\kappa_2<\kappa_1$ with $\kappa_2+\kappa_3<\kappa_1$, and set:
\begin{align*}
\bigB_2(\delta,\varepsilon,A):=&\Big\{\max_{k\leq \lfloor\delta^{-3/2-\kappa_3}\rfloor} \mathrm{diam}_{\mathfrak{H}}\Big(\partial_\circ\mathrm{cl} \big(\mathcal{D}_{T_{\delta}^{(\varepsilon),k+1}-1}^{\bullet}(\varepsilon)\setminus \mathcal{D}_{T_{\delta}^{(\varepsilon),k}}^{\bullet}(\varepsilon) \big)\Big) \geq A \cdot \delta^{\frac{1}{2}-\kappa_2} \\
&~\text{ or }~ \max_{k\leq \lfloor \delta^{-3/2-\kappa_3}\rfloor} \mathrm{Vol}\big(\mathcal{D}_{T_{\delta}^{(\varepsilon),k+1}-1}^{\bullet}(\varepsilon)\setminus \mathcal{D}_{T_{\delta}^{(\varepsilon),k}}^{\bullet}(\varepsilon)  \big) \geq A\cdot \delta^{2-\kappa_2}\Big\}
\end{align*}
and
$$
\bigB_3(\delta,\varepsilon,A):=\Big\{\max_{i\leq\lfloor \delta^{-3/2}\rfloor}\#\big\{k\geq 0:~\varepsilon^{\widetilde{\vartheta}}\cdot T_\delta^{(\varepsilon),k}\in [i\delta^{\frac{3}{2}}, (i+1)\delta^{\frac{3}{2}}]\big\}\geq A\cdot  \delta^{-\kappa_3}\Big\},
$$
for $\varepsilon,\delta\in(0,1)$,  with $\delta\geq \varepsilon^2$, and $ A>1$. To simplify notation, set $\mathbb{N}(\varepsilon):=\{m\in \mathbb{N}:~2^{-m}>\varepsilon^2\}$, for $\varepsilon\in(0,1)$. Now remark that for every $\varepsilon\in(0,1)$ and  $A^\prime,A>1$, we have
\begin{align}\label{eq:<:B:1}
\P\Big(\bigcup_{m\in\mathbb{N}(\varepsilon)}\bigB_1(2^{-m},\varepsilon,A)\Big)&\leq
\P\Big(\bigC(\varepsilon,A,A^\prime)\Big) +\P\Big(\bigcup_{m\in \mathbb{N}(\varepsilon)}\bigB_2(2^{-m},\varepsilon,A^\prime)\Big)\nonumber\\
&+\P\Big(\bigcup_{m\in \mathbb{N}(\varepsilon)} \bigB_3(2^{-m},\varepsilon,A^\prime)\Big)+ \P\big(\mathcal{D}_{2\varepsilon^{-\widetilde{\vartheta}}}^\bullet(\varepsilon) \not\subset B_{A^\prime}^\bullet\big),
\end{align}
where
$$\bigC(\varepsilon,A,A^\prime):= \Big(\bigcup_{m\in \mathbb{N}(\varepsilon)}\bigB_1(2^{-m},\varepsilon,A)\Big)\cap  \bigcap_{m\in \mathbb{N}(\varepsilon)}\Big(\bigB_2(2^{-m},\varepsilon,A^\prime)^{\mathrm{c}}\cap  \bigB_3(2^{-m},\varepsilon,A^\prime)^{\mathrm{c}} \Big)\cap\{\mathcal{D}_{2\varepsilon^{-\widetilde{\vartheta}}}^\bullet(\varepsilon) \subset B_{A^\prime}^\bullet\}.$$
By the spatial Markov property combined with a union bound, the probability of $\bigcup_{m\in \mathbb{N}(\varepsilon)}\bigB_2(2^{-m},\varepsilon,A^\prime)$ is bounded above by:
$$\sum_{m\in \mathbb{N}(\varepsilon)} 2^{(\frac{3}{2}+\kappa_3)m}\cdot \Big( \P\Big(\mathrm{diam}_{\mathfrak{H}}\big(\partial_\circ\mathcal{D}^{\bullet}_{T_{2^{-m}}^{(\varepsilon)}-1}\big)\geq  A^\prime\cdot 2^{-(\frac{1}{2}-\kappa_2+\kappa_3)m}\Big)+   \P\Big(\mathrm{Vol}\big(\mathcal{D}^{\bullet}_{T_{2^{-m}}^{(\varepsilon)}-1}\big)\geq A^\prime\cdot  2^{-(2-\kappa_2+\kappa_3)m}\Big) \Big)$$
and an application of Propositions \ref{prop:control:Vol:uniform} and \ref{Proposition:inclusion:hull} entails  that:
\begin{equation}\label{B:2:limsup}
   \lim_{A^\prime\to \infty}\limsup_{\varepsilon\to 0} \P\Big(\bigcup_{m\in \mathbb{N}(\varepsilon)}\bigB_2(2^{-m},\varepsilon,A^\prime)\Big)=0.
\end{equation}
Moreover, by standard properties of Geometric random variables, see \cite[Theorem 3.1]{Janson}, we also have:
\begin{equation}\label{B:3:limsup}
 \lim_{A^\prime\to \infty}\limsup_{\varepsilon\to 0} \P\Big(\bigcup_{m\in \mathbb{N}(\varepsilon)}\bigB_3(2^{-m},\varepsilon,A^\prime)\Big)=0.
\end{equation}
To control $\P\big(\mathcal{D}_{2\varepsilon^{-\widetilde{\vartheta}}}^\bullet(\varepsilon) \not\subset B_{A^\prime}^\bullet\big)$ simply notice that since $\mathcal{D}_{2\varepsilon^{-\widetilde{\vartheta}}}^\bullet(\varepsilon)$ is simply connected, we have:
\begin{align*}
    \P\big(\mathcal{D}_{2\varepsilon^{-\widetilde{\vartheta}}}^\bullet(\varepsilon) \not\subset B_{A^\prime}^\bullet\big)\leq \P\big(\mathrm{diam}_\mathfrak{H}\big(\mathcal{D}_{2\varepsilon^{-\widetilde{\vartheta}}}^\bullet(\varepsilon)\cap \partial_\circ \mathfrak{H}\big)\geq A/2\big)+\P\big(\mathrm{diam}_\mathfrak{H}\big(\mathcal{D}_{2\varepsilon^{-\widetilde{\vartheta}}}^\bullet(\varepsilon)\cap \partial_\circ \mathfrak{H}^{(2\varepsilon^{-\widetilde{\vartheta}})}\big)\geq A/2\big)
\end{align*}
and then  the spatial Markov property, combined with Lemmas \ref{control:distance:bord} and \ref{prop:formule:bord} and Proposition \ref{different:limite:beta:d}, yields:
\begin{equation}\label{B:4:limsup}
     \lim_{A^\prime\to \infty}\limsup_{\varepsilon\to 0} \P\Big(\mathcal{D}_{2\varepsilon^{-\widetilde{\vartheta}}}^\bullet(\varepsilon) \not\subset B_{A^\prime}^\bullet\Big)=0.
\end{equation}
Combining \eqref{eq:<:B:1} with \eqref{B:2:limsup}, \eqref{B:3:limsup}, \eqref{B:4:limsup},  to conclude, it  only remains to  establish that
\begin{equation}\label{Mundial2026:EspagneVSPortugal}
  \lim_{A\to \infty}\limsup_{\varepsilon\to 0}     \P\Big(\bigC(\varepsilon,A,A^\prime)\Big) =0.
\end{equation}
for every $A^\prime>1$. For the remainder of the proof we fix $A^{\prime}>1$, and introduce the quantity:
$$\mathfrak{W}:=\sup\Big\{ r \cdot \mathrm{Vol}\big(B_r(x)\big)^{-(1-\kappa_2)/4}:~x \in B_{A^\prime}^{\bullet}(\mathfrak{H}), r\in (0,1)\Big\}.$$
By the same Hölder-continuity argument used in \cite[Lemma 15]{Mie13}, it holds that $\P(\mathfrak{W}=\infty)=0$. For the remainder of the proof we take $\delta=2^{-m}$ for some $m\in \mathbb{N}$. For every $i\leq\lfloor \delta^{-3/2}\rfloor$, the quantity $\mathrm{diam}_{\mathfrak{H}}(\mathcal{H}_i(\delta,\varepsilon))$ is bounded above by:
\begin{align*}
\sum_{\varepsilon^{\widetilde{\vartheta}}\cdot T_\delta^{(\varepsilon,k)}\in [i \delta^{3/2},(i+1)\delta^{3/2}]} \Big(\mathrm{diam}_{\mathfrak{H}}\big(\mathcal{D}_{T_{\delta}^{(\varepsilon),k+1}-1}^{\bullet}(\varepsilon)\setminus \mathcal{D}_{T_{\delta}^{(\varepsilon),k}}^{\bullet}(\varepsilon) \big)+ \mathrm{diam}_{\mathfrak{H}}\big( \Gamma^{(\varepsilon)}\big([ (T_\delta^{(\varepsilon),k}-1),  T_\delta^{(\varepsilon),k}]\big)\big)\Big).
\end{align*} 
By construction, we have $$\mathrm{diam}_{\mathfrak{H}}\Big( \Gamma^{(\varepsilon)}\big([ (T_\delta^{(\varepsilon),k}-1),  T_\delta^{(\varepsilon),k}]\big)\Big)\leq \varepsilon\leq \delta^{1/2-\kappa_2}$$
since $\eta_{T_\delta^{(\varepsilon),k}}^{\varepsilon}$ is a geodesic on $\boldsymbol{\mathfrak{H}}^{(T_\delta^{(\varepsilon),k})}$. Moreover, we claim that under $\bigC(\varepsilon,A,A^\prime)$ it holds
\begin{equation}
 \mathrm{diam}_{\mathfrak{H}}\Big(\mathcal{D}_{T_{\delta}^{(\varepsilon),k+1}-1}^{\bullet}(\varepsilon)\setminus \mathcal{D}_{T_{\delta}^{(\varepsilon),k}}^{\bullet}(\varepsilon) \Big)\leq \big(1+2\cdot W\big)\cdot A^\prime\cdot \delta^{\frac{1}{2}-\kappa_2},
 \end{equation}
 for  $k\geq 0$ such that $T_\delta^{(\varepsilon), k}\in [0, (\lfloor \delta^{-3/2}\rfloor+1)\delta^{\frac{3}{2}}]$.
Before proving this inequality, let us explain why the proposition follows from it. Under $\bigC(\varepsilon,A,A^\prime)$,  we have:
$$\max_{i\leq\lfloor \delta^{-3/2}\rfloor} \mathrm{diam}_{\mathfrak{H}}(\mathcal{H}_i(\delta,\varepsilon))\leq  2\cdot (1+2\mathfrak{W}) \cdot (A^\prime)^2\cdot \delta^{\frac{1}{2}-\kappa_2-\kappa_3}$$
for every such $\delta$,  and we infer that  
$$\P\big(\bigC(\varepsilon,A,A^\prime)\big)\leq \P\big(W\geq \frac{A-2(A^{\prime})^2}{4 (A^{\prime})^2}\big),$$
which implies \eqref{Mundial2026:EspagneVSPortugal}.  To conclude, it remains to establish the claim. To this end, remark that if $x\in \mathcal{D}_{T_{\delta}^{(\varepsilon),k+1}-1}^{\bullet}(\varepsilon)\setminus \mathcal{D}_{T_{\delta}^{(\varepsilon),k}}^{\bullet}(\varepsilon) $ and $r>0$, with 
$$D_\infty\Big(x, \partial_\circ\mathrm{cl} \big(\mathcal{D}_{T_{\delta}^{(\varepsilon),k+1}-1}^{\bullet}(\varepsilon)\setminus \mathcal{D}_{T_{\delta}^{(\varepsilon),k}}^{\bullet}(\varepsilon) \big)\Big)>r,$$
then since $\mathcal{D}_{T_{\delta}^{(\varepsilon),k+1}-1}^{\bullet}(\varepsilon)\setminus \mathcal{D}_{T_{\delta}^{(\varepsilon),k}}^{\bullet}(\varepsilon) $ is simply connected, we will have 
$$B_r(x)\subset \mathcal{D}_{T_{\delta}^{(\varepsilon),k+1}-1}^{\bullet}(\varepsilon)\setminus \mathcal{D}_{T_{\delta}^{(\varepsilon),k}}^{\bullet}(\varepsilon).$$
And then, under $\bigC(\varepsilon,A,A^\prime)$, we  have $\mathrm{Vol}_\infty(B_r(x))\leq A^\prime \cdot  \delta^{2-\kappa_2}$ and thus $r\leq W\cdot (A^\prime)^{\frac{1-\kappa_2}{4}}\cdot\delta^{\frac{1}{2}(1-\kappa_2)}$. It follows that, under $\bigC(\varepsilon,A,A^\prime)$, for every $x,y\in \mathcal{D}_{T_{\delta}^{(\varepsilon),k+1}-1}^{\bullet}(\varepsilon)\setminus \mathcal{D}_{T_{\delta}^{(\varepsilon),k}}^{\bullet}(\varepsilon) $, we have
\begin{align*}
D_\infty(x,y)\leq& \mathrm{diam}_{\mathfrak{H}}\Big(\partial_\circ\big(\mathcal{D}_{T_{\delta}^{(\varepsilon),k+1}-1}^{\bullet}(\varepsilon)\setminus \mathcal{D}_{T_{\delta}^{(\varepsilon),k}}^{\bullet}(\varepsilon) \big)\Big)\\
&+2 \sup\Big\{D_\infty\big(z, \partial_\circ \big(\mathcal{D}_{T_{\delta}^{(\varepsilon),k+1}-1}^{\bullet}(\varepsilon)\setminus \mathcal{D}_{T_{\delta}^{(\varepsilon),k}}^{\bullet}(\varepsilon)\big) \big):~z\in \mathcal{D}_{T_{\delta}^{(\varepsilon),k+1}-1}^{\bullet}(\varepsilon)\setminus \mathcal{D}_{T_{\delta}^{(\varepsilon),k}}^{\bullet}(\varepsilon) \Big\}\\
&\leq \big(1+2\cdot \mathfrak{W}\big)\cdot A^\prime\cdot \delta^{\frac{1}{2}-\kappa_2},
\end{align*}
since $A^\prime>1$, which is the desired inequality.
\end{proof}

\section{Scaling limits of peeling explorations}\label{sec:peeling:scaling}
We are now in a position to establish the scaling limits of the peeling exploration. As in the previous section, we fix throughout this section the parameters $\vartheta \in \mathbb{R}_+$ and $d \in \mathbb{R}$, with $\vartheta=1$ if $d=0$.
 \begin{theorem}\label{theo:exi:Gamma}
Set $\vartheta^\prime:=3\wedge (2+\vartheta)$. The family:
\begin{equation}\label{eq:conv:theo:sous:suite}
\left(  \Big[\mathfrak{H},D_\infty, \Lambda_\infty(0), \mathrm{Vol}_\infty, \Lambda_\infty, \big(\Gamma^{(\varepsilon)}(\varepsilon^{-\vartheta^\prime} t)\big)_{t\geq 0}\Big] ,  \big( H_{\lfloor \varepsilon^{-\vartheta^\prime}t\rfloor}^{(\ell)}(\varepsilon), H_{\lfloor \varepsilon^{-\vartheta^\prime}t\rfloor}^{(\mathrm{r})}(\varepsilon), V_{\lfloor \varepsilon^{-\vartheta^\prime}t\rfloor}(\varepsilon)\big)_{t\geq 0}, \big[\mathcal{K}^\bullet(\varepsilon)\big]\right),
\end{equation}
for $\varepsilon\in (0,1)$, is asymptotically tight as $\varepsilon\downarrow 0$. Moreover, if $(\varepsilon_n)_{n\geq 0}$ is a sequence converging to $0$ under which \eqref{eq:conv:theo:sous:suite} converges in distribution, then up to enlarging the underlying probability space, we can assume that there exist:
\begin{itemize}
    \item $\Gamma:\mathbb{R}_+\to \mathfrak{H}$ a continuous function such that $\Gamma(0)=\Lambda_\infty(0)$;
    \item $\big(\widetilde{\cS}_t^{(\ell)},\ \widetilde{\cS}_t^{(\mathrm{r})}, \widetilde{\cV}_t\big)_{t\geq 0}$ and $\big[\widetilde{\mathcal{K}}^\bullet\big]$ random variables taking values in $\mathbb{D}(\mathbb{R}_+,\mathbb{R}^3)$ and $\big(\mathbb{R}_+^2\times\mathbb{M}_\infty^{\mathrm{GHPU}}\big)^\mathbb{N}$ respectively;
\end{itemize}
such that $\Big[\mathfrak{H},D_\infty, \mathrm{Vol}_\infty, \Lambda_\infty, \Gamma\Big]$ is a well defined random variable in $\mathbb{M}_\infty^{\mathrm{GHPU}}$ and \eqref{eq:conv:theo:sous:suite} converges in distribution towards:
$$\left(\Big[\mathfrak{H},D_\infty, \Lambda_\infty(0), \mathrm{Vol}_\infty, \Lambda_\infty, \Gamma \Big] , \big(\widetilde{\cS}_t^{(\ell)},\widetilde{\cS}_t^{(\mathrm{r})},\widetilde \cV_t\big)_{t\geq 0}, \big[\widetilde{\mathcal{K}}^\bullet\big]\right),$$
along the subsequence $(\varepsilon_n)_{n\geq 1}$.
\end{theorem}
Remark that the law $\big(\big(\widetilde{\cS}_t^{(\ell)},\ \widetilde{\cS}_t^{(\mathrm{r})}, \widetilde{\cV}_t\big)_{t\geq 0},\big[\widetilde{\mathcal{K}}^\bullet\big]\big)$ of any adherence point is determined by Proposition  \ref{different:limite:beta:d}.
\begin{proof}
The tightness of \eqref{eq:conv:theo:sous:suite} is now a consequence of Propositions \ref{different:limite:beta:d} and \ref{Proposition:inclusion:hull},  the fact that $\Gamma^{(\varepsilon)}(0)=\Lambda_\infty(0)$, the control of the modulus of continuity provided by Proposition \ref{prop:Tightness}, and the characterization of relatively compact sets for $\mathbb{M}_\infty^{\mathrm{GHPU}}$ from \cite[Lemma 2.12]{GM2017}. We write $P:\mathbb{M}_\infty^{\mathrm{GHPU}}\to \mathbb{M}_\infty^{\mathrm{GHPU}}$ for the map consisting in forgetting the second curve. In particular, notice that $P$ is a continuous map. Next, fix $(\varepsilon_n)_{n\geq 0}$  a sequence converging to $0$ under which \eqref{eq:conv:theo:sous:suite} converges in distribution. Next remark that the vector:
\begin{align}\label{eq:inter:theo:Gamma}
     \Bigg( \Big[\mathfrak{H},&D_\infty, \Lambda_\infty(0), \mathrm{Vol}_\infty, \Lambda_\infty\Big],  \Big[\mathfrak{H},D_\infty,\Lambda_\infty(0), \mathrm{Vol}_\infty, \Lambda_\infty, \big(\Gamma^{(\varepsilon)}(\varepsilon^{-\vartheta^\prime} t)\big)_{t\geq 0}\Big] ,    \\
     &\Big(H_{\lfloor \varepsilon^{-\vartheta^{\prime}}t\rfloor}^{(\ell)}(\varepsilon), H_{\lfloor \varepsilon^{-\vartheta^{\prime}}t\rfloor}^{(\mathrm{r})}(\varepsilon), V_{\lfloor \varepsilon^{-\vartheta^{\prime}}t\rfloor}(\varepsilon)\Big)_{t\geq 0}, \big[\mathcal{K}^\bullet(\varepsilon)\big]\Bigg)
\end{align}
also converges in distribution along $(\varepsilon_n)_{n\geq 1}$, since 
$$P\Big(\Big[\mathfrak{H},D_\infty, \Lambda_\infty(0), \mathrm{Vol}_\infty, \Lambda_\infty, \big(\Gamma^{(\varepsilon)}(\varepsilon^{-\vartheta^\prime} t)\big)_{t\geq 0}\Big] \Big)=\Big[\mathfrak{H},D_\infty, \Lambda_\infty(0),  \mathrm{Vol}_\infty, \Lambda_\infty\Big],$$
for every $\varepsilon\in(0,1)$. By disintegration \cite[Theorem~6.10]{Kallenberg}, applied to the joint limiting law and up to extend the probability space,  we can consider random variables $ \Big[\widetilde{\mathfrak{H}},\widetilde{D}_\infty, \widetilde{\Lambda}_\infty(0), \widetilde{V}_\infty, \widetilde{\Lambda}_\infty, \widetilde{\Gamma} \Big]$, $\big(\widetilde{\cS}_t^{(\ell)},\ \widetilde{\cS}_t^{(\mathrm{r})}, \widetilde{\cV}_t\big)_{t\geq 0}, \big[\widetilde{\mathcal{K}}^\bullet\big]$ such that \eqref{eq:inter:theo:Gamma} converges in distribution towards:
$$ \left( \Big[\mathfrak{H},D_\infty, \Lambda_\infty(0), \mathrm{Vol}_\infty, \Lambda_\infty\Big], \Big[\widetilde{\mathfrak{H}},\widetilde{D}_\infty, \widetilde{\Lambda}_\infty(0) , \widetilde{V}_\infty, \widetilde{\Lambda}_\infty, \widetilde{\Gamma} \Big] ,\big(\widetilde{\cS}_t^{(\ell)},\ \widetilde{\cS}_t^{(\mathrm{r})}, \widetilde{\cV}_t\big)_{t\geq 0}, \big[\widetilde{\mathcal{K}}^\bullet\big]\right),$$
along $(\varepsilon_n)_{n\geq 1}$. Moreover,  the continuity of $P$ yields that: $$P\Big(\Big[\widetilde{\mathfrak{H}},\widetilde{D}_\infty, \widetilde{\Lambda}_\infty(0), \widetilde{V}_\infty, \widetilde{\Lambda}_\infty, \widetilde{\Gamma} \Big]\Big)= \Big[\mathfrak{H},D_\infty, \Lambda_\infty(0), \mathrm{Vol}_\infty, \Lambda_\infty\Big].$$ Let us now consider a representative of $\Big[\widetilde{\mathfrak{H}},\widetilde{D}_\infty, \widetilde{\Lambda}_\infty(0), \widetilde{V}_\infty, \widetilde{\Lambda}_\infty, \widetilde{\Gamma} \Big]$, that with a slight abuse of notation we denote by $\big(\widetilde{\mathfrak{H}},\widetilde{D}_\infty, \widetilde{\Lambda}_\infty(0), \widetilde{V}_\infty, \widetilde{\Lambda}_\infty, \widetilde{\Gamma}\big)$. Then there exists a isometry $\Phi:\widetilde{\mathfrak{H}}\to \mathfrak{H}$ such that $\Lambda_\infty=\Phi\circ \widetilde{\Lambda}_\infty$ and $\mathrm{Vol}_\infty=\Phi_{\#}\widetilde{V}_\infty$. We set $\Gamma=\Phi\circ \widetilde{\Gamma}$. By construction, we have:
$$ \Big[\widetilde{\mathfrak{H}},\widetilde{D}_\infty, \widetilde{\Lambda}_\infty(0), \widetilde{V}_\infty, \widetilde{\Lambda}_\infty, \widetilde{\Gamma} \Big]=\Big[\mathfrak{H}, D_\infty, \Lambda_\infty(0), \mathrm{Vol}_\infty, \Lambda_\infty, \Gamma \Big]$$
and thus $\Big[\mathfrak{H}, D_\infty, \Lambda_\infty(0), \mathrm{Vol}_\infty, \Lambda_\infty, \Gamma \Big]$ is measurable and \eqref{eq:inter:theo:Gamma} converges in distribution towards:
$$ \left( \Big[\mathfrak{H},D_\infty, \Lambda_\infty(0), \mathrm{Vol}_\infty, \Lambda_\infty\Big], \Big[\mathfrak{H}, D_\infty, \Lambda_\infty(0), \mathrm{Vol}_\infty, \Lambda_\infty, \Gamma \Big] ,\big(\widetilde{\cS}_t^{(\ell)},\ \widetilde{\cS}_t^{(\mathrm{r})}, \widetilde{\cV}_t\big)_{t\geq 0}, \big[\widetilde{\mathcal{K}}^\bullet\big]\right),$$
along the subsequence $(\varepsilon_n)_{n\geq 1}$.
\end{proof}
For the remainder of the section we fix $(\varepsilon_n)_{n\geq 1}$ along which \eqref{eq:conv:theo:sous:suite} converges in distribution, and $\Gamma$ as in the statement of Theorem \ref{theo:exi:Gamma}. We conclude with some direct consequences on $\Gamma$. We first consider the drift-dominated regime $\vartheta\in[0,1)$ and $d\neq 0$.
 
\begin{corollary}\label{coro:limite:bord}
Assume that $\vartheta\in[0,1)$ and $d\neq 0$. Then, $\P$-a.s., we have
$$\Gamma(t)=\Lambda_\infty\big(\frac{2d}{3}t\big),\quad t\geq 0. $$
\end{corollary}
Let us mention, that using the notation of Theorem \ref{theo:exi:Gamma},   Proposition \ref{different:limite:beta:d} and Slutsky lemma entail that $\big(\widetilde\cS_t^{(\ell)},\widetilde\cS_t^{(\mathrm{r})},\widetilde \cV_t\big)_{t\geq 0}=(2d/3t,-2d/3t,0)_{t\geq0}$. Furthermore, it  should be possible to show that $\big[\widetilde{\mathcal{K}}^\bullet\big]$ is a copy of $\big[\mathcal{K}^\bullet\big]$ independent of 
$$\Big[\mathfrak{H}, D_\infty, \Lambda_\infty(0), \mathrm{Vol}_\infty, \Lambda_\infty, (\Lambda_\infty\big(\frac{2d}{3}t\big))_{t\geq 0} \Big].$$
However we will not use this result, and to avoid unnecessary technicalities we do not prove it.
\begin{proof}
By continuity, it suffices to show that  $D_\infty(\Lambda_\infty(2d t/3), \Gamma^{(\varepsilon)}(\varepsilon^{-\widetilde{\vartheta}}t))$ converges in probability towards $0$ as $\varepsilon\to 0$, for any fixed $t>0$.  Without loss of generality, assume that $d>0$ (the case $d<0$ follows by the same argument). Fix $t>0$ and consider $\varepsilon\in (0,1)$. We set $n(\varepsilon):=\lfloor \varepsilon^{-\widetilde{\vartheta}} t\rfloor +1$ and $r(\varepsilon):=\sup \{r\geq 0:~\Lambda_\infty(r)\in \mathcal{D}_{n(\varepsilon)}^\bullet(\varepsilon)\}$. Then, we have:
\begin{align*}
D_\infty\Big(\Lambda_\infty\big(\frac{2d}{3} t\big), \Gamma^{(\varepsilon)}\big(\varepsilon^{-\widetilde{\vartheta}}t\big)\Big)&\leq  D_\infty\Big(\Lambda_\infty\big(\frac{2d}{3} t\big),\Lambda_\infty\big(r(\varepsilon)\big)\Big)+ D_\infty\Big(\Lambda_\infty\big(r(\varepsilon)\big),\Gamma^{(\varepsilon)}\big(n(\varepsilon)\big)\Big)\\
&+D_\infty\Big( \Gamma^{(\varepsilon)}\big(n(\varepsilon)\big),\Gamma^{(\varepsilon)}\big(\varepsilon^{-\widetilde{\vartheta}}t\big)\Big).
\end{align*}
 By construction, we already have $ D_\infty\big( \Gamma^{(\varepsilon)}(n(\varepsilon)),\Gamma^{(\varepsilon)}(\varepsilon^{-\widetilde{\vartheta}}t)\big)\leq \varepsilon$ and by Propositions \ref{prop:formule:bord} and \ref{different:limite:beta:d}, it holds that $r(\varepsilon)\to\frac{2d}{3} t$ in probability as $\varepsilon\to 0$. To conclude, it suffices to show that 
$$ D_\infty\big(\Lambda_\infty(r(\varepsilon)),\Gamma^{(\varepsilon)}(n(\varepsilon))\big)$$
converges in probability to $0$ when $\varepsilon\to 0$. In this direction, remark that by the spatial Markov property and Lemma \ref{control:distance:bord}, it suffices to show that $M_{n(\varepsilon)}^{(\mathrm{r})}\to 0$, in probability, as $\varepsilon\to 0$ and this follows straightforwardly by Proposition \ref{different:limite:beta:d}.
\end{proof}
Let us now focus on the other regimes. Similarly we have:

\begin{corollary}
Assume $\vartheta\geq 1$. Then, $\P$-a.s., it holds that:
    \begin{itemize}
        \item[$\mathrm{(i)}$] the curve $\Gamma$ is $(1/3^{-})$-Hölder, i.e. for every $t,\kappa>0$ we have
        $$ \sup\Big\{|s_1-s_2|^{-\frac{1}{3}+\kappa}\cdot D_\infty(\Gamma(s_1), \Gamma(s_2)):~s_1,s_2\in [0,t] \text{ with }s_1\neq s_2\Big\}<\infty; $$
        \item[$\mathrm{(ii)}$] the image of $\Gamma$ has no $\mathrm{Vol}_\infty$-mass, i.e.
        $$\mathrm{Vol}_\infty\big(\Gamma(\mathbb{R}_+)\big)=0.$$
        \end{itemize}
\end{corollary}
\begin{proof}
We start with the proof of (i). In this direction, we fix $t >0$ and $\kappa\in(0,1/3)$, and for every $s\in(0,t)$ and $A>1$, we introduce the subset of $\mathbb{M}_\infty^{\mathrm{GHPU}}$:
\begin{align*}
&\mathbb{A}_{s}(A):=\Big\{\big[(X,d,\rho,\mu,\gamma_1,\gamma_2)\big]:~ \frac{d(\gamma_2(r),\gamma_2(r^\prime))}{|r-r^\prime|^{\frac{1}{3}-\kappa}}\leq A,~\text{for all}~r,r^\prime\in[0,t]~\text{ with }~|r^\prime-r|\geq s\Big\}.    
\end{align*}
It is straightforward that $\mathbb{A}_s(A)$ is a closed set and in particular measurable. Moreover,  Portemanteau theorem yields:
$$\mathbb{P}\Big(\Big[\mathfrak{H},D_\infty, \mathrm{Vol}_\infty, \Lambda_\infty, \Gamma \Big]\notin \mathbb{A}_{s}(A)\Big)\leq \limsup_{\varepsilon\to 0} \mathbb{P}\Big(\Big[\mathfrak{H},D_\infty, \mathrm{Vol}_\infty, \Lambda_\infty, (\Gamma^{(\varepsilon)}(\varepsilon^{-3}\cdot t))_{t\geq 0} \Big]\notin \mathbb{A}_{s}(A)\Big).$$
Hence, it holds:
$$\mathbb{P}\Big(\Big[\mathfrak{H},D_\infty, \mathrm{Vol}_\infty, \Lambda_\infty, \Gamma \Big]\notin \mathbb{A}_{s}(A)\Big)\leq \limsup_{\varepsilon\to 0}\mathbb{P}\Big(\sup_{\delta\in[\varepsilon^2,1]}\delta^{-\frac{1}{2}+\kappa}\cdot\varpi^{(\varepsilon)}\big(\delta, t\big)\geq A\Big).$$
By dominated convergence, letting $s\downarrow 0$, we obtain:
$$\mathbb{P}\Big(\Big[\mathfrak{H},D_\infty, \mathrm{Vol}_\infty, \Lambda_\infty, \Gamma \Big]\notin \bigcap_{s>0} \mathbb{A}_{s}(A)\Big)\leq \limsup_{\varepsilon\to 0}\mathbb{P}\Big(\sup_{\delta\in[\varepsilon^2,1]}\delta^{-\frac{1}{2}+\kappa}\cdot\varpi^{(\varepsilon)}\big(\delta, t\big)\geq A\Big).$$
and then an application of Proposition \ref{prop:Tightness} yields
$$\limsup \limits_{A\to \infty} \mathbb{P}\Big(\Big[\mathfrak{H},D_\infty, \mathrm{Vol}_\infty, \Lambda_\infty, \Gamma \Big]\notin\bigcap_{s>0} \mathbb{A}_{s}(A)\Big)=0.$$
This completes the proof of (i), taking $\kappa,t$ rationals. The second point (ii) now follows directly from (i) and since, according to \cite{LG2022:Haus} the measure $\mathrm{Vol}_\infty$ is a constant times the Hausdorff measure with gauge function $x\mapsto  x^4\log(\log(1/x))$.
\end{proof}


\begin{thebibliography}{99}
\bibitem{ALG}
{C. Abraham, J.-F. Le Gall}, Excursion theory for Brownian motion
indexed by the Brownian tree. {\it J.~Eur. Math. Soc. (JEMS)} 20, 2951--3016, (2018)

\bibitem{Angel2003}
{O. Angel}, Growth and percolation on the uniform
infinite planar triangulation.
{\it Geom. Funct. Anal.} 13, 935--974, (2003)

\bibitem{BaurMiermontRay}
{E. Baur, G. Miermont, G. Ray}, Classification of scaling limits of uniform quadrangulations with a boundary.
{\it Ann. Probab.} 47, 3397--3477, (2019)

\bibitem{BertoinCurienKortchemski}
{J. Bertoin, N. Curien, I. Kortchemski}, Random planar maps and growth-fragmentations.
{\it Ann. Probab.} 46, 207--260, (2018)

\bibitem{Bet:2015}
{J. Bettinelli}, Scaling limit of random planar quadrangulations with a boundary.
{\it Ann. Inst. Henri Poincaré Probab. Stat.} 51 (2), 432--477, (2015)

\bibitem{BM2017}
{J. Bettinelli, G. Miermont},
Compact Brownian surfaces I: Brownian disks.
{\it Probab. Theory Related Fields} 167, 555--614, (2017)


\bibitem{CC20218}
{A. Caraceni, N. Curien}, Geometry of the uniform infinite half-planar quadrangulation.
{\it Random Structures \& Algorithms} 52, 454--494, (2018)

\bibitem{C2015}
{N. Curien}, A Glimpse of the Conformal Structure of Random Planar Maps.
{\it Comm. Math. Phys.} 333, 1417--1463, (2015)

\bibitem{CurienPeeling}
{N. Curien}, {\em Peeling Random Planar Maps: École d'Été de Probabilités de Saint-Flour XLIX--2019}.
Lecture Notes in Mathematics 2335, Springer, (2023)


\bibitem{CurienLeGallHull}
{N. Curien, J.-F. Le Gall}, The hull process of the Brownian plane.
{\it Probab. Theory Related Fields} 166, 187--231, (2016)

\bibitem{CLG2017}
{N. Curien, J.-F. Le Gall}, Scaling limits for the peeling process on random maps.
{\it Ann. Inst. Henri Poincaré Probab. Stat.} 53 (1), 322--357, (2017)



\bibitem{DuplantierMillerSheffield}
{B. Duplantier, J. Miller, S. Sheffield}, Liouville quantum gravity as a mating of trees.
{\it Ast\'erisque} 427, viii+257, (2021)

\bibitem{DLG2002}
{T. Duquesne, J.-F. Le Gall}, Random trees, Lévy processes and spatial branching processes.
{\it Astérisque} 281, vi + 147 pp., Société Mathématique de France, (2002)

\bibitem{GarbanPeteSchramm}
{C. Garban, G. Pete, O. Schramm}, The scaling limits of near-critical and dynamical percolation.
{\it J. Eur. Math. Soc.} 20, 1195--1268, (2018)

\bibitem{Getoor1979}
{R. K. Getoor}, The Brownian escape process.
{\it Ann. Probab.} 7 (5), 864--867, (1979)

\bibitem{GY2003}
{A. Göing-Jaeschke, M. Yor}, A survey and some generalizations of Bessel processes.
{\it Bernoulli} 9 (2), 313--349, (2003)

\bibitem{GM2017}
{E. Gwynne, J. Miller}, Scaling limit of the uniform infinite half-plane quadrangulation
in the Gromov--Hausdorff--Prokhorov--uniform topology.
{\it Electron. J. Probab.} 22, article~84, 1--47, (2017)

\bibitem{GwynneMillerCharacterization}
{E. Gwynne, J. Miller}, Characterizations of SLE$_\kappa$ for $\kappa\in(4,8)$ on Liouville quantum gravity.
{\it Ast\'erisque} 429, 129--242, (2021)

\bibitem{GM2021}
{E. Gwynne, J. Miller}, Convergence of percolation on uniform quadrangulations with boundary
to SLE$_6$ on $\sqrt{8/3}$-Liouville quantum gravity.
{\it Astérisque} 429, 1--127, Société Mathématique de France, (2021)

\bibitem{Janson}
{S. Janson}, Tail bounds for sums of geometric and exponential variables.
{\it Statist. Probab. Lett.} 135, 1--6, (2018)

\bibitem{Kallenberg}
{O. Kallenberg}, {\em Foundations of Modern Probability}.
Second edition, Probability and Its Applications, Springer, New York, (2002)

\bibitem{K2023}
{A. Khezeli}, A unified framework for generalizing the Gromov--Hausdorff metric.
{\it Probab. Surv.} 20, 837--896, (2023)


\bibitem{Zurich}
{J.-F. Le Gall}, {\em Spatial Branching Processes, Random Snakes and Partial Differential Equations}.
Lectures in Mathematics ETH Zürich, Birkhäuser, Basel, (1999)


\bibitem{LeGallRandomGeometry}
{J.-F. Le Gall}, Random geometry on the sphere.
{\it Proceedings of the International Congress of
Mathematicians---Seoul 2014}, Vol.~I, 421--442, (2014)




\bibitem{LG2018}
{J.-F. Le Gall}, Subordination of trees and the Brownian map.
{\it Probab. Theory Related Fields} 171, 819--864, (2018)

\bibitem{LG2019}
{J.-F. Le Gall}, Brownian disks and the Brownian snake.
{\it Ann. Inst. Henri Poincaré Probab. Stat.} 55 (1), 237--313, (2019)

\bibitem{LG2022}
{J.-F. Le Gall}, The Brownian disk viewed from a boundary point.
{\it Ann. Inst. Henri Poincaré Probab. Stat.} 58 (2), 1091--1119, (2022)

\bibitem{LG2022:Haus}
{J.-F. Le Gall}, The volume measure of the Brownian sphere is a Hausdorff measure.
{\it Electron. J. Probab.} 27, article~113, 1--28, (2022)



\bibitem{LGR3}
{J.-F. Le Gall, A. Riera}, Spine representations for non-compact models of random geometry.
{\it Probab. Theory Related Fields} 181, 571--645, (2021)

\bibitem{LGR4}
{J.-F. Le Gall, A. Riera}, Spatial Markov property in Brownian disks.
{\it Ann. Inst. Henri Poincaré Probab. Stat.} 61 (3), 1523--1565, (2025)


\bibitem{LGR1}
{J.-F. Le Gall, A. Riera}, Growth-fragmentation processes in Brownian motion indexed by the Brownian tree.
{\it Ann. Probab.} 48, 1742--1784, (2020)


\bibitem{LGR6}
{J.-F. Le Gall, A. Riera}, Peeling the Brownian half-plane.
{\it Probab. Theory Related Fields} 195, 369--411, (2026)

\bibitem{AlexandreLeGrand}
{A. Legrand}, Some FKG inequalities for stochastic processes.
 arXiv:2407.13871, (2024).
 
\bibitem{MakarovSmirnov}
{N. Makarov, S. Smirnov}, Off-critical lattice models and massive SLEs.
In {\it XVIth International Congress on Mathematical Physics}, World Scientific, (2010)

\bibitem{MatsumotoYor2005}
{H. Matsumoto, M. Yor},
Exponential functionals of Brownian motion,
I: Probability laws at fixed time.
{\it Probab. Surv.} 2, 312--347, (2005)


\bibitem{Mie13}
{G. Miermont}, The Brownian map is the scaling limit of uniform random
plane quadrangulations. {\it Acta Math.} 210, 319--401, (2013).

\bibitem{MillerSheffieldI}
{J. Miller, S. Sheffield}, Liouville quantum gravity and the Brownian map I: the QLE$(8/3,0)$ metric.
{\it Invent. Math.} 219, 75--152, (2020)

\bibitem{MillerSheffieldII}
{J. Miller, S. Sheffield}, Liouville quantum gravity and the Brownian map II: geodesics and continuity of the embedding.
{\it Ann. Probab.} 49, 2732--2829, (2021)

\bibitem{MillerSheffieldIII}
{J. Miller, S. Sheffield}, Liouville quantum gravity and the Brownian map III: the conformal structure is determined.
{\it Probab. Theory Related Fields} 179, 1183--1211, (2021)



\bibitem{MourichouxBigeodesic}
{M. Mourichoux}, The bigeodesic Brownian plane.
{\it Ann. Probab.}, to appear.

\bibitem{NolinWerner}
{P. Nolin, W. Werner}, Asymmetry of near-critical percolation interfaces.
{\it J. Amer. Math. Soc.} 22, 797--819, (2009)


 \bibitem{RR2026}
{ F.~Rebei, A.~Riera}, {Dual metric explorations of the Brownian half-plane,} {\it In preparation.} (2026)

\bibitem{RevuzYor1999}
{D. Revuz, M. Yor}, {\em Continuous Martingales and Brownian Motion}.
Third edition, Grundlehren der mathematischen Wissenschaften 293, Springer, (1999)

\bibitem{These:R}
{A.~Riera} {\it Brownian Geometry.} Theses, Université Paris-Saclay, (2021)

\bibitem{2024structure}
{A.~Riera, A.~Rosales-Ortiz}, {The structure of the local time of
  {M}arkov processes indexed by {L}{\'e}vy trees. }{\it Probab. Theory Related Fields.} {189}, {1--99}, (2024)

  \bibitem{RRO2026}
{ A.~Riera, A.~Rosales-Ortiz}, {{L}{\'e}vy snakes, self-similar Markov trees and applications to the reconstruction of random geometries.} {\it In preparation.} (2026)



\bibitem{Watabiki1995}
{Y. Watabiki}, Construction of non-critical string field
theory by transfer matrix formalism in dynamical
triangulation.
{\it Nuclear Phys. B} 441, 119--163, (1995)


\end{thebibliography}
\end{document}